\documentclass[preprint,11pt]{article}
\usepackage{amssymb}
\usepackage{mathrsfs}
\usepackage{amsfonts}
\usepackage{graphicx}
\usepackage{colortbl,dcolumn}
\usepackage{tabularx}
\usepackage{amsmath}
\usepackage{psfrag}
\usepackage{booktabs}
\usepackage{array}
\usepackage{cite}
\numberwithin{equation}{section}
\usepackage[top=1.2in, bottom=1.2in, left=0.9in, right=0.9in]{geometry}

\newcommand{\R}{\mathbb{R}}
\newcommand{\N}{\mathbb{N}}

\newcommand{\dd}{\text{d}}
\newtheorem{theorem}{Theorem}[section]

\newtheorem{assumption}[theorem]{Assumption}
\newtheorem{proposition}[theorem]{Proposition}
\newtheorem{lemma}[theorem]{Lemma}

\begin{document}

%
 \title{Long-time error analysis of finite element fully discrete schemes for SPDEs with non-globally Lipschitz coefficients
 \footnote{R.Q. was supported  by the Reaserch Fund for Yancheng Teachers University under 204040025.   X.W. were supported by
 Natural Science Foundation of China (12471394, 12371417, 12071488) and Hunan Basic Science Research Center for Mathematical Analysis (2024JC2002).
%
}
}

\author{
Ruisheng Qi$\,^\text{a}$,
\quad Xiaojie Wang$\,^\text{b}$,
\\
\footnotesize $\,^\text{a}$ School of Mathematics and Statistics, Yancheng Teachers University, Yancheng, China\\
\footnotesize qiruisheng123@126.com\\
\footnotesize $\,^\text{b}$ School of Mathematics and Statistics, HNP-LAMA, Central South University, Changsha, China\\
\footnotesize x.j.wang7@csu.edu.cn\; and \;x.j.wang7@gmail.com\\
}
\maketitle
\begin{abstract}\hspace*{\fill}\\
  \normalsize


%
%
The present paper proposes new fully discrete schemes for long-time approximations of stochastic partial differential equations (SPDEs) with non-globally Lipschitz coefficients in a bounded domain $D \subset \R^d, d =1,2,3 $.
A novel family of linearly implicit time-stepping schemes is introduced, based on a standard Galerkin finite element spatial semi-discretization.
A distinguishing feature of the schemes is that the proposed finite element fully discrete approximations preserve uniform-in-time moment bounds in a Banach space $L^{r}(D), r >2$, without requiring any restriction on the time-space discretization stepsize ratio.
To show it, some non-standard arguments are developed. First, we derive long-time error estimates in the Banach space $L^r(D)$ for finite element fully discrete approximations of  the deterministic linear parabolic equation with non-smooth initial value, which is, to our knowledge, new for the literature on numerical PDEs and of independent interest.
These error estimates together with the contractive property of the semi-group in $L^{r}(D), r > 2$, the dissipativity of the nonlinearity and the particular benefit of the taming strategy help us establish the desired uniform-in-time moment bounds.
Then both strong and weak error bounds of the proposed schemes are carefully analyzed in a setting of low regularity, with uniform-in-time convergence rates obtained for cases of both space-time white and trace-class noises.
The analysis is highly nontrivial, due to the finite element discretization, the low regularity and the presence of the super-linearly growing nonlinearity.
Finally, numerical results are presented to verify the previous theoretical findings.

\end{abstract}

\section{Introduction}
As an important class of stochastic models,  stochastic partial differential equations (SPDEs)  find many applications in the area of science and engineering, such as finance, geosciences, statistical mechanics, meteorology
and biology. Since their  analytical solutions can be rarely available, one usually relies on numerical solutions to investigate the underlying models.
%
In the present paper,  we are interested in fully discrete finite element approximations of  the following parabolic SPDEs
\begin{align}\label{eq:parabolic-SPDE}
\,\dd X(t)
=
-AX(t)+F(X(t))\,\dd t
+
\,\dd W(t),\; X(0)=X_0,
\end{align}
in a real separable Hilbert space $H:=L^2(D)$ with inner product $\left<\cdot,\cdot\right>$ and the induced  norm
$\|\cdot\|=\left<\cdot,\cdot\right>^{\frac12}$. Here $D \subset \R^d, d =1,2,3 $ is a bounded domain and $-A$
is assumed to be the Laplace operator with Dirichlet boundary condition on the domain $D$.
Moreover, the stochastic process $\{W(t)\}_{t\geq0}$ is an $H$-valued (possibly cylindrical) $Q$-Wiener process and $F$ is a nonlinear Nemytskij operator associated with a real-valued function $f:\mathbb{R}\rightarrow\mathbb{R}$, i.e.,
\[
F(u)(x):=f(u(x)), \: x\in D.
\]

In the last decades, a large amount of  work is devoted to strong and weak approximations of SPDEs   \eqref{eq:parabolic-SPDE} over finite-time horizon (see, e.g., \cite{2012Weak,2018A, 2024dA,2023Strong,2021caiWeak,Charles2020Weak,chen2025Strong,feng2017finite,
Jentzen2015Strong,Anant2018Optimal,QQQi2018Optimal,Wang2015Weak,wang2020efficient,Zhang2025Weak}, to just mention a few).
Recently, there has been a growing interest in long-time approximations of SPDEs
(e.g.,\cite{Charles2014Approximation,brehier2022approximation,Br2025Analysis,jiang2025uniform,Cui2021Weak,Wang2024Approximation,Chen2020A,liu2025Geometric,qi2026long}).
For example, in a globally Lipschitz setting for the drift term $F$, the authors of \cite{Charles2014Approximation} and \cite{Chen2020A} used the classical linear implicit Euler and exponential Euler scheme, respectively, to approximate the invariant measure of the parabolic SPDE \eqref{eq:parabolic-SPDE}, with uniform-in-time weak convergence rates obtained.
Higher-order approximations of invariant measure of SPDEs with a gradient structure can be found in a very recent preprint \cite{Brehier2025preconditioning}.

For SPDEs with polynomial growing (thus non-globally Lipschitz) $F$, the long-time approximations and their uniform-in-time error analysis turn out to be a challenging problem.
In \cite{Cui2021Weak}, the authors relied on the spectral Galerkin drift-implicit Euler scheme to approximate the invariant measure of the Allen-Cahan type SPDEs \eqref{eq:parabolic-SPDE}, where uniform-in-time weak convergence rates were derived.
The expensive computational costs of the drift-implicit Euler scheme, however, force people to look for a cheaper alternative.
In 2022, Br{\'e}hier \cite{brehier2022approximation} investigated long-time semi-discretization in time of Allen-Cahan type SPDEs \eqref{eq:parabolic-SPDE} via an explicit tamed exponential Euler scheme, where moment bounds for the numerical scheme were obtained with a polynomial dependence with respect to the time horizon.
Later in \cite{Wang2024Approximation}, by not discretizing the stochastic convolution, the authors proposed a spectral Galerkin tamed accelerated exponential Euler method for uniform-in-time weak approximations of one-dimensional ($d=1$) SPDEs \eqref{eq:parabolic-SPDE}, while computing the temporally undiscretized stochastic convolution and the $L^\infty$-norm of the numerical solution for every time step is costly.
%
More recently in \cite{jiang2025uniform}, a novel and easy-to-implement spectral Galerkin explicit fully discretizaton scheme of exponential type was introduced for long-time approximations of SPDEs \eqref{eq:parabolic-SPDE} in multiple dimensions with non-globally Lipschitz coefficients. There, uniform-in-time moment bounds of the numerical approximations and uniform-in-time weak convergence rates were derived.
%

Clearly, the above uniform-in-time convergence rates for approximations of SPDEs with non-globally Lipschitz coefficients are all revealed, based on a spectral Galerkin spatial discretetization that only works for very regular domain such as a rectangle. The analysis of uniform-in-time convergence rate for the finite element method is, to the best of our knowledge, still missing in the literature.

In this paper, we aim to fill this gap and propose a class of novel, linearly implicit fully discrete finite element methods for SPDEs \eqref{eq:parabolic-SPDE} with polynomially growing nonlinearity. Let $V_h\subset H_0^1(D)$ be the continuous piecewise linear finite element space and $P_h$ an $L^2$-projection operator from $L^2(D)$ to $V_h$. Then the fully discrete finite element approximation is to find $V_h$-valued stochastic process $X_{\tau,h}^m$, $m\in \N$ such that
\begin{align}\label{eq:full-discretization}
X_{\tau,h}^m-X_{\tau,h}^{m-1}
+
\tau A_h X_{\tau,h}^m
=
\tau F_{\tau,h}(X_{\tau,h}^{m-1})
+
P_h (W(t_m)-W(t_{m-1})),\; X_{\tau,h}^0=P_hX_0,
\end{align}
where
\begin{align}\label{eq:modification-f-i}
F_{\tau,h}(u)(x):=f_{\tau,h}(u(x)),\;x\in D,
\end{align}
with $f_{\tau,h}$ being a modification of $f$ given by
\begin{align}\label{eq:modification-f-ii}
f_{\tau,h}(v)
:
=\frac{f(v)}{\Big(1+(\beta_1 \tau^\theta+\beta_2 h^\rho)|v|^{\frac{2q-2}\alpha}\Big)^\alpha},
\end{align}
for $\alpha\in (0,1], \theta, \rho, \beta_1, \beta_2>0$.
For simplicity, one can think of $f$ as a polynomial of odd degree $2q-1, q \geq 2$ with a negative leading coefficient. Here the degree $\alpha \in (0, 1]$ and the convergence parameters $\theta, \rho$ are chosen to obey the condition $\max\{\frac12 \alpha\rho, \alpha\theta\}<1+\frac d{2q(2q-1)}-\frac d4$.
We mention that such a modification of $f$ in \eqref{eq:modification-f-ii} is inspired by the spectral Galerkin fully discretization scheme of exponential type proposed by our previous work \cite{jiang2025uniform}.

In order to conduct the long-time error analysis,  a crucial ingredient is to establish uniform-in-time moment bounds of numerical approximations in the Banach space $L^{4q-2}(D)$, $q\geq 2$. Clearly, the nonlinearity $F$ obeys the following dissipativity property
\begin{align}
\left<-F(v),v\right>
\geq
C_1\|v\|^{2q}_{L^{2q}(D)}-C_2,\; \forall v\in L^{2q}(D),
\end{align}
and $P_hF$ still preserves the dissipativity property
\begin{align}\label{eq:dissipativity-property-Phf-in-L2-norm}
\left<-P_hF(v_h),v_h\right>
\geq
C_1\|v_h\|^{2q}_{L^{2q}(D)}-C_2,\; \forall v_h\in V_h,
\end{align}
for some constants $C_1,C_2>0$.
This suffices to prove the uniform-in-time moment bound for the numerical approximations in $L^{2}(D)$ (Lemma \ref{lem:uniform-moment-L2}).
By noting that the nonlinearity $F$ also enjoys the dissipativity property in $\dot{H}^1$:
\begin{align}
\left<-\nabla F(v),\nabla v\right>
\geq
-C_3\|\nabla v\|^2, \;\forall v\in H_0^1(D), \quad C_3 >0,
\end{align}
a natural idea is to establish the uniform bound in $\dot{H}^1$, which combined with the Sobolev embedding would promise the desired bound in the Banach space $L^{4q-2}(D)$, $q\geq 2$.
Unfortunately, such an idea might work in the spectral Galerkin spatial discretization but does not work in the finite element setting. Indeed, the  presence of the finite element projection $P_h$ in front of the drift term $F$ destroys the  dissipativity property of $F$ in $\dot{H}^1$:
\begin{align}
\big \langle
-A_h^{\frac12}P_hF(v_h),A_h^{\frac12}v_h
\big \rangle
=
\big \langle
-A^{\frac12}P_hF(v_h),A^{\frac12}v_h
\big \rangle
\ngeq
-C_3\|\nabla v_h\|^2, \;\forall v_h\in V_h.
\end{align}
%
%
We mention  that the finite element approximation of stochastic Cahn-Hilliard equation faces a similar difficulty (see, e.g., the introduction of \cite{Daisuke2018strong}).
Additional difficulties come from the possibly low regularity and multiple dimension setting.
Indeed, the considered SPDEs might evolve in $\dot {H}^\gamma$ for $\gamma <\frac12$.
To overcome these difficulties, our methodology consists of three steps.
As the first step, we use \eqref{eq:dissipativity-property-Phf-in-L2-norm} to easily get the uniform-in-time moment bounds for the numerical solution $X_{\tau,h}^m$ in $L^2(D)$-norm and the discretized version of the stochastic convolution $W_A(t)$ in a Banach space (Lemmas \ref{lem:bound-discrete-stochastic-convulution}, \ref{lem:uniform-moment-L2}). As the second step, we analyze the error estimates in $L^{4q-2}(D)$-norm of the fully discrete finite element method for the corresponding deterministic linear parabolic equation with non-smooth initial value (cf. Theorem \ref{thm:infty-determing-error}).
To the best of our knowledge, such error estimates are even new for the literature on numerical PDEs and of independent interest.
For the last step, by  combining  the corresponding deterministic error estimates  with $L^2(D)$-moment bound of $X_{\tau,h}^m$, uniform-in-time moment bounds for the discretized version of the stochastic  convolution $W_A(t)$, a contractive property of the semi-group in $L^{4q-2}(D)$, the dissipativity of the nonlinearity and the particular benefit of the taming strategy, we are able to  obtain the desired uniform-in-time moment bounds of the proposed fully discrete finite element methods (Theorem \ref{th:uniform-moment-bound}). It is worthwhile to mention that we do not put any restriction on the time-space discretization stepsize ratio, which is usually required in the finite element setting in the literature, even for finite-time moment bounds.

Armed with the uniform-in-time moment bounds and the new error estimates for the fully discrete finite element method, we carefully analyze both strong and weak convergence of the proposed fully discrete schemes. More accurately, by fixing method parameters $\rho =2, \theta =1$ and $\alpha <1+\frac d{2q(2q-1)}-\frac d4$ we obtain uniform-in-time strong convergence rates (Theorem \ref{them:bound-numerical-solution}):
\begin{align}
\sup_{m\in \mathbb{N}_0}
\|X(t_m)-X_{\tau,h}^m\|_{L^p(\Omega;H)}
\leq
C(h^{\min\{\gamma, 2\}}+\tau^{\min\{\frac\gamma2,1\}})
\end{align}
and uniform-in-time weak convergence rates (Theorem \ref{them:weak-vonvergence}):
\begin{align}
\left |\mathbb{E}\big[\varphi(X(t_m))\big]-\mathbb{E}\big[\varphi(X^m_{\tau,h})\big]\right|
\leq
C(1+t_m^{-\iota}+(t_m-\tau)^{-\frac12})(h^{ 2\iota }+\tau^{ \iota } ),
 \end{align}
 where $ \iota < \min\{\gamma,1\}$ and the parameter $\gamma$ coming from the condition \ref{eq:ass-AQ-condition} characterizes the spatial regularity of the noise process and the solution of SPDE. Different from strong and weak convergence analysis in the literature, new arguments and more careful estimates are developed in the uniform-in-time error analysis. For instance, two auxiliary processes are constructed in the strong error analysis (see \eqref{eq:strong-error-decomposition}) and uniform-in-time Malliavin regularity estimates of the numerical approximations are carefully analyzed (cf. Lemma \ref{lem:eq-mallivin-derivative-numercial-solu}).


The outline of this paper is as follows. In the next section, we introduce the considered SPDEs. In Section 3, we propose the fully discrete finite element method for \eqref{eq:parabolic-SPDE}. Section 4 provides error estimates of the fully discrete finite element approximation of the corresponding deterministic linear parabolic equation with non-smooth initial data. In Section 5, we apply the deterministic error estimates to  establish  the uniform-in-time moment bounds of the numerical solution. Section 6 is devoted to the uniform-in-time strong convergence for the proposed fully discretization. In Section 7, we  study the uniform-in-time weak convergence.  Numerical results are presented in Section 8 to verify the theoretical findings. Finally, a short conclusion is given.


\section{The considered SPDEs with uniform-in-time regularity}
In this  section, we make  some assumptions for   SPDEs \eqref{eq:parabolic-SPDE} and establish the uniform-in-time spatio-temporal regularity of the mild solutions.
%
First, the following assumptions are made,
concerning the linear operator $A$,
the nonlinear term $F$, the noise process $W(t)$ and the initial data $X_0$.

\begin{assumption}\label{assum:linear-operator-A} (Linear operator $A$)
Let $D$ be a bounded convex domain in $\mathbb{R}^d$ for $d\in\{1,2,3\}$ with  smooth boundary.
Let $-A \colon dom(A)\subset H\rightarrow H$
 be  the Laplacian with homogeneous Dirichlet boundary conditions,
 defined by
$-Au = \Delta u$
 with
$
u\in dom(A) :=H^2(D)\cap H_0^1(D)
$.
\end{assumption}

Assumption \ref{assum:linear-operator-A} guarantees that $-A$ generates an analytic and contractive semi-group in $H$ and $L^p(D)$, $p\geq2$, denoted by $E(t):=e^{-tA}, t>0$. Also, this assumption implies the existence of the eigen-system $\{(\lambda_j,e_j)\}_{j \in \mathbb{N}}$ in $H$ such that
$A e_j=\lambda_je_j$ for $j\in \mathbb{N}$ and $\lambda_j \rightarrow \infty \text{ as } j \rightarrow \infty$.
By the spectral theory, we can also define the fractional powers
of $A$  on $H$ in a simple way, e.g.,
$A^\alpha v=\sum_{j=1}^\infty\lambda_j^\alpha\left<v,e_j\right>e_j$, $\alpha\in \mathbb{R}$.
Note that
 $H^\alpha:=D(A^{\frac\alpha2})$
 is a real Hilbert space  with the inner
 product
 $\langle A^{\frac\alpha 2}\cdot, A^{\frac\alpha 2}\cdot \rangle$
 and the associated norm
 $
 \|\cdot\|_{\alpha}
 :=
 \|A^{\frac\alpha 2}\cdot\|
 $.
It is well-known (see e.g. \cite{pazy1983semigroups}) that the following regularity properties of $E(\cdot)$ hold: for any $t>0$, $\mu\geq0$, $\nu\in[0, 1]$
\begin{align}\label{eq:property-E}
\|E(t)\|_{\mathcal{L}(H)}
\leq
e^{-\lambda_1t},
\;
\| A^\mu E(t)\|_{\mathcal{L}(H)}
\leq
 Ct^{-\mu},
\;
\|A^{-\nu}(I-E(t))\|_{ \mathcal{L} (H)}
 \leq
 Ct^{\nu}.
 \end{align}
 Moreover, for any $0\leq t_1<t_2<\infty$ we have
 \begin{align}
\int_{t_1}^{t_2} \|A^{\frac\varrho2} E(s) v\|^2\,\dd s
&
\leq
C|t_2-t_1|^{1-\varrho}\|v\|^2,\;\forall v\in H, \varrho\in[0,1],
\label{III-spatio-temporal-S(t)}
\\
 \Big\|A^{\varrho}\int_{t_1}^{t_2}
         E(t_2-\sigma)v\,\dd \sigma\Big\|
 &
 \leq
 C|t_2-t_1|^{1-\varrho}\|v\|,\;\forall v\in H,\; \varrho\in[0,1].
 \label{IV-spatio-temporal-S(t)}
 \end{align}
Also, for any positive integer $l$, the semi-group $E(\cdot)$ satisfies the following contractive property:
\begin{align}
\|E(t)v\|_{L^{2l}(D)}
\leq
\|v\|_{L^{2l}(D)},\, \forall \,v\in L^{2l}(D), \,t\geq 0.
\end{align}

\begin{assumption}\label{assum:nonlinearity}(Nonlinearity) Let $q\in[2,\infty)$ be an integer for $d=1,2$ and $q=2$ for $d=3$, and let $F:L^{4q-2}(D)\rightarrow H$ be a deterministic mapping given by
\begin{align}
F(v)(x)=f(v(x)),
\quad
x\in D,
\end{align}
where $f(x)=\sum_{j=0}^{2q-1} a_{j}x^j,x\in \mathbb{R}$ with $a_{2q-1}<0$ and $a_j\in \mathbb{R}$, $j\in\{0, 1,2,\cdots,2q-2\}$.
\end{assumption}

Based on the above assumption, one can easily deduce the following properties of $f$.
\begin{lemma}
Let Assumption \ref{assum:nonlinearity} be fulfilled.
 Then, there exist constants $L_f\in \mathbb{R}$ and $R_f, c_0, c_1$, $c_2, c_3, c_4, c_5>0$ such that, for all $x,y\in \mathbb{R}$
\begin{align}
f'(x)
\leq
&
 L_f,
 \label{eq:definiton-I-f-I}
 \\
 |f'(x)|\vee|f''(x)|
 \leq
 &
  R_f(1+|x|^{2q-2}),
  \label{eq:definiton-II-f}
  \\
 (x+y)f(x)
 \leq
 &
 -c_0|x|^{2q}
 +
 c_1|y|^{2q}+c_2,
  \label{eq:definiton-III-f}
 \\
 |f(x)-f(y)|
 \leq
 &
 c_3(1+|x^{2q-2}+|y|^{2q-2})|x-y|,
  \label{eq:definiton-IIII-f}
  \\
  |f(x)|
\leq
&
c_4|x|^{2q-1}+c_5.
\label{eq:assumption-f}
\end{align}
\end{lemma}

Before coming to the noise process, we introduce additional notations and spaces.
Let $\mathcal{L}(H)$ denote the Banach space of all bounded linear operators on $H$, equipped with the usual operator norm. Also, let $\mathcal{L}_2(H)$
be the Hilbert space consisting of  all Hilbert-Schmidt operators from
$H$ to $H$, equipped with the inner product and the norm,
\begin{align}
\big<\Gamma_1,\Gamma_2\big>_{\mathcal{L}_2(H)}
=
\sum_{j=1}^\infty \big<\Gamma_1 \phi_j,\Gamma_2\phi_j\big>,
\qquad
\|\Gamma\|^2_{\mathcal{L}_2(H)}
=
\sum_{j=1}^\infty \big\|\Gamma \phi_j\|^2,
\end{align}
where $\{\phi_j\}_{j\in\N}$ is an arbitrary orthonormal basis  of $H$.
If $\Gamma\in\mathcal{L}_2(H)$ and $L\in \mathcal{L}(H)$, then $\Gamma L, L\Gamma\in \mathcal{L}_2(H)$ and
\begin{align}
\|\Gamma L\|_{\mathcal{L}_2(H)}
\leq
\|L\|_{\mathcal{L}(H)}
\|\Gamma\|_{\mathcal{L}_2(H)}
,
\|L\Gamma \|_{\mathcal{L}_2(H)}
\leq
\|L\|_{\mathcal{L}(H)}
\|\Gamma\|_{\mathcal{L}_2(H)}.
\end{align}

Now we make the following assumptions on the noise process and the initial data.
\begin{assumption}\label{assum:noise-term}(Noise process)
Let  $\{W(t)\}_{t \geq 0}$ be a standard $H$-valued (possibly cylindrical) $Q$-Wiener process with  a covariance operator
 $ \mathcal{L}(H)  \ni Q:H \rightarrow H $ being a  symmetric nonnegative operator
 such that the stochastic convolution defined by
 \begin{align}
 W_A(t):=\int_0^tE(t-s)\,\mathrm{d} W(s),
 \end{align}
 is well-defined in $L^{4q-2}(D)$ and satisfies, for any $p\geq1$
 \begin{align}\label{eq:asuum-L4q-2-W}
 \sup_{t\geq0}\|W_A(t)\|_{L^p(\Omega;L^{4q-2}(D))}<\infty.
 \end{align}
Moreover, we assume
\begin{align}\label{eq:ass-AQ-condition}
\|A^{\frac{\gamma - 1}2}Q^{\frac12}\|_{\mathcal{L}_2(H)}
<\infty,
\quad
\text{ for some } \quad \gamma\in
\big[\tfrac {(q-1)d}{2q-1},2\big].
\end{align}
\end{assumption}

%
\begin{assumption}(Initial data)
\label{assum:intial-value-data}
Let $X_0:\Omega \rightarrow H$ be $\mathcal{F}_0/\mathcal{B}(H)$-measurable and satisfy for any  $p\in \mathbb{N}$,
\begin{align}\label{eq:assum-intial-value-L4q-2}
\|X_0\|_{L^p(\Omega;L^{4q-2}(D))}<\infty.
\end{align}
Moreover, suppose that the following condition holds
\begin{align}\label{eq:assum-intial-value-Hgamma}
\|X_0\|_{L^p(\Omega;V)}
+
\|X_0\|_{L^p(\Omega;H^\gamma)}<\infty,
\end{align}
where $\gamma\in\big[\tfrac {(q-1)d}{2q-1},2\big] $ is the parameter from \eqref{eq:ass-AQ-condition}.
\end{assumption}

Before addressing the well-posedness of the mild solution of \eqref{eq:parabolic-SPDE}, we introduce a useful auxiliary lemma, concerning the Sobolev embedding theorem, which can be found in Triebel (1978, Theorem 2.8.1/Remark 2 and Theorem 4.6.1 in \cite{Triebel1978Interpolation}) and Yagi (2009, Theorem 1.36 in \cite{Yagi2009Abstract}).
\begin{lemma}
Let $D\subset \mathbb{R}^d$, $d=1,2,3$ be a bounded domain with Lipschitz boundary. Let $\beta\in[0,\infty)$ and $r\in(1,\infty)$ such that $\beta r<d$. Then there  exists a positive constant $C=C(\beta,r,d, D)$ such that
\begin{align}\label{lem:interpolation-property}
\|v\|_{L^s(D)}\leq C\|v\|_{W^{\beta,r}(D)}, \;s\in\left[r, \tfrac{r d}{d-\beta r}\right],\;\forall v\in W^{\beta,r}(D).
\end{align}
In addition,  the following embeddings hold
\begin{align}\label{eq:embedding-Hd-C}
H^\delta(D)\subset V:=C(D,\mathbb{R})\quad and\quad L^2(D) \subset W^{\delta,1}(D), \, \; \delta>\tfrac d 2,\; d=1, 2, 3.
\end{align}
\end{lemma}

Under all the above assumptions, the well-posedness of the SPDE \eqref{eq:parabolic-SPDE} has been established,
(see, e.g., \cite[Chapter 6]{cerrai2001second} or \cite{wang2025linearly}).

\begin{theorem}\label{them:property-solution}
Suppose Assumptions \ref{assum:linear-operator-A}-\ref{assum:nonlinearity} and the conditions \eqref{eq:asuum-L4q-2-W} and \eqref{eq:assum-intial-value-L4q-2} hold. Then the SPDE \eqref{eq:parabolic-SPDE} admits
a unique mild solution, given by
\begin{align}\label{eq:definition-mild-solution}
X(t)=E(t)X_0
+
\int_0^tE(t-s)F(X(s))\,\mathrm{d} s
+
 W_A(t),
\end{align}
satisfying
\begin{align}\label{them:bound-L4q-2-mild-solution}
\sup_{t\geq0}\|X(t)\|_{L^p(\Omega;L^{4q-2}(D))}<\infty.
\end{align}
If the conditions \eqref{eq:ass-AQ-condition} and \eqref{eq:assum-intial-value-Hgamma} additionally hold,  the mild solution enjoys further spatial-temporal regularity as follows:
\begin{align}\label{th:spatial-regu-exac-solution}
\sup_{t \geq 0 }\|X(t)\|_{L^p(\Omega;H^\gamma)}
\leq C<\infty,
\end{align}
and, for $0<s<t<\infty$
\begin{align}\label{th:temporap-regu-mild-solution}
\|X(t)-X(s)\|_{L^p(\Omega;H)}
\leq
C(t-s)^{\frac{\min\{\gamma,1\}} 2},
\end{align}
where $C=C(\gamma, Q, X_0)$ is time independent. Additionally, if $\gamma\in(\frac d2,2)$ for $d \in \{1, 2, 3\}$ or $\gamma\in(\tfrac {q-1}{2q-1},\frac12)$ with $Q=I$ for $d =1$, then
\begin{align}\label{them:spatial-bound-in-B-mild}
\sup_{t\in[0,\infty)}\|X(t)\|_{L^p(\Omega;V)}
\leq C<\infty.
\end{align}
\end{theorem}
{\it Proof of Theorem \ref{them:property-solution}.}
For the proof of \eqref{them:spatial-bound-in-B-mild}, please refer
to \cite[Lemma 5.5]{jiang2025uniform}.
Based on the estimate \eqref{them:bound-L4q-2-mild-solution}, one can follow similar arguments used in the proof of \cite[Theorem 2.1]{QQQi2018Optimal} and carefully apply the smooth property \eqref{eq:property-E}-\eqref{IV-spatio-temporal-S(t)} of the semi-group $E(t)$ to show \eqref{th:spatial-regu-exac-solution}-\eqref{th:temporap-regu-mild-solution}. Hence it remains to show \eqref{them:bound-L4q-2-mild-solution}. Note that $X(t)$ possesses the following decomposition $X(t)=Y(t)+W_A(t)$,
where
\begin{align}\label{eq:decomp-equation}
\dd Y(t)=-AY(t)\,\dd t+F(Y(t)+W_A(t))\,\dd t,\; Y(0)=X_0.
\end{align}
Now we are in a position to derive the uniform-in-time moment bounds of $Y(t)$ in $L^{2q}(D)$-norm.
By \eqref{eq:definiton-III-f}, taking the $L^2$-inner product of the above  equation by $Y^{2q-1}(t)$ and using  integration  by parts, we have
\begin{align}
\begin{split}
&\frac1{2q}\left(\|Y(t)\|_{L^{2q}(D)}^{2q}-\|Y(0)\|_{L^{2q}(D)}^{2q}\right)
\\
=
&
\int_0^t\left<Y^{2q-1}(s),-AY(s)+F\big(Y(s)+W_A(s)\big)\right>\,\dd s
\\
=
&
-\frac{2q-1}{q^2}\int_0^t\left<\nabla Y^q(s) ,\nabla Y^q(s)\right>\,\dd s
+
\int_0^t\left<Y^{2q-1}(s), F\big(Y(s)+W_A(s)\big)\right>\,\dd s
\\
\leq
&
-\frac{2q-1}{q^2}\int_0^t\left<\nabla Y^q(s) ,\nabla Y^q(s)\right>\,\dd s
-
c_0\int_0^t\|Y(s)+W_A(s)\|_{L^{4q-2}(D)}^{4q-2}\,\dd s
\\
&
+
c_1\int_0^t\|W_A(s)\|_{L^{4q-2}(D)}^{4q-2}\,\dd s
+c_2t
\\
\leq
&
-\frac{(2q-1)\lambda_1}{q^2}\int_0^t\| Y(s)\|_{L^{2q}(D)}^{2q}\,\dd s
-
\widehat{c}_0\int_0^t\|Y(s)\|_{L^{4q-2}(D)}^{4q-2}\,\dd s
\\
&
+
\widehat{c}_1\int_0^t\|W_A(s)\|_{L^{4q-2}(D)}^{4q-2}\,\dd s
+
\widehat{c}_2t
.
\end{split}
\end{align}
An application of Gronwall's inequality with $Y(0)=X_0$ then yields
\begin{align}
\begin{split}
\|Y(t)\|_{L^{2q}(D)}^{2q}
\leq
&
e^{-\frac{(4q-2)\lambda_1t}q}\|X_0\|_{L^{2q}(D)}^{2q}
+
C_1\int_0^te^{-\frac{(4q-2)\lambda_1(t-s)}q}\|W_A(s)\|_{L^{4q-2}(D)}^{4q-2}\,\dd s
\\
&
+
C_2\int_0^te^{-\frac{(4q-2)\lambda_1(t-s)}q}\,\dd s.
\end{split}
\end{align}
Taking $\frac p {2q}$-moment, $p\geq 2q$ and then taking supreme over $t$ lead to
\begin{align}
\sup_{t\geq 0}\|Y(t)\|_{L^p(\Omega;L^{2q}(D))}^{2q}
\leq
C\|X_0\|_{L^p(\Omega;L^{2q}(D))}^{2q}
+
C\sup_{t\geq 0}\|W_A(t)\|_{L^{\frac{p(4q-2)}{2q}}(\Omega;L^{4q-2}(D))}^{4q-2}
+
C
<\infty,
\end{align}
which in combination with the decomposition $X(t)=Y(t)+W_A(t)$ implies
\begin{align}\label{eq:bound-L2q-X}
\sup_{t\geq 0}\|X(t)\|_{L^p(\Omega;L^{2q}(D))}^{2q}
\leq
C\sup_{t\geq 0}\|Y(t)\|_{L^p(\Omega;L^{2q}(D))}^{2q}
+
C\sup_{t\geq 0}\|W_A(t)\|^{2q}_{L^{p}(\Omega;L^{2q}(D))}
<
\infty
.
\end{align}

Based on \eqref{eq:bound-L2q-X}, we are now devoted to the uniform-in-time moment bound
of $X(t)$ in $L^{4q-2}(D)$. We consider two cases including Case I: dimensions one and two and Case II: dimension three.

{\bf Case I: dimensions one and two.}

By \eqref{eq:property-E}, \eqref{eq:embedding-Hd-C},  a contractive property of the semi-group in $L^{4q-2}(D)$  and \eqref{lem:interpolation-property}
with $\beta= \frac{(q-1)d}{2q-1}$, $s=4q-2$ and $r=2$, we have, for $\delta\in \left(\frac d 2, 2-\frac{(q-1)d}{2q-1}\right)$, $d\in\{1,2\}$
\begin{align}
\begin{split}
\|Y(t)\|_{L^p(\Omega;L^{4q-2}(D))}
\leq
&
\|X(0)\|_{L^p(\Omega;L^{4q-2}(D))}
+
\int_0^t\left\|E(t-s)F(X(s))\right\|_{L^p(\Omega;L^{4q-2}(D))}\,\dd s
\\
\leq
&
\|X(0)\|_{L^p(\Omega;L^{4q-2}(D))}
+
\int_0^t\left\|A^{\frac{(q-1)d}{2(2q-1)}}E(t-s)F(X(s))\right\|_{L^p(\Omega;H)}\,\dd s
\\
\leq
&
\|X(0)\|_{L^p(\Omega;L^{4q-2}(D))}
+
C\int_0^te^{-\frac{\lambda_1(t-s)}2}(t-s)^{-\frac{(q-1)d}{2(2q-1)}-\frac\delta2}
\|F(X(s))\|_{L^p(\Omega;L^1(D))}\,\dd s
\\
\leq
&
\|X(0)\|_{L^p(\Omega;L^{4q-2}(D))}
+
C \Big(1+\Big(\sup_{t\geq 0}\|X(t)\|_{L^{p(2q-1)}(\Omega;L^{2q-1}(D))}\Big)^{2q-1}\Big)
<\infty,
\end{split}
\end{align}
where in the fourth inequality
we used the fact
$0<\frac{(q-1)d}{2(2q-1)}+\frac\delta2<1$
and, for any $t>0$
\[\int_0^te^{-\frac{\lambda_1(t-s)}2}(t-s)^{-\frac{(q-1)d}{2(2q-1)}-\frac\delta2}
\,\dd s\leq C<\infty.\]
Thus we deduce
\begin{align}
\sup_{t\geq 0}\|X(t)\|_{L^p(\Omega;L^{4q-2}(D))}
\leq
\sup_{t\geq 0}\|Y(t)\|_{L^p(\Omega;L^{4q-2}(D))}
+
\sup_{t\geq 0}\|W_A(t)\|_{L^p(\Omega;L^{4q-2}(D))},
\end{align}
as required in dimensions one and two.

{\bf Case II: dimension three.}

We now show \eqref{them:bound-L4q-2-mild-solution} in dimension three.  In this case, $4q-2=6$ and by \eqref{eq:bound-L2q-X} we have
\begin{align}
\sup_{t\geq 0}\|X(t)\|_{L^p(\Omega;L^4(D))}^4<\infty.
\end{align}
This and \eqref{lem:interpolation-property}
with $\beta= \frac12$, $s=2$ and $r=\frac65$ enable us to derive
\begin{align}
\|A^{-\frac12}F(X(s))\|_{L^p(\Omega;H)}
\leq
C\|F(X(s))\|_{L^p(\Omega;L^{\frac 65}(D))}
\leq
C\left(1+\sup_{s\geq 0}\|X(s)\|_{L^{3p}(\Omega;L^4(D))}\right)^3
<\infty.
\end{align}
Applying again \eqref{eq:property-E}, the contractive property of the semi-group in $L^6(D)$  and \eqref{lem:interpolation-property}
but with $\beta= \frac9 {10}$, $s=5$ and $r=2$, it follows
\begin{align}
\begin{split}
\|Y(t)\|_{L^p(\Omega;L^5(D))}
\leq
&
\|E(t-s)X_0\|_{L^p(\Omega;L^5(D))}
+
C\int_0^t\|A^{\frac9{20}}E(t-s)F(X(s))\|_{L^p(\Omega;H)}\,\dd s
\\
\leq
&\|X_0\|_{L^p(\Omega;L^6(D))}
+
\int_0^t(t-s)^{-\frac{19}{20}}e^{-\frac{\lambda_1(t-s)}2}\|A^{-\frac12}F(X(s))\|_{L^p(\Omega;H)}\,\dd s
\\
\leq
&
\|X_0\|_{L^p(\Omega;L^6(D))}
+
C\left(1+\sup_{s\geq 0}\|X(s)\|_{L^{3p}(\Omega;L^4(D))}\right)^3
<\infty.
\end{split}
\end{align}
As a consequence, we have
\begin{align}
\sup_{t\geq 0}\|X(t)\|_{L^p(\Omega;L^5(D))}<\infty.
\end{align}
Similarly as above, one can acquire
\begin{align}
\begin{split}
\|Y(t)\|_{L^p(\Omega;L^6(D))}
\leq
&
\|E(t-s)X_0\|_{L^p(\Omega;L^6(D))}
+
C\int_0^t\|A^{\frac12}E(t-s)F(X(s))\|_{L^p(\Omega;H)}\,\dd s
\\
\leq
&\|X_0\|_{L^p(\Omega;L^6(D))}
+
C\int_0^t\|A^{\frac12+\frac3{20}}E(t-s)A^{-\frac3{20}}F(X(s))\|_{L^p(\Omega;H)}\,\dd s
\\
\leq
&
\|X_0\|_{L^p(\Omega;L^6(D))}
+
C\int_0^t(t-s)^{-\frac12-\frac3{20}}e^{-\frac{\lambda_1(t-s)}2}\,\|F(X(s))\|_{L^p(\Omega;L^{\frac53}(D))}\dd s
\\
\leq
&
\|X_0\|_{L^p(\Omega;L^6(D))}
+
C\left(1+\sup_{s\geq 0}\|X(s)\|_{L^{3p}(\Omega;L^5(D))}\right)^3
<\infty,
\end{split}
\end{align}
which combined with \eqref{eq:asuum-L4q-2-W} shows \eqref{them:bound-L4q-2-mild-solution} in dimension three. Hence, this finishes the proof. $\square$

\section{The fully discrete schemes}
%
In this section, we propose fully discrete finite element methods for SPDEs \eqref{eq:parabolic-SPDE} and show the uniform-in-time  moment bounds of the solution to the fully discrete problem, which will be used later in the convergence analysis.

%
Let $V_h\subset H_0^1(D)$ be the space of continuous
functions that are piecewise linear over the quasi-uniform  triangulation $\mathcal{T}_h$ of $D$.
Define a discrete Laplace operator $-A_h:V_h\rightarrow V_h$ by
\begin{align}\label{eq:definition-discrete-A}
\left<A_hv_h,\chi_h\right>=a(v_h,\chi_h):=\left<\nabla v_h, \nabla \chi_h\right>,\quad \forall v_h,\;\chi_h\in V_h.
\end{align}
The operator $A_h$ is self-adjoint, positive definite on
 $V_h$, and has an orthonormal eigenbasis $\{e_{j,h}\}_{j=1}^{\mathcal{N}_h}$ in ${V}_h$ with
the  corresponding eigenvalues $\big\{\lambda_{j,h}\big\}_{j=1}^{\mathcal{N}_h}$, satisfying
\begin{align}
0 < \lambda_{1,h} < \lambda_{2,h} \leq \cdots \leq \lambda_{j,h}\leq\cdots\leq \lambda_{\mathcal{N}_h, h},
\end{align}
where $\dim(V_h)=\mathcal{N}_h$.
Moreover, we introduce a discrete norm  on $V_h$, defined by
\begin{align}
|v_h|_{\alpha,h}=\|A_h^{\frac\alpha2}v_h\|
=
\Big(\sum_{j=1}^{\mathcal{N}_h}\lambda_{j,h}^\alpha|
\left<v_h,e_{j,h}\right>|^2\Big)^{\frac12}, \;v_h\in V_h,\;\alpha\in \mathbb{R},
\end{align}
which is induced by the discrete inner product
$(v_h,w_h)_{\alpha,h}:=\left<A_h^\alpha v_h, w_h\right>$, $\forall v_h,w_h\in V_h$. Note that
\begin{align}\label{eq-eqvilent-discrete-norem-H1-norm}
|v_h|_1=\|A^{\frac12}v_h\|=\|\nabla v_h\|=\|A_h^{\frac12}v_h\|=|v_h|_{1,h},
\quad v_h\in V_h.
\end{align}
In addition, we introduce a Riesz representation operator $R_h:H_0^1(D)\rightarrow V_h$ defined by
\begin{align}\label{eq:definition-Rh}
a(R_hv,\chi_h)=a(v,\chi_h),
\quad
\forall v\in H_0^1(D),
\,
\chi_h\in V_h,
\end{align}
and a generalized projection operator $P_h:H\rightarrow V_h$ given by
\begin{align}\label{eq:definition-Ph}
\left<P_h v,\,\chi_h\right>=\left<v,\,\chi_h\right>,
\quad
\forall v\in H,
\,
\chi_h\in V_h.
\end{align}
Then the operators $P_h$ and $R_h$ have the following error estimates
\begin{align}\label{eq:property-Ih}
\|(I-P_h)v\|_{-s}
\leq
Ch^{s+r}\|v\|_r, 0\leq s\leq r\leq 2,
\end{align}
and
\begin{align}\label{eq:property-R_h}
\|(I-R_h)v\|_s
\leq Ch^{r-s}\|v\|_r,\; s\in[0,1], r\in[1,2].
\end{align}
It is clear that $P_h$  is also a projection operator from $H^{-1}$ to $V_h$ and
\begin{align}\label{eq:relation-A-Ah-Rh-Ph}
 P_hA=A_hR_h.
 \end{align}
If the mesh $\mathcal{T}_h$ is quasi-uniform, $P_h$ is bounded with respect to  $L^\infty(D)$ and $L^1(D)$ norms:
\begin{align}
\|P_hv\|_{L^\infty(D)}
&\leq
C\|v\|_{L^\infty(D)},\;\forall v\in L^\infty(D),
\label{eq:bound-Ph-Linfty}
\\
\|P_hv\|_{L^1(D)}
&\leq
 C\|v\|_{L^1(D)}, \;\;\forall v\in L^1(D).
 \label{eq:bound-Ph-L1}
\end{align}
For the proof of \eqref{eq:bound-Ph-Linfty} in dimension two,  we  refer to \cite[Lemma 6.1]{thomee2007galerkin}.  The same arguments can be adapted to establish   \eqref{eq:bound-Ph-Linfty} in dimensions one and three. The estimate \eqref{eq:bound-Ph-L1} follows by duality:
\begin{align}
\|P_h v\|_{L^1(D)}
=
\sup_{\chi\in L^\infty(D)}\frac{\left<P_hv,\chi\right>}{\|\chi\|_{L^\infty(D)}}
=
\sup_{\chi\in L^\infty(D)}\frac{\left<v,P_h\chi\right>}{\|\chi\|_{L^\infty(D)}}
=
C\|v\|_{L^1(D)},\;\forall v\in L^1(D).
\end{align}
By the Sobolev interpolation theory, the boundedness of the projection operator $P_h$ in $L^1(D)$ and $L^\infty(D)$ norms implies that there exists a constant $C=C(p)$ such that for any $p\in (1,\infty)$,
\begin{align}\label{eq:bound-P_h-p-norm}
\|P_hv\|_{L^p(D)}
\leq
C\|v\|_{L^p(D)},\, \forall\,v\in L^p(D).
\end{align}
Thanks to \eqref{eq:bound-P_h-p-norm}, $\eqref{eq:embedding-Hd-C}$  with $s=2$, $r=\frac p{p-1}$ and $\beta=\frac{d(p-2)}{2p}$, the inverse inequality $\|v_h\|\leq Ch^{d(\frac12-\frac {p-1}{p})}\|v_h\|_{L^{\frac p{p-1}}(D)}$, for $p\geq 2$
and
\cite[Lemma 5.3]{thomee2007galerkin}, we obtain
\begin{align}\label{eq:l2-lp-relation}
\begin{split}
\| A_h^{-\frac{d(p-2)}{4p}}P_hw\|
\leq
&
Ch^{\frac{d(p-2)}{2p}}\|P_hw\|
+
\|A^{-\frac{d(p-2)}{4p}}P_hw\|
\\
\leq
&
Ch^{\frac{d(p-2)}{2p}} h^{d\left(\frac12-\frac{p-1}p\right)}\|P_h  w\|_{L^{\frac p{p-1}}(D)}
+
C\|P_hw\|_{L^{\frac p{p-1}}(D)}
\\
\leq
&
C\|w\|_{L^{\frac p{p-1}}(D)}.
\end{split}
\end{align}
By duality,  we also have for  $p>2$ and $q=\frac p{p-1}$
\begin{align}\label{eq:relation-Lp-Ah-relatoin}
\begin{split}
\|v_h\|_{L^p(D)}
=
&
\sup_{w\in L^q(D)}\frac{\left<v_h,w\right>}{\|w\|_{L^q(D)}}
=
\sup_{w\in L^q(D)}\frac{\left<A_h^{\frac{d(p-2)}{4p}}v_h, A_h^{-\frac{d(p-2)}{4p}}P_hw\right>}{\|w\|_{L^q(D)}}
\leq
C\|A_h^{\frac{d(p-2)}{4p}}v_h\|.
\end{split}
\end{align}
In addition, the operators $A$ and $A_h$ obey
\begin{align}\label{eq:relation-A-Ah}
C_1\|A_h^{\frac \mu2}P_hv\|
\leq
\|A^{\frac \mu2}v\|
\leq
C_2
\|A_h^{\frac \mu2}P_hv\|,
\quad
\forall
v\in H^{\mu},
\:
\mu\in[-1,1].
\end{align}
Moreover, the operator $A_hP_hA^{-1}$ is bounded, that is
\begin{align}
\|A_hP_hv\|
\leq
C\|v\|_2, \;\;\forall v\in H^2.
\label{eq:bound-Ph-H2}
\end{align}
Indeed,   the inverse inequality $\|A_hP_h\|_{\mathcal{L}(H)}\leq Ch^{-2}$  and $A_hR_h=P_hA$ help us to obtain:
\begin{align}
\|A_hP_hv\|
\leq
\|A_hP_h(I-R_h)v\|
+
\|P_hAv\|
\leq
Ch^{-2}\|(I-R_h)v\|
+
C\|v\|_2
\leq
C\|v\|_2.
\end{align}
Combining \eqref{eq:bound-Ph-H2} and \eqref{eq:relation-A-Ah} gives
\begin{equation}
\label{eq:Ah-A-bound}
\|A_h^{\frac \mu2}P_hA^{-\frac \mu2}\|_{\mathcal{L}(H)}
=
\|(A_h^{\frac \mu2}P_hA^{-\frac \mu2})^*\|_{\mathcal{L}(H)}
=
\|A^{-\frac \mu2}A_h^{\frac \mu2}P_h\|_{\mathcal{L}(H)}
<\infty
,\;
\mu \in [-1, 2].
\end{equation}
Similar as in \eqref{eq:l2-lp-relation}, we have, the inverse inequality
$\|v_h\|\leq Ch^{-\frac d2}\|v_h\|_{L^1(D)}$ and
$\|A^{-\frac \kappa2} v\|
\leq C\|v\|_{L^1(D)}$, for any $\kappa>\frac d2$
\begin{align}\label{eq:L_2-L_1}
\|A_h^{-\frac\kappa2}P_h\chi\|
\leq
C\|A^{-\frac\kappa2}P_h\chi\|+Ch^\kappa\|P_h\chi\|
\leq
C\|v\|_{L^1(D)}.
\end{align}
By the duality, the above estimates  implies, for any $ \kappa\in (\frac d2,2)$
\begin{align}\label{eq:V-norm-control-by-Ah-norm}
\begin{split}
\|v_h\|_V
=&
\|v_h\|_{L^\infty(D)}
=
\sup_{\chi\in L^1(D)}\frac{\left<A_h^{\frac\kappa2} v_h,A_h^{-\frac\kappa2}P_h\chi\right>}{\|\chi\|_{L^1(D)}}
\\
\leq
&
C\sup_{\chi\in L^1(D)}\frac{\|A_h^{\frac\kappa2} v_h\|\|A_h^{-\frac\kappa2}P_h\chi\|}{\|\chi\|_{L^1(D)}}
\leq
C\|A_h^{\frac\kappa2} v_h\|,\;
\forall v_h\in V_h.
\end{split}
\end{align}

 Let
 $\tau$
 be a uniform time step-size and
 $t_m=m\tau,\, m\in \mathbb{Z}^+$.
Then the fully discrete finite element approximation of  the problem \eqref{eq:parabolic-SPDE} is to find $X_{\tau,h}^m\in V_h$ such that
\begin{align}\label{eq:full-discretization}
X_{\tau,h}^m-X_{\tau,h}^{m-1}
+
\tau A_h X_{\tau,h}^m
=
\tau F_{\tau,h}(X_{\tau,h}^{m-1})
+
P_h \Delta W_m,\; X_h^0=P_hX_0,\;
\end{align}
where we denote $\Delta W_m := W(t_m)-W(t_{m-1})$ for brevity and
$F_{\tau,h}:L^{2q-2}(D)\rightarrow H$ is given by
\begin{align}
F_{\tau,h}(u)(x):=f_{\tau,h}(u(x)),\;x\in D.
\end{align}
Here $f_{\tau,h}:\mathbb{R}\rightarrow \mathbb{R}$ is a modification of the mapping $f$, defined by
\begin{align}\label{eq:definiton-I-ftauh}
f_{\tau,h}(u):=\frac{f(u)}{\big(1+(\beta_1\tau^\theta+\beta_2 h^\rho)|u|^{\frac{2q-2}{\alpha}}\big)^\alpha},
\end{align}
where the parameters $\theta, \rho, \beta_1, \beta_2>0$ and we require $\max\{\alpha \theta,\frac12\alpha \rho\}<1+\frac d{4q(2q-1)}-\frac d4$.
By iteration,
the solution of \eqref{eq:full-discretization} can be rewritten as
\begin{align}\label{eq:fully-discrete-problem}
X_{\tau,h}^m=E_{\tau,h}^mP_h X_0
+
\tau \sum_{j=0}^{m-1}E_{\tau,h}^{m-j}P_hF_{\tau,h}(X_{\tau,h}^j)
+
\sum_{j=1}^{m}E_{\tau,h}^{m+1-j}P_h\Delta W_j,
\end{align}
for $m\in \mathbb{N}^+$, where the operator $E_{\tau,h}^m$ is defined by
$
E_{\tau,h}^m :=(I+\tau A_h)^{-m} P_h.$

We now present two useful results that play an important role in proving moment bounds of the numerical solution and the convergence analysis below. The first one concerns properties of the operator family $\{ E_{\tau,h}^m \}_{m \geq 1}$.
\begin{lemma}\label{lem:eq-smooth-ENm}
 Let Assumption \ref{assum:linear-operator-A} be fulfilled . Then there exists a constant $C$ independent of $h, \tau$ and $m\in\{1, 2,\cdots,\}$ such that for any $x \in H$
\begin{align}
\|A_h^{\frac\mu2}E_{\tau,h}^mP_hx\|
\leq
&
C\min\{t_m^{-\frac\mu2},t_m^{-2}\}\|x\|,\;\mu\in [0,2],
\label{lem:eq-spatial-regu-ENm}
\\
\Big(\tau\sum_{i=n}^m\|A_h^{\frac\varrho2}E_{\tau,h}^iP_hx\|^2\Big)^{\frac12}
\leq
&
Ct_{m-n}^{\frac{1-\varrho}2}\|x\|,\;\varrho\in[0,1],
\label{lem:eq-spatial-regu-ENm-sum-II}
\\
\Big\|\tau\sum_{i=n}^mA_h^{\varrho}E_{\tau,h}^iP_hx\Big\|
\leq
&
Ct_{m-n}^{1-\varrho}\|x\|,\;\varrho\in[0,1],
\label{lem:eq-spatial-regu-ENm-sum}
\\
\|A_h^{-\frac\nu2}(I-E_{\tau,h}^m)P_hx\|
\leq
&
C t_m^{\frac\nu2} \|x\|,\;\nu\in[0,2].
\label{lem:eq-temporal-regu-ENm}
\end{align}

\end{lemma}
The proof of this lemma will be given in the Appendix.

Based on the above results on $f$, we can derive some properties  of $f_{\tau,h}(\cdot)$, which can be shown by a slight modification of the proof of \cite[(3.8)-(3.11)]{jiang2025uniform}.
\begin{lemma}\label{assum:mofification-f} Let Assumption \ref{assum:nonlinearity} be fulfilled and let the mapping $f_{\tau,h}(\cdot)$ be given by \eqref{eq:definiton-I-ftauh}.
Then, there exist constants $\tau^*, h^* \in (0,\infty)$ such that for $0<\tau\leq \tau^*$ and $0<h\leq h^*$,
$f_{\tau,h}$ satisfies the following conditions: for some $\theta>0$ and $\rho>0$, $\alpha\in(0,1]$ satisfying  $\max\{\alpha\theta, \frac12\alpha\rho\}<1+\frac d{q(2q-1)}-\frac d4$, there exist constants $\widetilde{c}_0, \widetilde{c}_1, \widetilde{c}_2, \widetilde{c}_3, \widetilde{c}_4, \widetilde{c}_5$, independent of $\tau$ and $h$, such that for any $u,v \in \mathbb{R}$
\begin{align}
2(u+v)f_{\tau,h}(u)+\tau|f_{\tau,h}(u)|^2
\leq
&
\widetilde{c}_1(1+|v|^{2q})-\widetilde{c_0}|u|^2,
\label{asum:condition-ftauh-I}
\\
|f_{\tau,h}(u)|
\leq
&\widetilde{c}_2|f(u)|,
\label{asum:condition-ftauh-Ii}
\\
|f_{\tau,h}(u)|
\leq
&
\widetilde{c}_3(1+|u|+(\tau^\theta+h^\rho)^{-\alpha}|u|),
\label{asum:condition-ftauh-Iii}
\\
|f_{\tau,h}(u)-f_{\tau,h}(v)|
\leq
&
\widetilde{c}_4(1+|u|^{2q-2}+|v|^{2q-2})|u-v|,
\label{asum:condition-ftauh-IiiI}
\\
|f(v)-f_{\tau,h}(v)|
\leq
&
\widetilde{c}_5(\tau^\theta+h^\rho)|v|^{\frac{2q-2}\alpha}|f(v)|.
\label{lem:f-ftau-I}
\end{align}
\end{lemma}

%
\section{Error estimates in Banach spaces for the finite element fully discretizations of the deterministic linear parabolic equation}
In order to derive uniform-in-time moment bounds of the numerical solution $X_{\tau,h}^m$, we need to establish new error estimates in a Banach space $L^p(D), p \geq 2$ for the fully discrete finite element approximation of the linear parabolic equation:
\begin{align}\label{eq:linear-parabolic-eq}
u(t)+Au(t)=0,\;u(0)=v,
\end{align}
 whose solution can be written as $u(t)=E(t)v$. Define the fully discrete approximation operators
\[
\Phi_{\tau,h}^n
: =
E(t_n)-E_{\tau,h}^nP_h.
\]

\begin{theorem}
\label{thm:infty-determing-error}
Let Assumption \ref{assum:linear-operator-A} be fulfilled. Then, there exists a constant
 $C$ independent of $h,\tau$ and $t_n$ such that for any $p\geq 2$
\begin{align}\label{lem:error-determinstic-error-estimate}
\|\Phi_{\tau,h}^n v\|_{L^p(D)}
\leq
 C(h^{2+\frac dp-\frac d 2}+\tau^{1+\frac d{2p}-\frac d4}) \min\{t_n^{-1},t_n^{-2}\}  \|v\|, \quad d \in \{1,2,3\}.
\end{align}
\end{theorem}
{\it Proof of Theorem \ref{thm:infty-determing-error}.}
We begin with deriving error estimates of the semi-discrete finite element approximation of the linear equation. The semidiscrete problem is to find $u_h(t)\in V_h$ such   that
\begin{align}\label{eq:semidiscrete-probleme}
\frac{\dd u_h(t)}{\dd t}
+
A_hu_h(t)=0, u_h(0)=P_hv,
\end{align}
whose solution  can be written as $u_h(t)=E_h(t)P_hv$,
where $E_h(t):=e^{-tA_h}$ is the semi-group generated by $-A_h$. To streamline the error analysis, we denote
\[
e_h(t):=u(t)-u_h(t),
\quad
\xi_h(t) := P_hu(t)-u_h(t),
\quad
\rho(t) := ( I - R_h ) u(t).
\]
Owing to \eqref{eq:linear-parabolic-eq}, \eqref{eq:semidiscrete-probleme} and the fact $P_hA=A_hR_h$, one knows
\begin{align}
\frac{\dd \xi_h(t)}{\dd t}+A_h\xi_h(t)=A_h(R_h-P_h)u(t)=-A_hP_h\rho(t).
\end{align}
By integration and noting $\xi_h(0)=0$, one gets
\begin{align}
\xi_h(t)=
-\int_0^tE_h(t-s)A_hP_h\rho(s)\,\dd s
=
\int_0^tE_h'(t-s)P_h\rho(s)\,\dd s
.
\end{align}
Noting $t^2=s^2+2s(t-s)+(t-s)^2$, one can write
\begin{align}\label{eq:decompose-Xih}
\begin{split}
t^2\xi_h
=&
\left(\int_0^{\frac t2}+\int_{\frac t2}^t\right)
\big(s^2+2s(t-s)+(t-s)^2\big)E_h'(t-s)P_h\rho(s)\,\dd s
\\
=:
&
\sum_{j=1}^3(I_j+II_j).
\end{split}
\end{align}
Using the inverse inequality $\|v_h\|_{L^p(D)}\leq Ch^{\frac dp-\frac d2}\|v_h\|, \forall v_h\in V_h$, $p
\geq 2$ and
the fact $\|E_h'(t)P_hv\| \leq Ct^{-1}\|v\|$, we have, for any $t>0$
\begin{align}
\|E_h'(t)P_hv\|_{L^p(D)}
\leq
C t^{-1} h^{\frac dp-\frac d2}\|v\|.
\end{align}
By \eqref{eq:property-E} and \eqref{eq:property-R_h}, we deduce
\begin{align}\label{eq:error-rho}
\begin{split}
\|\rho(t)\|+t\|\rho_t(t)\|
\leq
&
Ch^2(\|Au(t)\|+t\|Au_t(t)\|)
\leq
Ch^2(\|AE(t)v\|+t\|AE'(t)v\|)
\\
\leq
&
Ch^2t^{-1}e^{-\frac{\lambda_1t}2}\|v\|.
\end{split}
\end{align}
Thus
\begin{align}\label{eq:estimate-I1}
\begin{split}
\begin{split}
\|I_1\|_{L^p(D)}
\leq
&
\int_0^{\frac t2}\|s^2E_h'(t-s)P_h\rho(s)\|_{L^p(D)}\,\dd s
\leq
C h^{\frac dp-\frac d2}\int_0^{\frac t2}s^2(t-s)^{-1}\|\rho(s)\|\,\dd s
\\
\leq
&
Ch^{2+\frac dp-\frac d2}t^{-1}\int_0^{\frac t2}se^{-\frac{\lambda_1 s}2}\|v\|\,\dd s
\leq
Ch^{2+\frac dp-\frac d2} t^2 \min\{t^{-1}, t^{-2}\}\|v\|,
\end{split}
\end{split}
\end{align}
where in the third inequality we used the fact
\[
t^{-1}\int_0^{\frac t2}se^{-\frac {\lambda_1 s}2}\,\dd s
\leq
Ct^{-1}\int_0^{\frac t2}s\,\dd s
\leq Ct,
\]
 for $0<t\leq 1$ and
\[
  t^{-1}\int_0^{\frac t2}se^{-\frac {\lambda_1 s}2}\,\dd s
  \leq
  C \int_0^{\frac t2}e^{-\frac {\lambda_1 s}2}\,\dd s
  \leq
   C,
\]
for $t\geq 1$.
For $II_1$,  after integration by parts, we infer
\begin{align}
\begin{split}
II_1
=
&
-\int_{\frac t2}^ts^2E_h'(t-s)P_h\rho(s)\,\dd s
\\
=
&
\left[E_h(t-s)s^2P_h\rho(s)\right]\Big|_{\frac t2}^t-\int_{\frac t2}^tE_h(t-s)\left(2sP_h\rho(s)+s^2P_h\rho_s(s)\right)\,\dd s
\\
=
&
t^2P_h\rho(t)-\frac {t^2}4E_h( t/2)P_h\rho( t/2)-\int_{\frac t2}^tE_h(t-s)\left(2sP_h\rho(s)+s^2P_h\rho_s(s)\right)\,\dd s.
\end{split}
\end{align}
Similarly, one can acquire, by \eqref{eq:error-rho} and the inverse inequality $\|v_h\|_{L^p(D)}\leq Ch^{\frac d p-\frac d2}\|v_h\|,\;$ for $p\geq 2$
\begin{align}\label{eq:estimate-II2}
\begin{split}
\|II_1\|_{L^p(D)}
\leq
&
Ct^2 h^{\frac dp-\frac d2} (\|\rho(t)\|
+
\|\rho(t/2)\|)
+
Ch^{\frac dp-\frac d2}\int_{\frac t 2}^t(s\|\rho(s)\|+s^2\|\rho_s(s)\|)\,\dd s
\\
\leq
&
Ct h^{2+\frac dp-\frac d2} (e^{-\frac{\lambda_1 t}2}+e^{-\frac{\lambda_1 t}4})\|v\|
+
Ch^{2+\frac dp-\frac d2}\int_{\frac t 2}^t e^{-\frac {\lambda_1 s}2}\,\dd s\|v\|
\\
\leq
&
Ch^{2+\frac dp-\frac d2} t^2 \min\{t^{-1},t^{-2}\}\|v\|,
\end{split}
\end{align}
and
\begin{align}\label{eq:estimate-I2-II2-II3}
\begin{split}
&\|I_2+II_2+II_3\|_{L^p(D)}
\\
\leq
&
Ch^{\frac dp-\frac d2}\left(\int_0^ts(t-s)\|E_h'(t-s)P_h\rho(s)\|\,\dd s
+
\int_{\frac t2}^t(t-s)^2\|E_h'(t-s)P_h\rho(s)\|\,\dd s\right)
\\
\leq
&
Ch^{2+\frac dp-\frac d2}\left(\int_0^t e^{-\frac{\lambda_1(t-s)}2}e^{-\frac {\lambda_1 s}2}\|v\|\,\dd s+\int_{\frac t2}^t(t-s)s^{-1}e^{-\frac {\lambda_1 s}2}\|v\|\,\dd s\right)
\\
\leq
&
Ch^{2+\frac dp-\frac d2} t^2 \min\{t^{-1},t^{-2}\}\|v\|.
\end{split}
\end{align}
For $I_3$,  we integrate by parts to obtain
\begin{align}
\begin{split}
I_3
=&
-\int_0^{\frac t 2}(t-s)^2E_h'(t-s)P_h\rho(s)\,\dd s
\\
=
&
-(t-s)^2E_h'(t-s)P_h\widehat{\rho}(s)|_{s=0}^{\frac t2}
+
\int_0^{\frac t 2}\big((t-s)^2E_h''(t-s)-2(t-s)E_h'(t-s)\big)P_h\widehat{\rho}(s)\,\dd s
\\
=
&
\frac{t^2}4E_h'(t/2)P_h\widehat{\rho}(t/2)
+
\int_0^{\frac t 2}\big((t-s)^2E_h''(t-s)-2(t-s)E_h'(t-s)\big)P_h\widehat{\rho}(s)\,\dd s,
\end{split}
\end{align}
where we denote
\[
\widehat{\rho}(t) := \int_0^t\rho(s)\,\dd s.
\]
Since $E_h''(t)=A_h^2E_h(t)$ and $\|A_h^\mu E_h(t)\|_{\mathcal{L}(H)}
\leq
C t^{-\mu} e^{-\frac{\lambda_1 t}2}$ for any $\mu\geq 0$, $t>0$, one can show
\begin{align}
t^2\|E_h''(t)P_hv\|+t\|E_h'(t)P_hv\|
\leq
 C e^{\frac{-\lambda_1 t}2}\|v\|.
\end{align}
By \eqref{eq:property-R_h} and \eqref{IV-spatio-temporal-S(t)}, we show
\begin{align}
\left\|\widehat{\rho} (t) \right\|
\leq
Ch^2\left\|\int_0^tA u(s)\,\dd s\right\|
\leq
Ch^2\|v\|.
\end{align}
Thus, using the inverse inequality $\|v_h\|_{L^p(D)}\leq Ch^{\frac dp-\frac d2}\|v_h\|, \forall v_h\in V_h$ shows, for any $p\geq 2$
\begin{align}\label{eq:estimate-I3}
\begin{split}
\|I_3\|_{L^p(D)}
\leq
&
Ch^{\frac dp-\frac d2}\Big(\|t^2E_h'(t/2)P_h\widehat{\rho}(t/2)\|
+
\int_0^{\frac t 2}\big((t-s)^2\|E_h''(t-s)P_h\widehat{\rho}(s)\|
\\
&
+
2(t-s)\|E_h'(t-s)P_h\widehat{\rho}(s)\|\big)\,\dd s\Big)
\\
\leq&
Ch^{2+\frac dp-\frac d2}\left( te^{\frac{-\lambda_1 t}4}
+
\int_0^{\frac t 2} e^{\frac{-\lambda_1 (t-s)}2}\,\dd s\right)\|v\|
\\
\leq
&
Ch^{2+\frac dp-\frac d2}t e^{-\frac{\lambda_1 t}4}\|v\|
\leq
Ch^{2+\frac dp-\frac d2}t^2\min\{t^{-1},t^{-2}\}\|v\|.
\end{split}
\end{align}
Finally, plugging estimates \eqref{eq:estimate-I1}, \eqref{eq:estimate-II2}, \eqref{eq:estimate-I2-II2-II3} and \eqref{eq:estimate-I3} into \eqref{eq:decompose-Xih} leads to, for any $p\geq2$
\begin{align}\label{eq:xih-estimate}
\|\xi_h(t)\|_{L^p(D)}
\leq
C\min\{t^{-1},t^{-2}\} h^{2+\frac dp-\frac d2}\|v\|.
\end{align}
To bound $\rho(t)$ in  $L^p$-norm, we employ the Gagliardo-Nirenberg inequality
\begin{align}
\|v\|_{L^p(D)}
\leq
C\sum_{K\in \mathcal{T}_h}\|v\|^{\frac {(p-2)d}{4p}}_{H^2(K)}\sum_{K\in \mathcal{T}_h}\|v\|_{L^2(K)}^{1-\frac {(p-2)d}{4p}}
\mathrm{}\end{align}
to conclude, for any $p\geq 2$
\begin{align}\label{eq:rho-estimate}
\begin{split}
\|\rho(t)\|_{L^p(D)}
\leq
&
C\sum_{K\in \mathcal{T}_h}\|\rho(t)\|^{\frac {(p-2)d}{4p}}_{H^2(K)}\sum_{K\in \mathcal{T}_h}\|\rho(t)\|_{L^2(K)}^{1-{\frac {(p-2)d}{4p}}}
\\
\leq
&
C(\|\rho(t)\|_{H^1(D)}+|u(t)|_{H^2(D)})^{\frac {(p-2)d}{4p}}\|\rho(t)\|^{1-{\frac {(p-2)d}{4p}}}
\\
\leq
&
Ch^{2+\frac dp- \frac{d}2}\big((1+h+h^2)\|u(t)\|_2\big)^{\frac {(p-2)d}{4p}}\|u(t)\|_2^{1-{\frac {(p-2)d}{4p}}}
\\
\leq
&
Ch^{2+\frac dp- \frac{d}2} t^{-1}e^{-\frac{\lambda_1 t}2}\|v\|
\leq C h^{2+\frac dp- \frac{d}2}\min\{t^{-1},t^{-2}\}\|v\|,
\end{split}
\end{align}
where in the second inequality we used the fact $R_hu$ is a piecewise continuous linear function over the quasiuniform triangulation $\mathcal{T}_h$ of $D$. Therefore, by \eqref{eq:xih-estimate} and \eqref{eq:rho-estimate}, we obtain, for any $p\geq 2$
\begin{align}
\|(E(t)-E_h(t)P_h)v\|_{L^p(D)}
\leq
C h^{2+\frac dp-\frac d2}\min\{t^{-1}, t^{-2}\}\|v\|.
\end{align}
It remains to bound $\|(E_h(t_n)-E_{\tau,h}^n)P_hv\|_{L^p(D)}$. For this, we need to bound
$\|(E_h(t_n)-E_{\tau,h}^n)P_hv\|$ and $\|A_h(E_h(t_n)-E_{\tau,h}^n)P_hv\|$.
By using the expansion of $P_hv$ in terms of
$\{(\lambda_{j,h},e_{j,h}^j)\}_{j=1}^{\mathcal{N}_h}$, one can get
\begin{align}\label{eq:e-E-error}
\left\|(E_h(t_n)-E_{\tau,h}^n)P_hv\right\|^2
=
\sum_{j=1}^{\mathcal{N}_h}\left(e^{-\lambda_{j,h}t_n}-(1+\tau \lambda_{j,h})^{-n}\right)^2\left<P_hv, e_{j,h}\right>^2.
\end{align}
To proceed further, consider two cases: $\tau \lambda_{j,h}\leq 1$ and $\tau \lambda_{j,h}>1$, and denote $r(z) := (1+z)^{-1}$, $z>0$.
As shown in the proof of \cite[Theorem 7.1]{thomee2007galerkin}, there exist two positive constants $C$ and $c$ such that
\begin{align}
|r(z)-e^{-z}|
& \leq
Cz^2, \quad \forall z\in[0,1],
\\
\label{eq:r(z)-ez}
r(z)
& \leq
e^{-cz}, \quad \forall z\in[0,1].
\end{align}
These two inequalities suffice to ensure that, for $n= 1, 2, 3, \cdots,$
\begin{align}
\left|r(z)^n-e^{-nz})\right|
=
\Big|(r(z)-e^{-z})\sum_{j=0}^{n-1}r(z)^{n-1-j}e^{-zj}\Big|
\leq
Cz^2n e^{-c(n-1)z}.
\end{align}
For the  case $\tau \lambda_{j,h}\leq 1$, the above inequality and the bound $\sup_{s\in[0,\infty)}se^{-cs}<\infty$ imply,
 \begin{align}\label{eq:bound-expansion-full-error-not-big-1}
 \begin{split}
 |e^{-\lambda_{j,h}t_n}-r(\tau \lambda_{j,h})^n|
 \leq
 &
 Cn\tau^2\lambda_{j,h}^2e^{-c\lambda_{j,h}t_{n-1}}
 \leq
 C\tau
  \lambda_{j,h}e^{-\frac{c \lambda_{j,h} t_n}2} t_n\lambda_{j,h} e^{-\frac{c \lambda_{j,h} t_n}2}
 \\
 \leq
 &
 C\tau t_n^{-1} e^{-\frac{c \lambda_{1,h} t_n}4}
 \leq
 C\tau\min\{t_n^{-1}, t_n^{-2}\}.
 \end{split}
 \end{align}
For the case $\tau\lambda_{j,h}>1$, using \eqref{eq:r(A)-taulambda-big-1} with $\mu=0$, \eqref{eq:r(A)-taulambda-equality=1} with $\mu=0$ and employing  the fact $\sup_{\lambda\geq 1} e^{-\frac{n\lambda}2}\leq C n^{-1}$  lead to
 \begin{align}
 \begin{split}
|e^{-\lambda_{j,h}t_n}-r(\tau \lambda_{j,h})^n|
\leq
&
|e^{-\lambda_{j,h}t_n}|+|r(\tau \lambda_{j,h})^n|
\\
\leq
&
C e^{-\frac{\lambda_{j,h}t_n}2} \sup_{\lambda\geq 1}e^{-\frac{n\lambda}2}
+
C\tau^2 t_n^{-2}
\\
\leq
&
Ce^{-\frac{\lambda_1t_n}4} n^{-1}
+
C\tau^2 t_n^{-2}
\\
\leq
&
C\tau t_n^{-1}e^{-ct_n}
+
C\tau^2 t_n^{-2}
\leq
C\tau\min\{t_n^{-1}, t_n^{-2}\},
\end{split}
 \end{align}
 which combined with \eqref{eq:bound-expansion-full-error-not-big-1} shows
 \begin{align}
 |e^{-\lambda_{j,h}t_n}-r(\tau \lambda_{j,h})^n|
 \leq
 C\tau\min\{t_n^{-1},t_n^{-2}\}.
 \end{align}
This together with \eqref{eq:e-E-error} enables us to obtain
\begin{align}
\|(E_h(t_n)-E_{\tau,h}^n)P_hv\|
\leq
C\tau\min\{t_n^{-1}, t_n^{-2}\} \|v\|.
\end{align}
Now let us bound $\|A_h(E_h(t_n))-E_{\tau,h}^n)P_hv\|$. Similarly as in \eqref{eq:property-E}, we derive
\begin{align}
\|A_h^{\frac{\mu}2}E_h(t_n)P_hv\|
\leq
C t_n^{-\frac\mu2}e^{-\frac{\lambda_1 t_n}2}\|v\|
\leq
C\min\{t_n^{-\frac\mu2},t_n^{-2}\}.
\end{align}
Then the above estimate in combination with \eqref{lem:eq-spatial-regu-ENm} implies
\begin{align}
\|A_h(E_h(t_n)-E_{\tau,h}^n)P_hv\|
\leq
C\min\{t_n^{-1}, t_n^{-2}\}\|v\|.
\end{align}
By the discrete analogue of the Gagliardo-Nirenberg inequality
\[
\|v_h\|_{L^p(D)}
\leq
C\|A_hv_h\|^{\frac {(p-2)d}{4p}}\|v_h\|^{1-{\frac {(p-2)d}{4p}}},
\quad
p\geq2,
\]
we have
\begin{align}
\begin{split}
\|(E_h(t_n)-E_{\tau,h}^n)P_hv\|_{L^p(D)}
\leq&
\|A_h(E_h(t_n)-E_{\tau,h}^n)P_hv\|^{\frac {(p-2)d}{4p}}\|(E_h(t_n)-E_{\tau,h}^n)P_hv\|^{1-\frac {(p-2)d}{4p}}
\\
\leq
&
C \tau^{1+\frac d{2p}-\frac d4} \min\{t_n^{-1},t_n^{-2}\}\|v\|.
\end{split}
\end{align}
This completes the proof of the lemma. $\Box$

In the subsequent weak error analysis, we also need the following error estimate:
\begin{equation}
\label{lem:deterministic-error-in-H-1}
\|\Phi_{\tau,h}^n v\|_{-1}
\leq
C(h^{\mu}+\tau^{\frac\mu2})\|v\|_{-1+\mu},
\quad
\mu\in[0,2].
\end{equation}
%
In the case $\mu\in[1,2]$, we refer to \cite[Lemma 2.8]{Yan2005galerkin}. By interpolation theory, it suffices to validate \eqref{lem:deterministic-error-in-H-1} in the case $\mu=0$, which follows from \eqref{eq:relation-A-Ah} and the stability of $E_{\tau,h}^m$ and $E(t)$ in $H$.
%
\section{Uniform-in-time moment bounds of the fully discrete  schemes}

The objective of this section is to derive uniform-in-time moment bounds of the considered fully discrete scheme \eqref{eq:full-discretization}, which essentially relies on uniform moment bounds for the discretized version of the stochastic  convolution
$W_A(t)$, defined by
\begin{align}
W_{\tau,h}^m:=\sum_{j=1}^{m}E_{\tau,h}^{m+1-j}P_h\Delta W_j.
\end{align}

\begin{lemma}
\label{lem:bound-discrete-stochastic-convulution}
Let Assumptions \ref{assum:linear-operator-A}, \ref{assum:noise-term} be fulfilled with $\gamma\in \left[\frac{(2q^2-q-1)d}{2q(2q-1)},2\right]$.
For any $p>1$, there exist a positive constant $C(Q,p)$ such that
\begin{align}\label{lem:bound-discrete-stochastic-convolution}
\sup_{m\in \mathbb{N}}\|W_{\tau,h}^m\|_{L^p(\Omega;L^{2q(2q-1)}(D))}
+
\sup_{m\in \mathbb{N}}\|A_h^{\frac\gamma2}W_{\tau,h}^m\|_{L^p(\Omega; H)}
\leq
C(Q,p)<\infty.
\end{align}
Additionally, if $\gamma\in(\frac d2,2]$ for $d \in \{1, 2, 3\}$ or $\gamma\in (0,\frac12)$ with $Q=I$ for $d =1$,
then there exists a constant $C=C(Q,p)$ such that
\begin{align}\label{lemeq:V-bound-discrete-stoch-conv}
\sup_{m\in \mathbb{N}}\|W_{\tau,h}^m\|_{L^p(\Omega; V)}
\leq
C(Q,p)<\infty.
\end{align}
\end{lemma}
The proof is provided in the appendix.
%
%
%
A slight modification of the proof of \cite[Lemma 4.4]{jiang2025uniform} gives the following lemma on the uniform-in-time moment  bound of the fully discretization \eqref{eq:full-discretization} in $L^2(D)$.
\begin{lemma}
\label{lem:uniform-moment-L2}
(Uniform-in-time moment bounds in $L^2(D)$) Let Assumptions \ref{assum:linear-operator-A}-\ref{assum:intial-value-data} be fulfilled with $\gamma\in\big[\tfrac{(2q^2-q-1)d}{2q(2q-1)},2\big]$ and $\max\{\alpha\theta,\frac12\alpha\rho\}<1+\tfrac d{4q(2q-1)}-\tfrac d4$.
For any $p\geq 1$, there exists a constant $C=C(X_0,Q,p,q)$ such that
\begin{align}
\sup_{m\in \mathbb{N}_0}\|X_{\tau,h}^m\|_{L^p(\Omega; H)}
\leq C(X_0,Q,p,q)<\infty.
\end{align}
\end{lemma}

\begin{theorem}\label{th:uniform-moment-bound}
Let all conditions in Lemma \ref{lem:uniform-moment-L2} hold true for $\gamma\in\left[\max\left\{\frac{d(2q-2+\alpha(2q-3))}{2(2q-2+\alpha(2q-1))},\tfrac {(2q^2-q-1)d}{2q(2q-1)}\right\},2\right]$ and let $X_{\tau,h}^m$ be produced by the fully discretization scheme \eqref{eq:full-discretization}.
Then there exists a positive constant $C=C(X_0,Q,p,q)$ such that, for any $p\geq 1$,
\begin{align}\label{th:spatal-regu-full-solu}
\sup_{m\in \mathbb{N}_0}\big(\|X_{\tau,h}^m\|_{L^p(\Omega; L^{2q(2q-1)}(D))}
+
\|A_h^{\frac \gamma 2}X_{\tau,h}^m\|_{L^p(\Omega; H)}\big)
\leq C<\infty,
\end{align}
and, for $\beta\in [0,\gamma]$
\begin{align}
\label{th:temporal-regu-full-solu}
\big \|A_h^{\frac\beta2}(X_{\tau,h}^m-X_{\tau,h}^{n}) \big\|_{L^p(\Omega; H)}
\leq
C|t_m-t_n|^{\min\{\frac12,\frac{\gamma-\beta}2\}}.
\end{align}
In addition, if $\gamma\in(\frac d2,2]$ or $\gamma\in\left(0,\frac12\right]$ with $Q=I$ in dimension one, then
\begin{align}\label{them:bound-numerical-V}
\sup_{m\in \mathbb{N}_0}\|X_{\tau,h}^m\|_{L^p(\Omega; V)}
\leq C<\infty.
\end{align}
\end{theorem}
{\it Proof of Theorem \ref{th:uniform-moment-bound}.}
To show this theorem, we introduce two auxiliary processes,
\begin{align}
\mathcal{R}_{\tau,h}^m
:=\tau\sum_{l=0}^{m-1}\big(E_{\tau,h}^{m-l}P_h-E(t_m-t_l)\big)F_{\tau,h}(X_{\tau,h}^l)
+
W_{\tau,h}^m
+
E_{
\tau,h}^mP_hX_0-E(t_m)X_0,
\end{align}
and
\begin{align}
\widetilde{Y}_{\tau,h}^m:=X_{\tau,h}^m-\mathcal{R}_{\tau,h}^m.
\end{align}
Recalling \eqref{eq:fully-discrete-problem}, we arrive at, for all $m\in \mathbb{N}_0$
\begin{align}
\widetilde{Y}_{\tau,h}^m
=
E(t_m)\widetilde{Y}_{\tau,h}^0+\tau\sum_{l=0}^{m-1}E(t_m-t_l)F_{\tau,h}(X_{\tau,h}^l)
=
E(\tau)\widetilde{Y}_{\tau,h}^{m-1}
+
\tau E(\tau)F_{\tau,h}(X_{\tau,h}^{m-1}).
\end{align}
By the proof of \cite[Theorem 4.3]{jiang2025uniform}, it suffices to show
$\sup_{m\in \mathbb{N}_0}\|\mathcal{R}_{\tau,h}^m\|_{L^p(\Omega;L^{2q(2q-1)}(D))}<\infty$.
Note first that we have, by \eqref{eq:property-E},  \eqref{lem:eq-spatial-regu-ENm}, \eqref{eq:relation-Lp-Ah-relatoin}, \eqref{lem:interpolation-property}
with $s=2q(2q-1)$, $r=2$ and $\beta=\frac{(2q^2-q-1)d}{2q(2q-1)}$
\begin{align}
\begin{split}
\|E_{\tau,h}^{m+1}P_hv-E(t_{m+1})v\|_{L^{2q(2q-1)}(D)}
\leq
&
C  \Big( \Big \|A_h^{\frac{(2q^2-q-1)d}{4q(2q-1)}}E_{\tau,h}^{m+1}P_hv \Big \|
+
\Big \|A^{\frac{(2q^2-q-1)d}{4q(2q-1)}}E(t_{m+1})v \Big\| \Big )
\\
\leq
&
C\min\Big\{t_{m+1}^{-2},
t_{m+1}^{-\frac{(2q^2-q-1)d}{4q(2q-1)}}\Big\}\left\|v\right\|,
\end{split}
\end{align}
or
\begin{align}
\begin{split}
\|E_{\tau,h}^{m+1}P_hv-E(t_{m+1})v\|_{L^{2q(2q-1)}(D)}
\leq
C
\Big \| A^{\frac{(2q^2-q-1)d}{4q(2q-1)}}v \Big \|.
\end{split}
\end{align}
Since $\max\{\alpha\theta,\frac12\alpha\rho\}<1+\frac d{4q(2q-1)}-\frac d4$, there exists a constant $\overline{\delta}\in (0,1)$ such that  $(1+\frac d{4q(2q-1)}-\frac d4)\overline{\delta}\geq \max\{\alpha\theta,\frac12\alpha\rho\}$.
Therefore, the above two estimates together with \eqref{asum:condition-ftauh-Iii} and \eqref{lem:error-determinstic-error-estimate} for
$p=2q(2q-1)$ imply
\begin{align}\label{eq:bound-R}
\begin{split}
&
\sup_{m\in \mathbb{N}_0}\|\mathcal{R}_{\tau,h}^{m+1}\|_{L^p(\Omega;L^{2q(2q-1)}(D))}
\\
\leq &
\sup_{m\in \mathbb{N}_0}\tau\sum_{l=0}^m\left\|\big(E_{\tau,h}^{m+1-l}P_h
-
E(t_{m+1}-t_l)\big)F_{\tau,h}(X_{\tau,h}^l)\right\|_{L^p(\Omega;L^{2q(2q-1)}(D))}
\\
&
+
\sup_{m\in \mathbb{N}_0}\left\|W_{\tau,h}^{m+1}\right\|_{L^p(\Omega;L^{2q(2q-1)}(D))}
+
\sup_{m\in \mathbb{N}_0}\left\|(E(t_{m+1})
-
E_{\tau,h}^{m+1}P_h)X_0\right\|_{L^p(\Omega;L^{2q(2q-1)}(D))}
\\
\leq &
\sup_{m\in \mathbb{N}_0}\tau\sum_{l=0}^m\left\|\big(E_{\tau,h}^{m+1-l}P_h
-
E(t_{m+1}-t_l)\big)F_{\tau,h}(X_{\tau,h}^l)\right\|^{\overline{\delta}}_{L^p(\Omega;L^{2q(2q-1)}(D))}
\\
&
\cdot
 \left\|\big(E_{\tau,h}^{m+1-l}P_h
-
E(t_{m+1}-t_l)\big)F_{\tau,h}(X_{\tau,h}^l)\right\|^{1-\overline{\delta}}_{L^p(\Omega;L^{2q(2q-1)}(D))}
+
C(Q,p)
\\
\leq
&
C\sup_{m\in \mathbb{N}_0}\tau\sum_{l=0}^m
\left(h^{2+\frac {d}{2q(2q-1)}-\frac d2}+\tau^{1+\frac d{4q(2q-1)}-\frac d4}\right)^{\overline{\delta}} \min\{t_{m+1-l}^{-\overline{\delta}},t_{m+1-l}^{-2\overline{\delta}}\}
\min\left\{t_{m+1-l}^{-\frac{(1-\overline{\delta})(2q^2-q-1)d}{4q(2q-1)}},t_{m+1-l}^{-2(1-\overline{\delta})}\right\}
\\
&
\cdot
\big(1+(1+(h^\rho+\tau^\theta)^{-\alpha})\|X_{\tau,h}^l\|_{L^p(\Omega;H)}\big)
+
C(Q,p)
\\
\leq
&
C\sup_{m\in \mathbb{N}_0}\tau\sum_{l=0}^m
 \min\left\{t_{m+1-l}^{-\overline{\delta}-\frac{(1-\overline{\delta})(2q^2-q-1)d}{4q(2q-1)}},t_{m+1-l}^{-2}\right\}
\sup_{m\in \mathbb{N}^{+}}\|X_{\tau,h}^m\|_{L^p(\Omega;H)}
+
C(Q,p)
<\infty,
\end{split}
\end{align}
where in the last inequality, we used the fact
\[
\sup_{m\in \mathbb{N}_0 }\tau\sum_{l=0}^m
\min
\Big\{
t_{m+1-l}^{-\overline{\delta}-\frac{(1-\overline{\delta})(2q^2-q-1)d}{4q(2q-1)}},t_{m+1-l}^{-2}
\Big\}
<
\infty,
\]
as $\overline{\delta}+\frac{(1-\overline{\delta})(2q^2-q-1)d}{4q(2q-1)}<1$.
Following similar arguments as in the proof of \cite[(4.25)]{jiang2025uniform},  we deduce
\begin{align}
\sup_{m\in \mathbb{N}_0}\|\widetilde{Y}_{\tau,h}^m\|_{L^p(\Omega;L^{4q-2}(D))}
<\infty,
\end{align}
which together with \eqref{eq:bound-R} shows
\begin{align}
\sup_{m\in \mathbb{N}_0}\|X_{\tau,h}^m\|_{L^p(\Omega;L^{4q-2}(D))}
\leq
\sup_{m\in \mathbb{N}_0}\|\widetilde{Y}_{\tau,h}^m\|_{L^p(\Omega;L^{4q-2}(D))}
+
\sup_{m\in \mathbb{N}_0}\|\mathcal{R}_{\tau,h}^m\|_{L^p(\Omega;L^{4q-2}(D))}
<\infty.
\end{align}
This guarantees
\begin{align}\label{eq:bound-Fth-H}
\sup_{m\in \mathbb{N}_0}\|F_{\tau,h}(X_{\tau,h}^m)\|_{L^p(\Omega;H)}
\leq
\sup_{m\in \mathbb{N}_0}\|F(X_{\tau,h}^m)\|_{L^p(\Omega;H)}<\infty.
\end{align}
To show \eqref{th:spatal-regu-full-solu} and \eqref{th:temporal-regu-full-solu}, we consider two cases: $\gamma\in\big[\frac{(2q^2-q-2)d}{2q(2q-1)},\kappa\big]$ and $\gamma \in  (\kappa, 2]$,
where $\kappa\in (\frac d2,2)$.
For the former case $\gamma\in\big[\frac{(2q^2-q-2)d}{2q(2q-1)},\kappa\big]$, we have, by \eqref{eq:Ah-A-bound},
\eqref{lem:eq-spatial-regu-ENm}, \eqref{lem:bound-discrete-stochastic-convolution}
and \eqref{eq:bound-Fth-H}
\begin{align}\label{eq:spyia-bound-Xh}
\begin{split}
\sup_{m\in \mathbb{N}_0}\|A_h^{\frac\gamma2}X_{\tau,h}^m\|_{L^p(\Omega;H)}
\leq
&
\sup_{m\in \mathbb{N}_0}\left\|A_h^{\frac\gamma2}\Big(E_{\tau,h}^mP_hX_0
+
\tau\sum_{l=0}^{m-1}E_{\tau,h}^{m-l}P_hF_{\tau,h}(X_{\tau,h}^l)+W_{\tau,h}^m\Big)\right\|_{L^p(\Omega;H)}
\\
\leq
&
C(\|X_0\|_{L^p(\Omega;H^\gamma)}+C \sup_{m\in \mathbb{N}_0}\|F_{\tau,h}(X_{\tau,h}^m)\|_{L^p(\Omega;H)}\tau \sup_{m\in \mathbb{N}_0}\sum_{l=0}^{m-1}\min\{t_{m-l}^{-\frac\gamma2}, t_{m-l}^{-2}\}
\\
&
+
 \sup_{m\in \mathbb{N}_0}\|A_h^{\frac\gamma2}W_{\tau,h}^m\|_{L^p(\Omega;H)}
<\infty,
\end{split}
\end{align}
and by \eqref{eq:Ah-A-bound}, \eqref{lem:eq-spatial-regu-ENm}, \eqref{lem:eq-spatial-regu-ENm-sum-II}, \eqref{lem:eq-temporal-regu-ENm}, the Burkholder-Davis-Gundy type inequality
\begin{align}\label{eq:temporal-Xh}
\begin{split}
&\|A_h^{\frac\beta2}(X_{\tau,h}^m-X_{\tau,h}^n)\|_{L^p(\Omega;H)}
\\
\leq
&
\|A_h^{\frac\beta2}(I-E_{\tau,h}^{m-n})X_{\tau,h}^n\|_{L^p(\Omega;H)}
+
\sum_{i=n}^{m-1}\tau\|A_h^{\frac\beta2}E_{\tau,h}^{m-i}P_hF_{\tau,h}(X_{\tau,h}^{i-1})\|_{L^p(\Omega;H)}
\\
&+
\Big(\sum_{i=n}^{m-1}\tau\|A_h^{\frac\beta2}E_{\tau,h}^{m-i}P_hQ^{\frac12}\|_{\mathcal{L}_2(H)}^2\Big)^{\frac12}
\\
\leq
&
Ct_{m-n}^{\frac{\gamma-\beta}2}\|A_h^{\frac\gamma2}X_{\tau,h}^n\|_{L^2(\Omega;H)}
+
C\tau \sum_{i=n}^{m-1}\min\{t_{m-i}^{-\frac\beta2}, t^{-2}_{m-i}\}\sup_{m\in \mathbb{N}_0}\|F(X_{\tau,h}^m)\|_{L^p(\Omega;H)}
\\
&
+
Ct_{m-n}^{\frac{\min\{\gamma-\beta,1\}}2} \|A^{\frac{\gamma-1}2}Q^{\frac12}\|_{\mathcal{L}_2(H)}
\\
\leq
&
Ct_{m-n}^{\frac{\min\{\gamma-\beta,1\}}2},
\end{split}
\end{align}
for $\beta\in[0,\gamma]$, where in the last inequality we used the following estimate, for $\gamma\in \big[\frac{(2q^2-q-2)d}{2q(2q-1)},\kappa\big]$
\begin{align}
\tau \sum_{i=n}^{m-1}\min\{t_{m-i}^{-\frac\beta2}, t^{-2}_{m-i}\}
\leq
Ct_{m-n}^{\frac{\gamma-\beta}2}\tau \sum_{i=n}^{m-1}\min\left\{t_{m-i}^{-\frac\gamma2}, t^{-2-\frac{\gamma-\beta}2}_{m-i}\right\}
\leq
Ct_{m-n}^{\frac{\gamma-\beta}2}.
\end{align}
%
Next, let us focus on the other case $\gamma\in (\kappa,2]$. In this case, one can observe
\[\sup_{m\in \mathbb{N}_0}\|X^m_{\tau,h}\|_{L^{p}(\Omega;L^\infty(D))}
\leq
C\sup_{m\in \mathbb{N}_0}\|A_h^{\frac {\kappa}2}X^m_{\tau,h}\|_{L^{p}(\Omega;H)}<\infty,\]
 as already verified in the former case. Thus, employing \eqref{eq:V-norm-control-by-Ah-norm} and \eqref{eq:temporal-Xh} yields
 \begin{align}
 \begin{split}
 &\|P_h\big(F_{\tau,h}(X_{\tau,h}^{m-1})-F_{\tau,h}(X_{\tau,h}^{l-1})\big)\|_{L^p(\Omega;H)}
 \\
 \leq
 &
 \left\|
 \left(1+\|X_{\tau,h}^{m-1}\|_{L^\infty(D)}^{2q-2}+\|X_{\tau,h}^{l-1}\|_{L^\infty(D)}^{2q-2})\right)
 \|X_{\tau,h}^{m-1}-X_{\tau,h}^{l-1}\|\right\|_{L^p(\Omega;\mathbb{R})}
 \\
 \leq
 &
 C(t_m-t_l)^{\frac{\min\{\kappa,1\}}2}.
 \end{split}
 \end{align}
 This in combination with \eqref{lem:eq-spatial-regu-ENm}, \eqref{lem:eq-spatial-regu-ENm-sum} and \eqref{eq:bound-Fth-H}   implies, for $\beta\in (\kappa,2]$
\begin{align}
\begin{split}
&\left\|A_h^{\frac\beta2}\tau\sum_{l=n}^{m-1}E_{\tau,h}^{m-l}P_hF_{\tau,h}(X_{\tau,h}^{l-1})\right\|_{L^p(\Omega;H)}
\\
\leq
&
\left\|A_h^{\frac\beta2}\tau\sum_{l=n}^{m-1}E_{\tau,h}^{m-l}P_hF_{\tau,h}(X_{\tau,h}^{m-1})\right\|_{L^p(\Omega;H)}
\\
&
+
\tau\sum_{l=n}^{m-1}\|A_h^{\frac\beta2}E_{\tau,h}^{m-l}\|_{\mathcal{L}(H)}
\|F_{\tau,h}(X_{\tau,h}^{m-1})-F_{\tau,h}(X_{\tau,h}^{l-1})\|_{L^p(\Omega;H)}
\\
\leq
&
C(t_m-t_n)^{\frac{2-\beta}2} \sup_{m\in \mathbb{N}_0}\|F_{\tau,h}(X_{\tau,h}^l)\|_{L^p(\Omega;H)}
+
C\tau\sum_{l=n}^{m-1}\min\{t_{m-l}^{-\frac\beta2},t^{-2}_{m-l}\} (t_m-t_l)^{\frac{\min\{\kappa,1\}}2}
\\
\leq
&
C(t_m-t_n)^{\frac{2-\beta}2}
+
C(t_m-t_n)^{\frac{2-\beta}2}
\tau\sum_{l=n}^{m-1}\min\left\{t_{m-l}^{-1+\frac{\min\{\kappa,1\}}2},t^{-2-\frac{2-\beta}2+\frac{\min\{\kappa,1\}}2}_{m-l}\right\}
\\
\leq
&
C(t_m-t_n)^{\frac{2-\beta}2}.
\end{split}
\end{align}
Along the lines of the proof \eqref{eq:spyia-bound-Xh} and \eqref{eq:temporal-Xh}, we can show \eqref{th:spatal-regu-full-solu} and \eqref{th:temporal-regu-full-solu}.

Based on the estimates \eqref{th:spatal-regu-full-solu},  \eqref{th:temporal-regu-full-solu}
and \eqref{lemeq:V-bound-discrete-stoch-conv}£¬ we can follow the similar arguments of the proof of \eqref{eq:spyia-bound-Xh} to prove \eqref{them:bound-numerical-V}.
Hence, the proof of this theorem is complete. $\square $

As a direct consequence of Theorem \ref{th:uniform-moment-bound}, the following results hold.
\begin{proposition}\label{pro:pro:bound-Xh-f}
Let all conditions in Theorem \ref{th:uniform-moment-bound} hold true.
For any $p\geq 1$, there exists a constant $C=C(X_0,Q,p,q)$ such that
\begin{align}\label{pro:bound-Xh-f}
\begin{split}
&\sup_{m\in \mathbb{N}_0}(
\|F(X_{\tau,h}^m)\|_{L^p(\Omega; H)}
+
\sup_{m\in \mathbb{N}_0}\|f'(X_{\tau,h}^m)\|_{L^p(\Omega; L^{2q}(D))}
\\
&
+
\|f''(X_{\tau,h}^m)\|_{L^p(\Omega; L^{2q}(D))
}
+
\|(X_{\tau,h}^m)^{\frac{2q-2}\alpha}F(X_{\tau,h}^m)\|_{L^p(\Omega; L^1(D))}
\leq C<\infty.
\end{split}
\end{align}
\end{proposition}

\section{Strong convergence analysis of the full discretization scheme}
This section is devoted to the strong convergence of the fully-discrete finite element method \eqref{eq:full-discretization}.
%

%

To do this,  we need to introduce two auxiliary processes. The first one is
\begin{align}
\widetilde{X}^m_{\tau,h}=E_{\tau,h}^mP_hX_0+\tau\sum_{l=0}^{m-1}E_{\tau,h}^{m-l}P_hF(X(t_{l+1}))+W_{\tau,h}^m.
\end{align}
In view of \eqref{lem:eq-spatial-regu-ENm}, \eqref{lem:bound-discrete-stochastic-convolution}, \eqref{th:spatal-regu-full-solu} and  \eqref{th:spatial-regu-exac-solution}, one can acquire
\begin{align}\label{eq:property-first-auxi}
\sup_{m\in \mathbb{N}}\|A_h^{\frac\gamma2}\widetilde{X}_{\tau,h}^m\|_{L^p(\Omega;H)}
<\infty.
\end{align}
Another one
 is to find $\overline{X}_{\tau,h}^m\in V_h$ such that
\begin{align}
\overline{X}_{\tau,h}^m-\overline{X}_{\tau,h}^{m-1}+\tau A_h\overline{X}_{\tau,h}^m=\tau P_hF(X_{\tau,h}^m)+P_h\Delta W^m,
\end{align}
whose solution can be reformulated as
\begin{align}\label{eq:second-auxi-problem}
\overline{X}_{\tau,h}^m
=
E_{\tau,h}^mP_hX_0
+
\sum_{l=1}^mE_{\tau,h}^{m+1-l}P_hF(X_{\tau,h}^l)
+
W_{\tau,h}^m.
\end{align}
Under Assumptions \ref{assum:linear-operator-A}-\ref{assum:intial-value-data}, one can deduce
\begin{align}\label{eq:property-second}
\sup_{m\in \mathbb{N}_0}
\|A_h^{\frac\gamma2}\overline{X}_{\tau,h}^m\|_{L^p(\Omega,H)}
\leq
C(p,Q)<\infty.
\end{align}
With the above two auxiliary processes,
we  separate the considered error $\|X(t_m)-X_h^m\|_{L^p(\Omega;H)}$ as
\begin{align}
\label{eq:strong-error-decomposition}
\|X(t_m)-X_{\tau,h}^m\|_{L^p(\Omega;H)}
\leq
\|X(t_m)-\widetilde{X}_{\tau,h}^m\|_{L^p(\Omega;H)}
+
\|\widetilde{X}_{\tau,h}^m-\overline{X}_{\tau,h}^m\|_{L^p(\Omega;H)}
+
\|\overline{X}_{\tau,h}^m
-X_{\tau,h}^m\|_{L^p(\Omega;H)}.
\end{align}
In what follows, we will bound the above three errors, separately.
Following the same argument as in the proof of \cite[(4.37)]{QQQi2018Optimal}, one can bound the error
$\|X(t_m)-\widetilde{X}_{\tau,h}^m\|_{L^p(\Omega;H)}$ as follows.
\begin{lemma}\label{lem:ERROR-AUXI-problem}
Let all conditions in Theorem \ref{th:uniform-moment-bound} be fulfilled.
 Then, there
 exists a positive constant $C=C(p,Q)$ independent of $h$ and $\tau$  such that for  $p\geq1$
 \begin{align}\label{lemma:error-bound-auxi}
 \sup_{m\in \mathbb{N}_0}\|X(t_m)-\widetilde{X}_{\tau,h}^m\|_{L^p(\Omega;H)}
 \leq
 C(h^\gamma+\tau^{\frac\gamma2}).
 \end{align}
\end{lemma}
Next lemma will present the convergence result of the error $\|\overline{X}_{\tau,h}^m
-X_{\tau,h}^m\|_{L^p(\Omega;H)}$.
\begin{lemma}\label{lem:ERROR-AUXI-problem-ii}
Let all conditions in Theorem \ref{th:uniform-moment-bound} be fulfilled.
 Then, there
 exists a positive constant $C=C(p,Q)$ independent of $h$ and $\tau$  such that for  $p\geq1$
 \begin{align}\label{lemma:error-bound-auxi}
 \sup_{m\in \mathbb{N}_0}\|\overline{X}_{\tau,h}^m
-X_{\tau,h}^m\|_{L^p(\Omega;H)}
 \leq
 C (h^{\min\{\gamma,\rho\}}+\tau^{\min\{\theta,\frac\gamma2\}}).
 \end{align}
\end{lemma}
{\it Proof of Lemma \ref{lem:ERROR-AUXI-problem-ii}.}
Subtracting \eqref{eq:second-auxi-problem} from \eqref{eq:fully-discrete-problem}, the
error $\|\overline{X}_{\tau,h}^m-X_{\tau,h}^m\|_{L^p(\Omega,H)}$ can be split into the following two terms
\begin{align}\label{eq:SECOND-PRO-Decom}
\begin{split}
\|\overline{X}_{\tau,h}^m-X_{\tau,h}^m\|_{L^p(\Omega;H)}
\leq
&
\left\|\sum_{l=1}^{m}\tau E_{\tau,h}^{m+1-l}P_h (F(X_{\tau,h}^l)-F(X_{\tau,h}^{l-1}))\right\|_{L^p(\Omega;H)}
\\
&
+
\sum_{l=1}^{m}\tau \|E_{\tau,h}^{m+1-l}P_h (F(X_{\tau,h}^{l-1})-F_{\tau,h}(X_{\tau,h}^{l-1}))\|_{L^p(\Omega;H)}
\\
=:
&J_1+J_2.
\end{split}
\end{align}
By \eqref{lem:eq-spatial-regu-ENm} and for $\kappa\in (\frac d2, 2)$, \eqref{lem:f-ftau-I} and \eqref{pro:bound-Xh-f}, the second term $J_2$ can be estimated as follows:
\begin{align}\label{eq:J2-bound}
\begin{split}
J_2
\leq
&
\sum_{l=1}^{m}\tau \min\{t_{m+1-l}^{-\frac \kappa2}, t_{m+1-l}^{-2}\}\|A_h^{-\frac \kappa2}( F(X_{\tau,h}^{l-1})-F_{\tau,h}(X_{\tau,h}^{l-1}))\|_{L^p(\Omega;H)}
\\
\leq
&
\sum_{l=1}^{m}\tau \min\{t_{m+1-l}^{-\frac \kappa2}, t_{m+1-l}^{-2}\}\|F(X_{\tau,h}^{l-1})-F_{\tau,h}(X_{\tau,h}^{l-1})\|_{L^p(\Omega;L^1(D))}
\\
\leq
&
C(h^\rho+\tau^\theta)
\sum_{l=1}^{m}\tau \min\{t_{m+1-l}^{-\frac \kappa2}, t_{m+1-l}^{-2}\}
\sup_{m\in \mathbb{N}_0}\|(X_{\tau,h}^m)^{\frac{2q-2}\alpha}F(X_{\tau,h}^m)\|_{L^{p}(\Omega;L^1(D))}
\\
\leq
&
C(h^\rho+\tau^\theta).
\end{split}
\end{align}
To bound the first term $J_1$, we apply Taylor's formula to decompose it further into four additional terms as follows
\begin{align}\label{eq:bound-J11-term}
\begin{split}
J_1
\leq
&
\left\|\sum_{l=1}^{m}\tau E_{\tau,h}^{m+1-l}P_h F'(X_{\tau,h}^{l-1})(I-E_{\tau,h})X_{\tau,h}^{l-1}\right\|_{L^p(\Omega;H)}
\\
&
+
\left\|\sum_{l=1}^{m}\tau^2 E_{\tau,h}^{m+1-l}P_h F'(X_{\tau,h}^{l-1})E_{\tau,h}P_hF_{\tau,h}(X_{\tau,h}^{l-1})\right\|_{L^p(\Omega;H)}
\\
&
+
\left\|\sum_{l=1}^{m}\tau E_{\tau,h}^{m+1-l}P_h F'(X_{\tau,h}^{l-1})E_{\tau,h}P_h\Delta W^l\right\|_{L^p(\Omega;H)}
\\
&
+
\left\|\sum_{l=1}^{m}\tau E_{\tau,h}^{m+1-l}P_h R_F(X_{\tau,h}^l,X_{\tau,h}^{l-1})\right\|_{L^p(\Omega;H)}
\\
=:
&
J_{11}+
J_{12}+J_{13}+J_{14},
\end{split}
\end{align}
where the remainder term $R_F$ reads,
\[
R_F(X_{\tau,h}^l,X_{\tau,h}^{l-1})
:
=
\int_0^1F''\big(X_{\tau,h}^{l-1}+\lambda(X_{\tau,h}^l-X_{\tau,h}^{l-1})\big)
\big(X_{\tau,h}^l-X_{\tau,h}^{l-1},X_{\tau,h}^l-X_{\tau,h}^{l-1}\big)(1-\lambda)\,\dd \lambda.
\]
In the sequel we treat the above four terms one by one. Thanks to \eqref{eq:L_2-L_1}, \eqref{lem:eq-spatial-regu-ENm}, \eqref{lem:eq-temporal-regu-ENm}, \eqref{th:spatal-regu-full-solu} and \eqref{pro:bound-Xh-f}, we derive, for any fixed $\kappa\in (\frac d2,2)$,
\begin{align}\label{eq:bound-j11}
\begin{split}
J_{11}
\leq
&
C\sum_{l=1}^{m}\tau \min\{t_{m+1-l}^{-\frac\kappa2},t_{m+1-l}^{-2}\}
\|A_h^{-\frac \kappa2}P_h F'(X_{\tau,h}^{l-1})(I-E_{\tau,h})X_{\tau,h}^{l-1}\|_{L^p(\Omega;H)}
\\
\leq
&
C\sum_{l=1}^{m}\tau \min\{t_{m+1-l}^{-\frac\kappa2},t_{m+1-l}^{-2}\}
\| F'(X_{\tau,h}^{l-1})(I-E_{\tau,h})X_{\tau,h}^{l-1}\|_{L^p(\Omega;L^1(D))}
\\
\leq
&
C\sum_{l=1}^{m}\tau \min\{t_{m+1-l}^{-\frac\kappa2},t_{m+1-l}^{-2}\}
\| f'(X_{\tau,h}^{l-1})\|_{L^{2p}(\Omega;H)}
\|(I-E_{\tau,h})X_{\tau,h}^{l-1}\|_{L^{2p}(\Omega;H)}
\\
\leq
&
C\sum_{l=1}^{m}\tau^{1+\frac\gamma2} \min\{t_{m+1-l}^{-\frac\kappa2},t_{m+1-l}^{-2}\}
\sup_{m\in \mathbb{N}_0}\| f'(X_{\tau,h}^m)\|_{L^{2p}(\Omega;H)}
\sup_{m\in \mathbb{N}_0}\|A_h^{\frac\gamma2}X_{\tau,h}^m\|_{L^{2p}(\Omega;H)}
\\
\leq
&
C\tau^{\frac\gamma2},
\end{split}
\end{align}
where in the last inequality we used the fact,
\begin{align}\label{eq:sum}
\sup_{m\in \mathbb{N}_0}\sum_{l=1}^{m}\tau \min\{t_{m+1-l}^{-\frac\kappa2},t_{m+1}^{-2}\}<\infty,
\quad
\kappa\in(\tfrac d2,2).
\end{align}
Similarly,
 using \eqref{eq:L_2-L_1}, \eqref{lem:eq-spatial-regu-ENm},  \eqref{eq:sum} and \eqref{pro:bound-Xh-f}  implies that, for any fixed $\kappa\in(\frac d2,2)$
\begin{align}\label{eq:bound-j12}
\begin{split}
J_{12}
\leq&
C\sum_{l=1}^{m}\tau^2 \min\{t_{m+1-l}^{-\frac \kappa2},t_{m+1-l}^{-2}\}\|P_h F'(X_{\tau,h}^{l-1})E_{\tau,h}P_hF_{\tau,h}(X_{\tau,h}^{l-1})\|_{L^p(\Omega;L^1(D))}
\\
\leq
&
C\sum_{l=1}^{m}\tau^2 \min\{t_{m+1-l}^{-\frac \kappa2},t_{m+1-l}^{-2}\}\sup_{m\in \mathbb{N}_0}\| f'(X_{\tau,h}^{m})\|_{L^{2p}(\Omega;H)}
\sup_{m\in \mathbb{N}_0}\|F(X_{\tau,h}^{m})\|_{L^{2p}(\Omega;H)}
\\
\leq
&
C\tau.
\end{split}
\end{align}
 Owing to  the stochastic Fubini theorem (e.g., see \cite[Theorem 4.18]{da2014stochastic}) and the
Burkholder-Davis-Gundy-type inequality, we  obtain, for the term $J_{13}$,
\begin{align}
\begin{split}
J_{13}
=&
\left\|\sum_{l=1}^{m}\int_{t_{l-1}}^{t_l}\int_{t_{l-1}}^{t_l} E_{\tau,h}^{m+1-l}P_h F'( X_{\tau,h}^{l-1})E_{\tau,h}P_h\,\dd W(\sigma)\,\dd s\right\|_{L^p(\Omega;H)}
\\
=
&
\left\|\sum_{l=1}^{m}\int_{t_{l-1}}^{t_l}\int_{t_{l-1}}^{t_l} E_{\tau,h}^{m+1-l}P_h F'(X_{\tau,h}^{l-1})E_{\tau,h}P_h\,\dd s\,\dd W(\sigma) \right\|_{L^p(\Omega;H)}
\\
\leq
&
C\left(\sum_{l=1}^{m}\int_{t_{l-1}}^{t_l}\left\|\int_{t_{l-1}}^{t_l} E_{\tau,h}^{m+1-l}P_h F'(X_{\tau,h}^{l-1})E_{\tau,h}P_h\,\dd s\right\|^2_{L^p(\Omega;\mathcal{L}^0_2)}\,\dd \sigma\right)^{\frac12}
\\
\leq
&
C\left(\sum_{l=1}^{m}\tau^3\min\{1,t_{m+1-l}^{-4}\}\left\|  F'(X_{\tau,h}^{l-1})E_{\tau,h}P_h\right\|^2_{L^p(\Omega;\mathcal{L}^0_2)}\, \right)^{\frac12}.
\end{split}
\end{align}
For $\gamma\in (0,\frac d2]$, from  \eqref{pro:bound-Xh-f}, \eqref{eq:relation-Lp-Ah-relatoin}, \eqref{lem:eq-spatial-regu-ENm}
and  the H\"{o}lder inequality, one can deduce
\begin{align}
\begin{split}
J_{13}
\leq
&
\left(\sum_{l=1}^{m}\tau^3\min\left\{1,t_{m+1-l}^{-4}\right\}\|f'(X_{\tau,h}^{l-1})\|^2_{L^p(\Omega;L^4)}\| A_h^ {\frac d8}E_{\tau,h}P_h\|^2_{\mathcal{L}^0_2}\right)^{\frac12}
\\
\leq&
\left(\sum_{l=1}^{m}\tau^{3-\frac{4-4\gamma+d}4}\min\left\{1,t_{m+1-l}^{-4}\right\}
\|f'(X_{\tau,h}^{l-1})\|^2_{L^p(\Omega;L^4(D))}
\|A_h^{\frac{\gamma-1}2}P_h\|^2_{\mathcal{L}_2^0}\right)^{\frac12}
\\
\leq
&
C\tau^{\frac{4-d+4\gamma}8} \sup_{m \in \mathbb{N}_0}\|f'(X_{\tau,h}^{m})\|_{L^p(\Omega;L^4(D))}
\leq
 C\tau^{\frac\gamma2}.
\end{split}
\end{align}
For $\gamma\in (\frac d2,2]$, applying \eqref{eq:V-norm-control-by-Ah-norm}
 and \eqref{th:spatal-regu-full-solu} implies
$
\sup_{m\in \mathbb{N}_0}\|f'(X_{\tau,h}^m)\|_{L^p(\Omega;L^\infty(D))}<\infty$.
Further, we have, by  \eqref{lem:eq-spatial-regu-ENm}, \eqref{eq:relation-A-Ah} and Assumption \ref{assum:noise-term}
\begin{align}
\begin{split}
J_{13}
\leq
&
\left(\sum_{l=1}^{m}\tau^3\min\{1,t_{m+1-l}^{-4}\}\|f'(X_{\tau,h}^{l-1})\|^2_{L^p(\Omega;L^\infty(D))}\| E_{\tau,h}P_h\|^2_{\mathcal{L}_2^0}\right)^{\frac12}
\\
\leq&
\left(\sum_{l=1}^{m}\tau^{3-\max\{0,1-\gamma\}}\min\{1,t_{m+1-l}^{-4}\}
\sup_{m\in \mathbb{N}}\|f'(X_{\tau,h}^{m})\|^2_{L^p(\Omega;L^\infty(D))}
\|A_h^{\frac{\gamma-1}2}P_h\|^2_{\mathcal{L}_2^0}\right)^{\frac12}
\\
\leq
&
C\tau^{\frac\gamma2}\|A^{\frac{\gamma-1}2}Q^{\frac12}\|_{\mathcal{L}_2(H)}.
\end{split}
\end{align}
Thus,  the above three estimates ensure
\begin{align}\label{eq:bound-j13}
J_{13}\leq C\tau^{\frac\gamma2}.
\end{align}
Now we turn our attention to the term $J_{14}$.
Thanks to \eqref{eq:sum}, \eqref{eq:L_2-L_1} and \eqref{lem:eq-spatial-regu-ENm}, we derive
\begin{align}\label{eq:-BOUND-J14-ii}
\begin{split}
J_{14}
\leq
&
C\sum_{l=1}^{m}\tau \min\{t^{-\frac\kappa2},t^{-2}_{m+1-l}\}\| A_h^{-\frac\kappa 2}P_h R_F(X_{\tau,h}^l,X_{\tau,h}^{l-1})\|_{L^p(\Omega;H)}
\\
\leq
&
C\sum_{l=1}^{m}\tau \min\{t^{-\frac\kappa2},t^{-2}_{m+1-l}\}\|  R_F(X_{\tau,h}^l,X_{\tau,h}^{l-1})\|_{L^p(\Omega;L^1(D))}.
\end{split}
\end{align}
For $\gamma\in \big(\tfrac{(2q^2-q-2)d}{2q(2q-1)},\tfrac d 2\big]$, applying   \eqref{pro:bound-Xh-f}, \eqref{eq:sum}, \eqref{th:temporal-regu-full-solu}, \eqref{eq:relation-Lp-Ah-relatoin}
and  the H\"{o}lder inequality implies
\begin{align}\label{eq:bound-J14-I}
\begin{split}
J_{14}
\leq
&
C
\sum_{l=1}^{m}\tau \min\{t_{m+1-l}^{-\frac\kappa2},t^{-2}_{m+1-l}\}\int_0^1\| f''(X_{\tau,h}^{l-1}+\lambda(X_{\tau,h}^l-X_{\tau,h}^{l-1}))\|_{L^{2p}(\Omega;L^{2q})}\,\dd \lambda \|X_{\tau,h}^{l}-X_{\tau,h}^{l-1}\|^2_{L^{4p}(\Omega;L^{\frac{4q}{2q-1}})}
\\
\leq
&
C
(1+\sup_{m\in \mathbb{N}_0}\| X_{\tau,h}^{m}\|_{L^{2p(2q-2)}(\Omega;L^{2q(2q-2)}(D))})^{2q-2}
\sum_{l=1}^{m}\tau \min\{t_{m+1-l}^{-\frac\kappa2},t^{-2}_{m+1-l}\} \|A_h^{\frac{d}{8q}}(X_{\tau,h}^{l}-X_{\tau,h}^{l-1})\|^2_{L^{4p}(\Omega;H)}
\\
\leq
&
C\tau^{\min\{\gamma-\frac{d}{4q},1\}}
\leq
C\tau^{\min\{\frac\gamma2+\frac\gamma2-\frac{d}{4q},1\}}
\leq
C\tau^{\min\{\frac\gamma 2,1\}}=C\tau^{\frac\gamma2},
\end{split}
\end{align}
where in the last inequality we used the fact $\frac{\gamma}2-\frac d{4q}>0$.
For $\gamma\in(\frac d 2,  2]$,  by applying \eqref{eq:V-norm-control-by-Ah-norm}, \eqref{th:spatal-regu-full-solu},
\eqref{eq:sum} and \eqref{th:temporal-regu-full-solu}, one can observe
\begin{align}
\begin{split}
J_{14}
\leq
&
C
\sum_{l=1}^{m}\tau \min\{t_{m+1-l}^{-\frac\kappa2},t^{-2}_{m+1-l}\}\int_0^1\| f''(X_{\tau,h}^{l-1}+\lambda(X_{\tau,h}^l-X_{\tau,h}^{l-1}))\|_{L^{2p}(\Omega;L^\infty(D))}\,\dd \lambda \|X_{\tau,h}^{l}-X_{\tau,h}^{l-1}\|^2_{L^{4p}(\Omega;H)}
\\
\leq
&
C\tau^{\min\{\gamma,1\}}\sup_{m\in \mathbb{N}_0}\sum_{l=1}^{m}\tau \min\{t_{m+1-l}^{-\frac\kappa2},t^{-2}_{m+1-l}\}
(1+\sup_{m\in \mathbb{N}_0}\| X_{\tau,h}^{m}\|_{L^{2p(2q-2)}(\Omega;L^\infty(D))})^{2q-2}
\\
\leq
&
C\tau^{\min\{\gamma,1\}},
\end{split}
\end{align}
which in combination with \eqref{eq:bound-J14-I} and \eqref{eq:-BOUND-J14-ii} arrives at
\begin{align}\label{eq:bound-j14}
J_{14}
\leq
C\tau^{\frac\gamma2}.
\end{align}
Putting the estimates \eqref{eq:bound-j11}, \eqref{eq:bound-j12}, \eqref{eq:bound-j13} and \eqref{eq:bound-j14} back into \eqref{eq:bound-J11-term}  results in
\begin{align}
J_1
\leq
C\tau^{\frac\gamma2},
\end{align}
which together with \eqref{eq:J2-bound} and \eqref{eq:SECOND-PRO-Decom} shows  \eqref{lemma:error-bound-auxi}
and ends the proof of this lemma. $\square$

Finally,  we turn our attention to  the error
$\|\widetilde{X}_{\tau,h}^m-\overline{X}_{\tau,h}^m\|_{L^p(\Omega;H)}$.
\begin{lemma}\label{them:bound-two-auxili}
Let all conditions in Theorem \ref{th:uniform-moment-bound} be satisfied.
 For any fixed time-point $T\in (0,\infty)$, let $\tau=T/M$, $M\in \mathbb{N}$ be a time-step size. Then
  there
 exist  positive constants $\tau^*, h^*$ and $C=C(T,p,Q)$ such that for all
 $\tau \leq \tau^*$, $h\leq h^*$
 and $\forall\, p\geq 1$,
\begin{align}\label{lem:eq-bound-solution-full-stochatic-I}
\sup_{m\in \{1,2,\dots,  M\}}
\|\widetilde{X}_{\tau,h}^m-\overline{X}_{\tau,h}^m\|_{L^p(\Omega;H)}
\leq
C(h^{\min\{\gamma, \rho\}}+\tau^{\min\{\frac\gamma2,\theta\}}).
\end{align}
If additionally $L_f<\lambda_1$, then one can show the uniform-in-time strong error estimates: for any $\tau\in (0, \tau^*)$,
\begin{align}\label{lem:eq-uniform-in-time-bound-solution-full-stochatic-ii}
\sup_{m\in \mathbb{N}}
\|\widetilde{X}_{\tau,h}^m-\overline{X}_{\tau,h}^m\|_{L^p(\Omega;H)}
\leq
C(h^{\min\{\gamma, \rho\}}+\tau^{\min\{\frac\gamma2,\theta\}}).
\end{align}
\end{lemma}
{\it Proof of Lemma \ref{them:bound-two-auxili}.}
Note first that
 $\widetilde{X}_{\tau,h}^m-\overline{X}_{\tau,h}^m$  obeys
\begin{align}
\widetilde{X}_{\tau,h}^m-\overline{X}_{\tau,h}^m-(\widetilde{X}_{\tau,h}^{m-1}-\overline{X}_{\tau,h}^{m-1})
+
\tau A_h(\widetilde{X}_{\tau,h}^m-\overline{X}_{\tau,h}^m)
=
\tau P_h \big(F(X(t_m))-F(X_{\tau,h}^{m})\big).
\end{align}
Multiplying this equation by $\widetilde{X}_{\tau,h}^m-\overline{X}_{\tau,h}^m$ and using the fact
$(A_h v_h,v_h )=\|\nabla v_h\|^2$, $\forall v_h\in V_h$ yield
\begin{align}\label{eq:inner-product-e}
\begin{split}
&\frac12(\|\widetilde{X}_{\tau,h}^m-\overline{X}_{\tau,h}^m\|^2
-
\|\widetilde{X}_{\tau,h}^{m-1}-\overline{X}_{\tau,h}^{m-1}\|^2)
+
\tau \|\widetilde{X}_{\tau,h}^m-\overline{X}_{\tau,h}^m\|_1^2
\\
\leq
&
\left<\widetilde{X}_{\tau,h}^m-\overline{X}_{\tau,h}^m
-
(\widetilde{X}_{\tau,h}^{m-1}-\overline{X}_{\tau,h}^{m-1}),\widetilde{X}_{\tau,h}^m-\overline{X}_{\tau,h}^m\right>
+
\tau\left<\nabla (\widetilde{X}_{\tau,h}^m-\overline{X}_{\tau,h}^m),
\nabla (\widetilde{X}_{\tau,h}^m-\overline{X}_{\tau,h}^m)\right>
\\
=
&
\tau\left<F\big(X(t_m)\big)-F\big(\widetilde{X}_{\tau,h}^m\big),\widetilde{X}_{\tau,h}^m-\overline{X}_{\tau,h}^m\right>
+
\tau\left<F(\widetilde{X}_{\tau,h}^m)-F(\overline{X}_{\tau,h}^m),\widetilde{X}_{\tau,h}^m-\overline{X}_{\tau,h}^m\right>
\\
&
+
\tau \left<F(\overline{X}_{\tau,h}^m)-F(X_{\tau,h}^{m}),\widetilde{X}_{\tau,h}^m-\overline{X}_{\tau,h}^m\right>.
\end{split}
\end{align}
Due to Assumption \ref{assum:nonlinearity} and the inequality $ab\leq \frac1{2\epsilon}a^2+\frac\epsilon2b^2$, for any $\epsilon>0$
\begin{align}\label{eq:error-equation-iner}
\begin{split}
&
\frac12(\|\widetilde{X}_{\tau,h}^m-\overline{X}_{\tau,h}^m\|^2-\|\widetilde{X}_{\tau,h}^{m-1}-\overline{X}_{\tau,h}^{m-1}\|^2)
+
\tau \|\widetilde{X}_{\tau,h}^m-\overline{X}_{\tau,h}^m\|_1^2
\\
\leq
&
L_f\tau \|\widetilde{X}_{\tau,h}^m-\overline{X}_{\tau,h}^m\|^2
+
\frac{\tau}{2\epsilon}\|A_h^{-\frac12}P_h\big(F(X(t_m))-F(\widetilde{X}_{\tau,h}^m)
+
F(\overline{X}_{\tau,h}^m)-F(X_{\tau,h}^m)\big)\|^2
\\
&
+
\frac{\tau\epsilon}{2} \| \widetilde{X}_{\tau,h}^m-\overline{X}_{\tau,h}^m\|_1^2.
\end{split}
\end{align}
Summation on $m$, taking $\epsilon=1$ and employing the fact $\widetilde{X}_{\tau,h}^0-\overline{X}_{\tau,h}^0=0$ show
\begin{align}
\begin{split}
\|\widetilde{X}_{\tau,h}^m-\overline{X}_{\tau,h}^m\|^2
\leq
&
2L_f\tau \sum_{l=1}^m \|\widetilde{X}_{\tau,h}^l-\overline{X}_{\tau,h}^l\|^2
\\
&
+
\tau\sum_{l=1}^m \|A_h^{-\frac12}P_h\big(F(X(t_l))-F(\widetilde{X}_{\tau,h}^l)
+
F(\overline{X}_{\tau,h}^l)-F(X_{\tau,h}^l)\big)\|^2.
\end{split}
\end{align}
Further, by Gronwall's inequality and letting $\tau< \frac1{2 L_f}$ we arrive at
\begin{align}
\|\widetilde{X}_{\tau,h}^m-\overline{X}_{\tau,h}^m\|^2
\leq
C(T)\tau\sum_{l=1}^m \|A_h^{-\frac12}P_h\big(F(X(t_l))-F(\widetilde{X}_{\tau,h}^l)
+
F(\overline{X}_{\tau,h}^l)-F(X_{\tau,h}^l)\big)\|^2.
\end{align}
Taking expectation and using the H\"{o}lder inequality lead to
\begin{align}
\mathbb{E}[\|\widetilde{X}_{\tau,h}^m-\overline{X}_{\tau,h}^m\|^{2p}]
\leq
C(T)\tau\sum_{l=1}^m \mathbb{E}[\|A_h^{-\frac12}P_h\big(F(X(t_l))-F(\widetilde{X}_{\tau,h}^l)
+
F(\overline{X}_{\tau,h}^l)-F(X_{\tau,h}^l)\big)\|^{2p}].
\end{align}
In view of \eqref{th:spatal-regu-full-solu}, \eqref{th:spatial-regu-exac-solution}, \eqref{eq:property-first-auxi} and \eqref{lemma:error-bound-auxi}, we infer
\begin{align}\label{eq:bound-FX-FwidetildeX}
\begin{split}
&\|A_h^{-\frac12}P_h\big(F(X(t_m))-F(\widetilde{X}_{\tau,h}^m)
\big)\|_{L^{2p}(\Omega;H)}
\\
\leq
&
C\|F(X(t_m))-F(\widetilde{X}_{\tau,h}^m)
\|_{L^{2p}(\Omega;L^{\frac65}(D))}
\\
\leq
&
C
\left\|\|X(t_m)-\widetilde{X}_{\tau,h}^m\| (1+\|X(t_m)\|^{2q-2}_{L^{3(2q-2)}(D)}+\|\widetilde{X}_{\tau,h}^m\|^{2q-2}_{L^{3(2q-2)}(D)}
\big)\right\|_{L^{2p}(\Omega;\mathbb{R})}
\\
\leq
&
C(h^{\gamma}+\tau^{\frac\gamma2}).
\end{split}
\end{align}
Similarly as above,
\begin{align}\label{eq:bound-FX-FwidetildeX-ii}
\begin{split}
\|A_h^{-\frac12}P_h\big(F(\overline{X}_{\tau,h}^m))-F(X_{\tau,h}^m)
\big)\|_{L^{2p}(\Omega;H)}
\leq
C(h^{\min\{\gamma,\rho\}}+\tau^{\min\{\frac\gamma2,\theta\}}).
\end{split}
\end{align}
Gathering the above three estimates together results in
\begin{align}\label{eq:e-finite-time}
\|\widetilde{X}_{\tau,h}^m-\overline{X}_{\tau,h}^m\|_{L^{2p}(\Omega;H)}
\leq
C(T)(h^{\min\{\gamma,\rho\}}+\tau^{\min\{\frac\gamma2,\theta\}}),
\end{align}
which shows \eqref{lem:eq-bound-solution-full-stochatic}.

In the case $L_f<\lambda_1$, we take $\epsilon=\lambda_1-L_f$ in \eqref{eq:error-equation-iner}
and utilize the fact $\|\nabla v_h\|^2\geq \lambda_1\|v_h\|^2$ for any $v_h\in V_h$, to obtain
\begin{align}\label{eq:error-equation-iner-I}
\begin{split}
&
\tfrac12(\|\widetilde{X}_{\tau,h}^m-\overline{X}_{\tau,h}^m\|^2-\|\widetilde{X}_{\tau,h}^{m-1}-\overline{X}_{\tau,h}^{m-1}\|^2)
+
\tfrac{\lambda_1-L_f}2\tau \|\widetilde{X}_{\tau,h}^m-\overline{X}_{\tau,h}^m\|^2
\\
\leq
&
\tfrac{\tau}{2(\lambda_1-L_f)}\|A_h^{-\frac12}P_h\big(F(X(t_m))-F(\widetilde{X}_{\tau,h}^m)
+
F(\overline{X}_{\tau,h}^m)-F(X_{\tau,h}^m)\big)\|^2
.
\end{split}
\end{align}
Employing the estimate $(1+x)^{-1}\leq e^{-cx}, x\in[0,1]$  for some $c>0$ implies
\begin{align}
\begin{split}
&\|\widetilde{X}_{\tau,h}^m-\overline{X}_{\tau,h}^m\|^2
\leq
\tfrac 1{1+\tau(\lambda_1-L_f)}\|\widetilde{X}_{\tau,h}^{m-1}-\overline{X}_{\tau,h}^{m-1}\|^2
\\
&
+
\tfrac \tau{(1+\tau(\lambda_1-L_f))(\lambda_1-L_f)}\|A_h^{-\frac12}P_h\big(F(X(t_m))
-
F(\widetilde{X}_{\tau,h}^m)+F(\overline{X}_{\tau,h}^m)-F(X_{\tau,h}^m)\big)\|^2
\\
\leq
&
\tfrac{\tau}{\lambda_1-L_f} e^{-c \tau(\lambda_1-L_f)}\|A_h^{-\frac12}P_h\big(F(X(t_m))
-
F(\widetilde{X}_{\tau,h}^m)+F(\overline{X}_{\tau,h}^m)-F(X_{\tau,h}^m)\big)\|^2
\\
&
+
e^{-c \tau(\lambda_1-L_f)}\|\widetilde{X}_{\tau,h}^{m-1}-\overline{X}_{\tau,h}^{m-1}\|^2.
\end{split}
\end{align}
By summation over $m$, noting that $\widetilde{X}_{\tau,h}^0 - \overline{X}_{\tau,h}^0 = 0$ and denoting $\widetilde{c}:=c(\lambda_1-L_f)$, we arrive at
\begin{align}
\begin{split}
&\|\widetilde{X}_{\tau,h}^m-\overline{X}_{\tau,h}^m\|^2
\leq
e^{-\widetilde{c}\tau}\|\widetilde{X}_{\tau,h}^{m-1}-\overline{X}_{\tau,h}^{m-1}\|^2
\\
&
+
\tfrac {\tau e^{-\widetilde{c}\tau}}{\lambda_1-L_f}\|A_h^{-\frac12}P_h\big(F(X(t_m))
-
F(\widetilde{X}_{\tau,h}^m)+F(\overline{X}_{\tau,h}^m)-F(X_{\tau,h}^m)\big)\|^2
\\
\leq
&
\tfrac{\tau}{\lambda_1-L_f} \sum_{j=1}^m e^{-\widetilde{c} t_{m+1-j} }\|A_h^{-\frac12}P_h\big(F(X(t_j))-F(\widetilde{X}_{\tau,h}^j)+F(\overline{X}_{\tau,h}^j)-F(X_{\tau,h}^j)\big)\|^2.
\end{split}
\end{align}
Similarly as in \eqref{eq:e-finite-time}, we use \eqref{eq:bound-FX-FwidetildeX} and \eqref{eq:bound-FX-FwidetildeX-ii} to show
\begin{align}
\begin{split}
\mathbb{E}[\|\widetilde{X}_{\tau,h}^m-\overline{X}_{\tau,h}^m\|^{2p}]
\leq
&
C\tau\sum_{j=1}^m e^{-\widetilde{c} t_{m+1-j}}\mathbb{E}[\|A_h^{-\frac12}P_h\big(F(X(t_m))-F(\widetilde{X}_{\tau,h}^m)
+
F(\overline{X}_{\tau,h}^m)-F(X_{\tau,h}^m)\big)\|^{2p}]
\\
\leq
&
C\tau\sum_{l=1}^me^{-\widetilde{c} t_{m+1-j}}(h^{\min\{\gamma,\rho\}}+\tau^{\min\{\frac\gamma2,\theta\}})^{2p}
\\
\leq
&
C(h^{\min\{\gamma,\rho\}}+\tau^{\min\{\frac\gamma2,\theta\}})^{2p},
\end{split}
\end{align}
where in the first and last inequalities, we used the fact
\begin{align}
 \tau\sup_{m\in \mathbb{N}_0}\sum_{j=1}^m e^{-\widetilde{c} t_{m+1-j}}
 <\infty.
 \end{align}
This shows \eqref{lem:eq-uniform-in-time-bound-solution-full-stochatic-ii} and ends the proof of this lemma. $\square$

Armed with Lemmas \ref{lem:ERROR-AUXI-problem}-\ref{them:bound-two-auxili}, one can easily deduce strong convergence  of the fully-discrete finite element method \eqref{eq:full-discretization} as follows.
\begin{theorem}\label{them:bound-numerical-solution}
Let all conditions in Theorem \ref{th:uniform-moment-bound} be satisfied.
 For any fixed time-point $T\in (0,\infty)$, let $\tau=T/M$, $M\in \mathbb{N}$ be a time-step size. Then
  there
 exist  positive constants $\tau^*, h^*$ and $C=C(T,p,Q)$ such that for all
 $\tau \leq \tau^*$, $h\leq h^*$
 and $\forall\, p\geq 1$,
\begin{align}\label{lem:eq-bound-solution-full-stochatic}
\sup_{m\in \{1,2,\dots,  M\}}
\|X(t_m)-X_{\tau,h}^m\|_{L^p(\Omega;H)}
\leq
C(h^{\min\{\gamma, \rho\}}+\tau^{\min\{\frac\gamma2,\theta\}}).
\end{align}
If additionally $L_f<\lambda_1$, then one can show the uniform-in-time strong error estimates: for any $\tau\in (0, \tau^*)$,
\begin{align}\label{lem:eq-uniform-in-time-bound-solution-full-stochatic}
\sup_{m\in \mathbb{N}_0}
\|X(t_m)-X_{\tau,h}^m\|_{L^p(\Omega;H)}
\leq
C(h^{\min\{\gamma, \rho\}}+\tau^{\min\{\frac\gamma2,\theta\}}).
\end{align}
\end{theorem}

\section{Uniform-in-time weak convergence analysis}
This section aims to establish uniform-in-time weak error bounds of the proposed fully discrete schemes, which forces us to first derive Malliavin regularity estimates of the numerical approximations.
\subsection{Malliavin regularity estimates of fully discrete finite element approximations}
In this part, we  are devoted to the Malliavin regularity of the numerical solution, which plays an important role in controlling the stochastic integral error term in the weak convergence analysis.
%
Let us start with the introduction of Malliavin derivative. Let $U_0 := Q^{\frac12}(H)$. For any deterministic mapping $\Psi\in L^2\big([0,T]; \mathcal{L}_2(U_0,\mathbb{R})\big)$, let
 $M: L^2\big([0,T]; \mathcal{L}_2(U_0,\mathbb{R})\big)\rightarrow L^2(\Omega;\mathbb{R})$ be an isonormal process, given by
\begin{align}\label{eq:stochastic-integrand-Psi}
M(\Psi):=\int_0^T\Psi(t)\,\dd W(t).
\end{align}
 With this, we  define the family  of  smooth cylindrical random variables by,
\begin{align}
\mathcal{S}(H)
=
\Big\{G=\sum_{i=1}^ng_i\big(M(\Psi_1),\cdots,M(\Psi_m)\big)h_i, g_i\in C_p^\infty(\mathbb{R}^m,\mathbb{R}), h_i\in H, i=1,\cdots, n, n\in \mathbb{N}\Big\},
\end{align}
for $\Psi_j\in L^2\big([0,T]; \mathcal{L}_2(U_0,\mathbb{R})\big)$, $j=1, 2, 3\cdots, m$, $m\in \mathbb{N}$,
where $C_p^\infty(\mathbb{R}^m,\mathbb{R})$  represents  the space of all real-valued $C^\infty$-functions on $\mathbb{R}^m$ with polynomial growth.
For $G\in S(H)$, the  Malliavin derivative of $G$, at time $s\in[0,T]$,  is defined as
\begin{align}
D_sG
:=
\sum_{i=1}^n\sum_{j=1}^m \partial_jg_i\big(M(\Psi_1),\cdots,M(\Psi_m)\big)h_i\otimes \Psi_j(s).
\end{align}
Let $ \mathbb{D}^{1,2}(H)$
be the closure of the set of smooth random variables $\mathcal{S}(H)$ in the space $L^2(\Omega;H)$
with  respect to  the norm
\begin{align}
\|G\|_{\mathbb{D}^{1,2}(H)}
=
\Big(\mathbb{E}\big[\|G\|^2\big]
+
\mathbb{E}\int_0^T\|D_tG\|_{\mathcal{L}_2^0}^2\,\dd t\Big)^{\frac12}.
\end{align}
Then,  the Malliavin integration by parts formula is valid (see, e.g.,\cite[Lemma 2.1]{Daisuke2011Weak}), namely, for any $G\in \mathbb{D}^{1,2}(H)$ and adapted process $\Psi(t)\in L^2\big(\Omega; L^2\big([0,T],\mathcal{L}_2^0\big)$, it holds
\begin{align}\label{eq:integration-by-part-formula}
\mathbb{E}
\big<DG, \Psi\big>_{L^2([0,T], \mathcal{L}_2^0)}
=
\mathbb{E}\Big<G,\int_0^T\Psi(t)\,\dd W(t)\Big>_H.
\end{align}
Additionally, we define the process  $D_s^uG$ by $\big<D_sG,u\big>=D_s^uG$, which represents Malliavin derivative in the direction $u\in U_0$.
Moreover, the chain rule of the Malliavin derivative is valid. For another separable Hilbert space $\mathcal{H}$, if $\rho\in C_b^1(H,\mathcal{H})$, then $\rho(G)\in \mathbb{D}^{1,2}(\mathcal{H})$
and $D_t^\mu(\rho(G))=\rho'(G)\cdot D_t^uG$.

Now, we intend to establish the Malliavin regularity of the numerical approximations, which requires additional properties on $f$ and $f_{\tau,h}$. Under Assumption \ref{assum:nonlinearity}, one can easily deduce the following properties of $f$.
\begin{lemma}
Let Assumption \ref{assum:nonlinearity} be fulfilled. Then, there exist constants $c_6$ and $c_7, c_8, c_9>0$ such that, for all $x,y\in \mathbb{R}$
\begin{align}
f'(x)|x|-sign(x)(2q-2)f(x)
\leq
&
c_6-c_7|x|x^{2q-2},
\label{eq:assumption-f-derivx-2q-2f}
\\
|f'(x)|x|-sign(x)(2q-2)f(x)|
\leq
&
c_8+c_9|x|^{2q-1}.
\label{eq:assumption-f-derivx-2q-2f-II}
\end{align}
\end{lemma}

\begin{lemma}\label{assum:derivate-ftauh-bound}
Let Assumption \ref{assum:nonlinearity} be fulfilled  and $L_f<\lambda_1$.
Then, there exist two constants $\tau^*\in (0,\infty)$ and $h^*\in(0,\infty)$ such that $0<\tau\leq \tau^*$ and $0<h\leq h^*$,
the transformation $f_{\tau,h}$ defined in \eqref{eq:definiton-I-ftauh} satisfies the following condition:
\begin{align}\label{asum:condition-ftauh-derivate}
\begin{split}
2f_{\tau,h}'(x)y^2+\tau |f_{\tau,h}'(x)y|^2
\leq
(\lambda_1+L_f)y^2.
\end{split}
\end{align}
\end{lemma}
{\it Proof of Lemma \ref{assum:derivate-ftauh-bound}.}
It is easy to see
\begin{align}
\begin{split}
f_{\tau,h}'(x)
=
&
\tfrac{f'(x)}{\big(1+(\beta_1\tau^\theta+\beta_2h^\rho)|x|^{\frac{2q-2}\alpha})^\alpha}
-
sign(x)
\tfrac{(2q-2)f(x)(\beta_1\tau^\theta+\beta_2h^\rho)|x|^{\frac{2q-2}\alpha-1}}{\big(1+(\beta_1\tau^\theta+\beta_2h^\rho)
|x|^{\frac{2q-2}\alpha})^{\alpha+1}}
\\
=
&
\tfrac{f'(x)}{\big(1+(\beta_1\tau^\theta+\beta_2h^\rho)|x|^{\frac{2q-2}\alpha})^{\alpha+1}}
\\
&
+
\tfrac{\big(f'(x)|x|-(2q-2)sign(x)f(x)\big)(\beta_1\tau^\theta+\beta_2h^\rho)|x|^{\frac{2q-2}\alpha-1}}
{\big(1+(\beta_1\tau^\theta+\beta_2h^\rho)
|x|^{\frac{2q-2}\alpha})^{\alpha+1}}.
\end{split}
\end{align}
This together with  Assumption \ref{assum:nonlinearity}  implies
\begin{align}\label{eq:derive-fhtau}
\begin{split}
2f_{\tau,h}'(x)
\leq
&
2L_f
+
\tfrac{(2c_6-2c_7x^{2q-2}|x|)(\beta_1\tau^\theta+\beta_2h^\rho)|x|^{\frac{2q-2}\alpha-1}}{\big(1+(\beta_1\tau^\theta+\beta_2h^\rho)
|x|^{\frac{2q-2}\alpha})^{\alpha+1}}
\\
\leq
&
2L_f
+
\tfrac{2c_6(\beta_1\tau^\theta+\beta_2h^\rho)^{\frac{\alpha}{2q-2}}
((\beta_1\tau^\theta+\beta_2h^\rho)|x|^{\frac{2q-2}\alpha})^{\frac{2q-2-\alpha}{2q-2}}}
{\big(1+(\beta_1\tau^\theta+\beta_2h^\rho)
|x|^{\frac{2q-2}\alpha})^{\alpha+1}}
\\
&
-
\tfrac{2c_7x^{2q-2}(\beta_1\tau^\theta+\beta_2h^\rho)|x|^{\frac{2q-2}\alpha}}{\big(1+(\beta_1\tau^\theta+\beta_2h^\rho)
|x|^{\frac{2q-2}\alpha})^{\alpha+1}}
\\
\leq
&
2L_f
+
2c_6(\beta_1\tau^\theta+\beta_2h^\rho)^{\frac{\alpha}{2q-2}}
-
\tfrac{2c_7(\beta_1\tau^\theta+\beta_2h^\rho)|x|^{\frac{(2q-2)(\alpha+1)}\alpha}}{\big(1+(\beta_1\tau^\theta+\beta_2h^\rho)
|x|^{\frac{2q-2}\alpha})^{\alpha+1}}
.
\end{split}
\end{align}
Similarly as above,
\begin{align}
\begin{split}
\Big|\tfrac{\tau^{\frac12}f'(x)}{\big(1+(\beta_1\tau^\theta+\beta_2h^\rho)|x|^{\frac{2q-2}\alpha}\big)^{\alpha+1}}\Big|
\leq
&
\tfrac{R_f \tau^{\frac12}(1+|x|^{2q-2})}{\big(1+(\beta_1\tau^\theta+\beta_2h^\rho)|x|^{\frac{2q-2}\alpha}\big)^{\alpha+1}}
 \\
 \leq
 &
 R_f\tau^{\frac12}
 +
 \tfrac{R_f \tau^{\frac12}|x|^{2q-2}}
 {\big(1+(\beta_1\tau^\theta+\beta_2h^\rho)|x|^{\frac{2q-2}\alpha}\big)^{\alpha+1}},
\end{split}
\end{align}
and
\begin{align}
\begin{split}
&\left|\frac{\tau^{\frac12}(f'(x)|x|-(2q-2)sign(x)f(x))(\beta_1\tau^\theta+\beta_2h^\rho)|x|^{\frac{2q-2}\alpha-1}}
{\big(1+(\beta_1\tau^\theta+\beta_2h^\rho)
|x|^{\frac{2q-2}\alpha}\big)^{\alpha+1}}\right|
\\
\leq
&
\frac{\tau^{\frac12}(\beta_1\tau^\theta+\beta_2h^\rho)(c_8+c_9|x|^{2q-1})|x|^{\frac{2q-2}\alpha-1}}{\big(1+(\beta_1\tau^\theta+\beta_2h^\rho)
|x|^{\frac{2q-2}\alpha}\big)^{\alpha+1}}
\\
\leq
&
\frac{c_8\tau^{\frac12}(\beta_1\tau^\theta+\beta_2h^\rho)^{\frac{\alpha}{2q-2}}
((\beta_1\tau^\theta+\beta_2h^\rho)|x|^{\frac{2q-2}\alpha})^{\frac{2q-2-\alpha}{2q-2}}}
{\big(1+(\beta_1\tau^\theta+\beta_2h^\rho)
|x|^{\frac{2q-2}\alpha}\big)^{\alpha+1}}
\\
&
+
\frac{c_9\tau^{\frac12}(\beta_1\tau^\theta+\beta_2h^\rho)|x|^{\frac{(2q-2)(\alpha+1)}\alpha}}
{\big(1+(\beta_1\tau^\theta+\beta_2h^\rho)
|x|^{\frac{2q-2}\alpha}\big)^{\alpha+1}}
\\
\leq
&
c_8\tau^{\frac12}(\beta_1\tau^\theta+\beta_2h^\rho)^{\frac{\alpha}{2q-2}}
+
\frac{c_9\tau^{\frac12}(\beta_1\tau^\theta+\beta_2h^\rho)|x|^{\frac{(2q-2)(\alpha+1)}\alpha}}
{\big(1+(\beta_1\tau^\theta+\beta_2h^\rho)
|x|^{\frac{2q-2}\alpha})^{\alpha+1}}
\end{split}
\end{align}
Therefore,
\begin{align}
\begin{split}
\tau|f_{\tau,h}'(x)|^2
\leq
&
\bigg(R_f\tau^{\frac12}
 +
 c_8\tau^{\frac12}(\beta_1\tau^\theta+\beta_2h^\rho)^{\frac{\alpha}{2q-2}}
 \\
 &
+
\frac{R_f \tau^{\frac12}|x|^{2q-2}+c_9\tau^{\frac12}(\beta_1\tau^\theta+\beta_2h^\rho)|x|^{\frac{(2q-2)(\alpha+1)}\alpha}}
{\big(1+(\beta_1\tau^\theta+\beta_2h^\rho)
|x|^{\frac{2q-2}\alpha})^{\alpha+1}}
\bigg)^2
\\
\leq
&
2\tau\Big(R_f
 +
 c_8(\beta_1\tau^\theta+\beta_2h^\rho)^{\frac{\alpha}{2q-2}}\Big)^2
 \\
 &
 +
4\bigg( \frac{R^2_f \tau|x|^{4q-4}+c_9^2\tau(\beta_1\tau^\theta+\beta_2h^\rho)^2|x|^{\frac{(4q-4)(\alpha+1)}\alpha}}
{\big(1+(\beta_1\tau^\theta+\beta_2h^\rho)
|x|^{\frac{2q-2}\alpha})^{2\alpha+2}}\bigg).
\end{split}
\end{align}
This in combination with \eqref{eq:derive-fhtau} shows
\begin{align}
\begin{split}
&2f_{\tau,h}'(x)
+
\tau|f_{\tau,h}'(x)|^2
\\
\leq
&
2L_f
+
2c_6(\beta_1\tau^\theta+\beta_2h^\rho)^{\frac{\alpha}{2q-2}}
+2\tau\Big(R_f
 +
 c_8(\beta_1\tau^\theta+\beta_2h^\rho)^{\frac{\alpha}{2q-2}}\big)^2
\\
&
-
\frac{2c_7(\beta_1\tau^\theta+\beta_2h^\rho)|x|^{\frac{(2q-2)(\alpha+1)}\alpha}}{\big(1+(\beta_1\tau^\theta+\beta_2h^\rho)
|x|^{\frac{2q-2}\alpha}\big)^{\alpha+1}}
 +
 \frac{4(R^2_f \tau|x|^{4q-4}+c^2_9\tau(\beta_1\tau^\theta+\beta_2h^\rho)^2|x|^{\frac{(4q-4)(\alpha+1)}\alpha})}
{\big(1+(\beta_1\tau^\theta+\beta_2h^\rho)
|x|^{\frac{2q-2}\alpha}\big)^{2\alpha+2}}
\\
\leq
&
2L_f
+
2c_6(\beta_1\tau^\theta+\beta_2h^\rho)^{\frac{\alpha}{2q-2}}
+
2\tau\Big(R_f
 +
 c_8(\beta_1\tau^\theta+\beta_2h^\rho)^{\frac{\alpha}{2q-2}}\Big)^2
 -
\frac{2c_7(\beta_1\tau^\theta+\beta_2h^\rho)|x|^{\frac{(2q-2)(\alpha+1)}\alpha}}{\big(1+(\beta_1\tau^\theta+\beta_2h^\rho)
|x|^{\frac{2q-2}\alpha}\big)^{\alpha+1}}
\\
&
 +
 \frac{\frac{4R^2_f(1-\alpha) \tau^{1-\frac{2\alpha\theta}{1+\alpha}}}{\alpha+1}+\frac{8R^2_f \alpha\tau^{1-\frac{2\alpha\theta}{1+\alpha}+\theta} }{\alpha+1}|x|^{\frac{(2q-2)(\alpha+1)}\alpha}+4c^2_9\tau(\beta_1\tau^\theta+\beta_2h^\rho)^2|x|^{\frac{(4q-4)(\alpha+1)}\alpha}}
{\big(1+(\beta_1\tau^\theta+\beta_2h^\rho)
|x|^{\frac{2q-2}\alpha}\big)^{2\alpha+2}}
\\
\leq
&
2L_f
+
2c_6(\beta_1\tau^\theta+\beta_2h^\rho)^{\frac{\alpha}{2q-2}}
+
2\tau\Big(R_f
 +
c_8(\beta_1\tau^\theta+\beta_2h^\rho)^{\frac{\alpha}{2q-2}}\Big)^2
 +
 \frac{4R^2_f(1-\alpha)}{\alpha+1} \tau^{1-\frac{2\alpha\theta}{1+\alpha}}
 \\
 &
 +
  \frac{\left(-2c_7+\frac{8R^2_f \beta_1^{-1} \alpha }{\alpha+1}\tau^{1-\frac{2\alpha\theta}{1+\alpha}}+4c_9^2\beta_1^{-\alpha}\tau^{1-\alpha\theta}\right)(\beta_1\tau^\theta+\beta_2h^\rho)|x|^{\frac{(2q-2)(\alpha+1)}\alpha}
  (1+(\beta_1\tau^\theta+\beta_2h^\rho)
|x|^{\frac{2q-2}\alpha})^{\alpha+1}
  }
{\big(1+(\beta_1\tau^\theta+\beta_2h^\rho)
|x|^{\frac{2q-2}\alpha}\big)^{2\alpha+2}},
\end{split}
\end{align}
where in the second inequality we used the Young inequality
$\tau^{\frac{2\alpha\theta}{1+\alpha}}|x|^{4q-4}
\leq \frac{1-\alpha}{1+\alpha}+\frac{2\alpha}{1+\alpha}\tau^\theta|x|^{\frac{(2q-2)(1+\alpha)}\alpha}$.
Under the condition $L_F<\lambda_1$, there exist constants $\tau^*$ and $h^*$ such that
for $\tau\in (0,\tau^*)$ and $h\in(0,h^*)$,
\begin{align}
-2c_7
+
\tfrac{8R^2_f \beta_1^{-1} \alpha }{\alpha+1}\tau^{1-\frac{2\alpha\theta}{1+\alpha}}+4c_9^2\beta_1^{-\alpha}\tau^{1-\alpha\theta}
\leq
0,
\end{align}
and
\begin{align}
2c_6(\beta_1\tau^\theta+\beta_2h^\rho)^{\frac{\alpha}{2q-2}}
+
2\tau\Big(R_f
 +
c_8(\beta_1\tau^\theta+\beta_2h^\rho)^{\frac{\alpha}{2q-2}}\Big)^2
 +
 \tfrac{4R^2_f(1-\alpha) }{\alpha+1}\tau^{1-\frac{2\alpha\theta}{1+\alpha}}
 \leq
 \lambda_1-L_f.
\end{align}
This shows \eqref{asum:condition-ftauh-derivate} and ends the proof of this lemma. $\square$

Next, we consider the Malliavin derivative of the continuous version of the discrete  solution  process and prove some estimates
needed later. Below, we introduce a continuous version of the fully finite element approximation \eqref{eq:full-discretization}, defined by $X_{\tau,h}(t)=X_{\tau,h}^m$ for $t=t_m$ and, for $t\in [t_m, t_{m+1})$

\begin{align}\label{eq:continouse-verstion-Xtauh}
\begin{split}
X_{\tau,h}(t)
=
X_{\tau,h}^m
+
\int_{t_m}^t(-E_{\tau,h}A_hX_{\tau,h}^m+E_{\tau,h}F_{\tau,h}(X_{\tau,h}^m))\,\dd t
+
\int_{t_m}^tE_{\tau,h}P_h\,\dd W(t).
\end{split}
\end{align}

\begin{lemma}\label{lem:eq-mallivin-derivative-numercial-solu}
Suppose Assumptions
 \ref{assum:linear-operator-A}-\ref{assum:intial-value-data} are valid for $\gamma\in(\frac d2,2]$ or $\gamma\in\left(0,\frac12\right)$ with $Q=I$
 in dimension one. Let $X_{\tau,h}^m$ be the solution of the fully discretization \eqref{eq:full-discretization}. Then the Malliavin derivative of $X_{\tau,h}(t)$ satisfies, for $s<t_m$ and $t\in[t_{m}, t_{m+1})$
 \begin{align}\label{eq:bound-malivin-derive-bound}
 \|A^{\frac{\beta-1}2}D_s^zX_{\tau,h}(t)\|_{L^p(\Omega;H)}
 \leq
 C\|A_h^{\frac{\beta-1}2}D_s^zX_{\tau,h}(t)\|_{L^p(\Omega;H)}
 \leq
 C(Q,X_0,p,\gamma)\|A_h^{\frac{\beta-1}2}P_hz\|,
 \end{align}
where $\beta=\min\{\gamma,1\}$  and the constant $C(Q,X_0,\gamma)$ is independent of $h$, $\tau$, $s$ and $t$.
\end{lemma}
{\it Proof of Lemma \ref{lem:eq-mallivin-derivative-numercial-solu}.}
For $s\geq t_{m-1}$, one has $D_s^zX_{\tau,h}^n=0$, for $n=1,2,\cdots, m-1.$ For $t_{m-1}\leq s<t_m$, one sees
\begin{align}
D_s^zX_{\tau,h}^m=E_{\tau,h}P_hz+E_{\tau,h}P_hDF_{\tau,h}(X_{\tau,h}^{m-1})D_sX_{\tau,h}^{m-1}=E_{\tau,h}P_hz.
\end{align}
Hence, it is easy to show, for $s\in[t_{m-1}, t_m)$
\begin{align}\label{eq:milivin-derive-bound-II}
\|D_s^zX_{\tau,h}^m\|_{L^p(\Omega;H)}
\leq
\|E_{\tau,h}P_hz\|
\leq
C\tau^{-\frac{1-\beta}2}\|A^{\frac{\beta-1}2}z\|,
\end{align}
and
\begin{align}\label{eq:milivin-derive-bound-I}
\|A_h^{\frac{\beta-1}2}D_s^zX_{\tau,h}^m\|_{L^p(\Omega;H)}
\leq
\|A_h^{\frac{\beta-1}2}E_{\tau,h}P_hz\|_{L^p(\Omega;H)}
\leq
C\|A^{\frac{\beta-1}2}z\|.
\end{align}
By \eqref{eq:full-discretization} and the chain rule, the Malliavin derivative of
$X_{\tau,h}^m$ can be derived as follows, for $s< t_{m-1}$, $l_s\tau \leq s<(l_s+1)\tau$
\begin{align}\label{eq:mallivin-derivate-X}
D_s^zX_{\tau,h}^{m}
-
D_s^zX_{\tau,h}^{m-1}
=
-
\tau A_hD_s^zX_{\tau,h}^{m}
+
\tau P_hDF_{\tau,h}(X_{\tau,h}^{m-1})D_s^zX_{\tau,h}^{m-1},
\end{align}
where $l_s=\left[\frac s \tau\right]$.
By iteration, the solution of the above problem can be rewritten as
\begin{align}
D_s^zX_{\tau,h}^m
=E_{\tau,h}^{m-l_s}P_hz
+
\tau\sum_{j=l_s+1}^{m-1}E_{\tau,h}^{m-j}P_hDF_{\tau,h}(X_{\tau,h}^j)D_s^zX_{\tau,h}^j.
\end{align}
 Multiplying \eqref{eq:mallivin-derivate-X} by $D_s^zX_{\tau,h}^{m}$
and using the  equality $(x-y)x=\frac12(x^2-y^2+(x-y)^2)$ yield
\begin{align}
\begin{split}
&\frac12(\|D_s^zX_{\tau,h}^{m}\|^2-\|D_s^zX_{\tau,h}^{m-1}\|^2+\|D_s^zX_{\tau,h}^{m}-D_s^zX_{\tau,h}^{m-1}\|^2)
\\
=
&
-\tau\|\nabla D_s^zX_{\tau,h}^{m}\|^2
+
\tau(DF_{\tau,h}(X_{\tau,h}^{m-1})D_s^zX_{\tau,h}^{m-1},D_s^zX_{\tau,h}^{m}-D_s^zX_{\tau,h}^{m-1})
\\
&
+
\tau(DF_{\tau,h}(X_{\tau,h}^{m-1})D_s^zX_{\tau,h}^{m-1},D_s^zX_{\tau,h}^{m-1})
\\
\leq
&
-\tau\|\nabla D_s^zX_{\tau,h}^{m}\|^2
+
\frac12\tau^2\|DF_{\tau,h}(X_{\tau,h}^{m-1})D_s^zX_{\tau,h}^{m-1}\|^2
\\
&
+
\frac12\|D_s^zX_{\tau,h}^{m}-D_s^zX_{\tau,h}^{m-1}\|^2
+
\tau(DF_{\tau,h}(X_{\tau,h}^{m-1})D_s^zX_{\tau,h}^{m-1},D_s^zX_{\tau,h}^{m-1}).
\end{split}
\end{align}
Therefore, we rely on Lemma \ref{assum:derivate-ftauh-bound} to get
\begin{align}
\begin{split}
\tfrac12(\|D_s^zX_{\tau,h}^{m}\|^2-\|D_s^zX_{\tau,h}^{m-1}\|^2)
\leq
&
-\tau \lambda_1\|D_s^zX_{\tau,h}^{m}\|^2
+
\tau(L_f+\frac{\lambda_1-L_f}2)\|D_s^zX_{\tau,h}^{m-1}\|^2,
\end{split}
\end{align}
implying
\begin{align}
\begin{split}
\|D_s^zX_{\tau,h}^{m}\|^2
\leq
&
\frac{1+\tau(\lambda_1+L_f)}{1+2\tau \lambda_1}
\|D_s^zX_{\tau,h}^{m-1}\|^2
\leq
\frac1{1+\frac{\tau (\lambda_1-L_f)}{1+(\lambda_1+L_f)\tau}}\|D_s^zX_{\tau,h}^{m-1}\|^2
\\
\leq
&
\frac1{1+\frac{\tau (\lambda_1-L_f)}{1+(\lambda_1+L_f)\tau^*}}\|D_s^zX_{\tau,h}^{m-1}\|^2.
\end{split}
\end{align}
This means that, for any $\tau\in (0, \tau^*)$
\begin{align}\label{eq:bound-1+tauAh-DF}
\|(1+\tau A_h)^{-1}(1+\tau P_hDF_{\tau,h}( X_{\tau,h}^{m-1})\|_{\mathcal{L}(H)}^2
\leq
\frac1{1+\frac{\tau (\lambda_1-L_f)}{1+(\lambda_1+L_f)\tau^*}}.
\end{align}
By $V_z(m,s):=D^z_sX_{\tau,h}^{m}-E_{\tau,h}^{m-l_s}P_hz$, it follows
\begin{align}
\begin{split}
V_z(m,s)
=
&
V_z(m-1,s)
-
\tau A_hV_z(m,s)
+
\tau P_h DF_{\tau,h}(X_{\tau,h}^{m-1})\cdot V_z(m-1,s)
\\
&
+
\tau P_h DF_{\tau,h}(X_{\tau,h}^{m-1})\cdot E_{\tau,h}^{m-1-l_s}P_hz.
\end{split}
\end{align}
Moreover, by \eqref{lem:eq-spatial-regu-ENm}, \eqref{eq:bound-1+tauAh-DF} and the fact $V(l_s+1,z)=D_s^zX_{\tau,h}^{l_s+1}-E_{\tau,h}P_hz=0$
\begin{align}
\begin{split}
\|V_z(m,s)\|
\leq
&
\|(1+\tau A_h)^{-1}(1+\tau P_hDF_{\tau,h}(X_{\tau,h}^{m-1}) )\cdot V_z(m-1,s)\|
\\
&
+
\|(1+\tau A_h)^{-1}\tau P_h DF_{\tau,h}(X_{\tau,h}^{m-1})\cdot E_{\tau,h}^{m-1-l_s}P_hz\|
\\
\leq
&
\|V_z(m-1,s)\|
+
(1+\lambda_1\tau)^{-1}\tau\| DF_{\tau,h}(X_{\tau,h}^{m-1})\cdot E_{\tau,h}^{m-1-l_s}P_hz\|
\\
\leq
&
C\tau\sum_{j=l_s+1}^{m-1}
\| DF_{\tau,h}(X_{\tau,h}^j)\cdot E_{\tau,h}^{j-l_s}P_hz\|
\\
\leq
&
C\tau\sum_{j=l_s+1}^{m-1} \min\left\{t_{j-l_s}^{-\frac\mu2},t_{j-l_s}^{-2}\right\}
\left(1+\|X_{\tau,h}^j\|^{2q-2}_V\right)\|A_h^{-\frac\mu 2}P_hz\|,
\end{split}
\end{align}
for $\mu\in[0, 1]$,
which together with  \eqref{eq:sum} leads to
\begin{align}
\begin{split}
\|V_z(m,s)\|_{L^p(\Omega;H)}
\leq
&
C\tau\sum_{j=l_s+1}^{m-1} \min\left\{t_{j-l_s}^{-\frac\mu2}, t_{j-l_s}^{-2}\right\}\|A_h^{-\frac\mu2}P_hz\| (1+\sup_{m\in \mathbb{N}}\|X_{\tau,h}^j\|^{2q-2}_{L^{(2q-2)p}(\Omega,V)})
\\
\leq
&
C\|A_h^{-\frac\mu2}P_hz\|.
\end{split}
\end{align}
Then, from  \eqref{eq:relation-A-Ah} and \eqref{lem:eq-spatial-regu-ENm}, it follows that, for $s<t_{m-1}$
\begin{align}
\begin{split}
\|D^z_sX_{\tau,h}^{m}\|_{L^p(\Omega;H)}
\leq
&
\|V_z(m,s)\|_{L^p(\Omega;H)}
+
\|E_{\tau,h}^{m-l_s}P_hz\|
\\
\leq
&
C\left(1+\min\left\{t_{m-l_s}^{-\frac{1-\beta}2}, t_{m-l_s}^{-2}\right\}\right)\|A^{\frac{\beta-1}2}z\|.
\end{split}
\end{align}
Hence by the above estimate and \eqref{eq:milivin-derive-bound-II}, we have for $s<t_m$
 \begin{align}\label{eq:milivin-derive-bound-IIi}
\|D_s^zX_{\tau,h}^m\|_{L^p(\Omega;H)}
\leq
C\left(1+\min\left\{t_{m-l_s}^{-\frac{1-\beta}2}, t_{m-l_s}^{-2}\right\}\right)\|A^{\frac{\beta-1}2}z\|.
\end{align}
Furthermore, the a priori estimate of $X_{\tau,h}^j$, the smooth regularity of $E_{\tau,h}^j$ in Lemma \ref{lem:eq-smooth-ENm} and the above estimate imply, for $s<t_{m}$
\begin{align}\label{eq:bund-malvin-derive-xhn}
\begin{split}
&\|A_h^{\frac{\beta-1}2}D^z_sX_{\tau,h}^{m}\|_{L^p(\Omega;H)}
\\
\leq
&
\|A_h^{\frac{\beta-1}2}E_{\tau,h}^{m-l_s}P_hz\|
+
\sum_{j=l_s+1}^{m-1}\|A_h^{\frac{\beta-1}2}E_{\tau,h}^{m-j}DF_{\tau,h}(X_{\tau,h}^j)D_s^zX_{\tau,h}^j\|_{L^p(\Omega;H)}
\\
\leq
&
\|A_h^{\frac{\beta-1}2}P_hz\|
+
C\sum_{j=l_s+1}^{m-1}\left\|(1+\|X_{\tau,h}^j\|_V^{2q-2})\|D_s^zX_{\tau,h}^j\|\right\|_{L^p(\Omega;\mathbb{R})}
\\
\leq
&
\|A^{\frac{\beta-1}2}z\|
+
C\tau\sum_{j=l_s+1}^{m-1} \left(1+\min\{t_{j-l_s}^{-\frac{1-\beta}2}, t_{j-l_s}^{-2}\}\right) \|A^{\frac{\beta-1}2}z\|
\\
\leq
&
C\|A^{\frac{\beta-1}2}z\|.
\end{split}
\end{align}

Next, we show \eqref{eq:bound-malivin-derive-bound} based on the above estimates.
By \eqref{eq:continouse-verstion-Xtauh} and the chain rule, the Malliavin derivative of $X_{\tau,h}(t)$, for $t\in [t_m, t_{m+1})$
can be derived as follows, for $s<t_m$,
\begin{align}\label{eq:malavin-derive-continouse-verstion-Xtauh}
\begin{split}
D_s^zX_{\tau,h}(t)
=
D_s^zX_{\tau,h}^m
-
(t-t_m)E_{\tau,h}A_hD_s^zX_{\tau,h}^m+(t-t_m)E_{\tau,h}DF_{\tau,h}(X_{\tau,h}^m)D_sX_{\tau,h}^m.
\end{split}
\end{align}
Similarly as in \eqref{eq:bund-malvin-derive-xhn}, we obtain, for $s<t_m$,
\begin{align}
\begin{split}
&\|A_h^{\frac{\beta-1}2}D_s^zX_{\tau,h}(t)\|_{L^p(\Omega;  H)}
\\
\leq
&
\|A_h^{\frac{\beta-1}2}D_s^zX_{\tau,h}^m\|_{L^p(\Omega;H)}
+
(t-t_m)\|A_hE_{\tau,h}A_h^{\frac{\beta-1}2}D_s^zX_{\tau,h}^m\|_{L^p(\Omega;H)}
\\
&
+
(t-t_m)\|E_{\tau,h}A_h^{\frac{\beta-1}2}P_hDF_{\tau,h}(X_{\tau,h}^m)D_s^zX_{\tau,h}^m\|_{L^p(\Omega;H)}
\\
\leq
&
C(1+
+
\tau\tau^{-1})\|A_h^{\frac{\beta-1}2}D_s^zX_{\tau,h}^m\|_{L^p(\Omega;H)}
+
\tau\|P_hDF_{\tau,h}(X_{\tau,h}^m)D_s^zX_{\tau,h}^m\|_{L^p(\Omega;H)}
\\
\leq
&
C \left(1+\tau\min\{t_{m-l_s}^{-\frac{1-\beta}2}, t_{m-l_s}^{-2}\}\right)\|A^{\frac{\beta-1}2}z\|
\leq C(Q,X_0,p,\gamma)\|A^{\frac{\beta-1}2}z\|.
\end{split}
\end{align}
This shows, for $s<t_m$ and $t\in [t_m, t_{m+1})$
\begin{align}
\begin{split}
\|A_h^{\frac{\beta-1}2}D_sX_{\tau,h}(t)\|_{L^p(\Omega;  \mathcal{L}_2^0)}
=
&
\left(\mathbb{E}\left[\Big(\sum_{l\in \mathbb{N}^+}\|A_h^{\frac{\beta-1}2}D_s^{Q^{\frac12}e_l}X_{\tau,h}(t)\|^2\Big)^{\frac p2}\right]\right)^{\frac1p}
\\
\leq
&
C\left(\left[\Big(\sum_{l\in \mathbb{N}^+}\|A^{\frac{\beta-1}2}Q^{\frac12}e_l\|^2\Big)^{\frac p2}\right]\right)^{\frac1p}
\\
\leq
&
C\|A^{\frac{\beta-1}2}Q^{\frac12}\|_{\mathcal{L}_2(H)}.
\end{split}
\end{align}
The proof of this lemma is thus complete. $\Box$

\subsection{Uniform-in-time weak convergence rate}
The goal of this subsection is to carry out the uniform-in-time weak convergence analysis of the fully discrete finite element approximation
\eqref{eq:full-discretization}.

Under the condition $L_f<\lambda_1$, we can acquire  an exponential convergence to equilibrium for the SPDE \eqref{eq:parabolic-SPDE}
(see, e.g. \cite[Proposition 3.3]{brehier2022approximation}, \cite[Theorem 12.5]{goldys2006stochasti}).
\begin{proposition}\label{pro:exponential-convergence-equli}
Let Assumptions \ref{assum:linear-operator-A}-\ref{assum:noise-term} be fulfilled and let $L_f<\lambda_1$.
Let $X(t,x)$ be the mild solution of \eqref{eq:parabolic-SPDE} with initial value $X(0):=x\in H$. Then there exist constants $c>0$ and
$C>0$ such that for any $\varphi\in C_b^2(H)$, $t\geq 0$ and $u,v\in H$
\begin{align}
\left|\mathbb{E}[\varphi(X(t,u))]-\mathbb{E}[\varphi(X(t,v))]\right|
\leq
C\|\varphi\|_{C_b^1(H)}
e^{-ct}\|u-v\|.
\end{align}
\end{proposition}

Due to the Doob theorem, Proposition \ref{pro:exponential-convergence-equli} ensures the existence of the unique invariant measure of SPDE \eqref{eq:parabolic-SPDE}.
\begin{theorem}
Let Assumptions \ref{assum:linear-operator-A}-\ref{assum:noise-term} be fulfilled and let $L_f<\lambda_1$. Then the SPDE \eqref{eq:parabolic-SPDE} admits a unique invariant measure $\mu$.
\end{theorem}

For the weak convergence analysis below, we need to consider the spectral Galerkin method for \eqref{eq:parabolic-SPDE}. For $ N\in \mathbb{N}$, we define a finite-dimensional subspace $H^N\subset H$ by $H^N:=span\{e_1,e_2,\cdots, e_N\}$, where $\{e_i\}_{i=1}^N$ is the $N$-first eigenvectors of the dominant linear operator $A$. Let $P_N$ be the projection operator from $H$ to $H^N$, defined by
$P_N v¡êo=\sum_{n=1}^N\left<v,e_n\right>e_n$, $v\in H^\vartheta$, $\vartheta\in \mathbb{R}^+$. It is easy to show
\begin{align}\label{eq:projection-property-Pn}
\|(I-P_N)v\|
\leq
\lambda_N^{-\frac{\vartheta}2}\|v\|_\vartheta,\;\forall v\in  H^\vartheta, \vartheta\geq 0.
\end{align}
Then the spectral Galerkin approximation of \eqref{eq:parabolic-SPDE} is given by
\begin{align}\label{eq:spectral-galerkin-parabolic-SPDE}
\,\dd X^N(t)
=
-A_NX^N(t)+P_NF(X^N(t))\,\dd t
=
\,P_N\dd W(t),\; X^N(0)=P_NX_0,
\end{align}
where $A_N=P_NA$. Let $\varphi\in C_b^2(H)$  and we define
\begin{align}
\mu^N(t,x):=\mathbb{E}[\varphi(X^N(t,x))], \;t\geq 0, \;x\in H,
\end{align}
which is the unique solution of the Kolmogorov equation associated with $X^N(t,x)$, (see \cite[Theorem 9.16]{da2014stochastic})
\begin{align}
\partial_t \mu^N(t,x)
=D\mu^N(t,x)\cdot(-A_Nx+P_NF(x))
+
\frac12 \mathrm{Tr}[D^2\mu^N(t,x)P_NQ^{\frac12}(P_NQ^{\frac12})^*],
\end{align}
with $\mu^N(0,\cdot)=\varphi(P_N\cdot)$.

The next lemma presents the uniform-in-time moments bound of the numerical solution $X^N(t,x)$, (see, e.g.,  \cite[Lemma 5.5]{jiang2025uniform}, \cite[Lemma 2]{Cui2021Weak}).
\begin{lemma}
Suppose Assumptions
 \ref{assum:linear-operator-A}-\ref{assum:intial-value-data} are valid for $\gamma\in (\frac  d2,2]$ or $\gamma\in (0,\frac12)$ with $Q=I$ in dimension one. Let $X^N(t), t\geq 0$ be the mild solution of \eqref{eq:spectral-galerkin-parabolic-SPDE} with the initial value $X^N(0):=X_0^N$. Then, there exists a constant $C:=C(Q,p,q,\gamma)$ such that
\begin{align}
\sup_{t\geq 0}\|X^N(t)\|_{L^p(\Omega;V)}
\leq
C
\left(1
+
\|X^N_0\|_{L^{\frac{(2q-1)[(2q-2)((8q-8)-(2q-3)d)+4]p}{4-(2q-3)d}}(\Omega;V)}^{\frac{(2q-1)[(2q-2)((8q-8)-(2q-3)d)+4]}{4-(2q-3)d}}\right),
\end{align}
where  $C$  is independent of $m$.
\end{lemma}

Also, we need the regularity estimates for $\mu^N(\cdot, \cdot)$, $N\in \mathbb{N}$, (see \cite[Lemma 5.5]{jiang2025uniform}).
\begin{lemma}\label{lem:bound-Komogrolv-equation}
For any $\varphi\in C_b^2(H)$ and $0<\vartheta_0,\vartheta_1,\vartheta_2$, $\vartheta_1+\vartheta_2<1$,
there exist constants $c>0$, $C(Q,\vartheta_0,\varphi)>0$, and $C(Q,\vartheta_1,\vartheta_2,\varphi)>0$ such that
$x,y,z\in H^N$ and $t>0$
\begin{align}
\begin{split}
|D\mu^N(t,x)y|
\leq
&
C(Q,\vartheta_0,\varphi)
\left(1+\sup_{s\in[0,t]}\mathbb{E}\left[\|X^N(s,x)\|_V^{2q}\right]\right)\big(1+t^{-\vartheta_0}\big)e^{-ct}\|y\|_{-2\vartheta_0},
\\
\left|D^2\mu^N(t,x)\cdot(y,z)\right|
\leq
&
C(Q,\vartheta_1,\vartheta_2,\varphi)
\left(1+\sup_{s\in[0,t]}\mathbb{E}\left[\|X^N(s,x)\|_V^{8q-2}\right]\right)(1+t^{-\vartheta_1-\vartheta_2})e^{-ct}\|y\|_{-2\vartheta_1}
\|z\|_{-2\vartheta_2}.
\end{split}
\end{align}
\end{lemma}

\begin{theorem}\label{them:weak-vonvergence}
Let Assumptions
 \ref{assum:linear-operator-A}-\ref{assum:intial-value-data} and Assumption \ref{assum:derivate-ftauh-bound} be valid  for  $\gamma\in(\frac d2,2]$ or $\gamma\in(0,\frac12)$ with $Q=I$
in dimension one. Let $X(t)$ and $X_{\tau,h}^m$, $m\in \mathbb{N}$, be defined by
 \eqref{eq:definition-mild-solution} and \eqref{eq:fully-discrete-problem}, respectively.
Let the condition $\max\{\frac12 \alpha\rho, \alpha\theta\}<1+\frac d{2q(2q-1)}-\frac d4$ hold. Then  there exists a constant $C(Q,X_0,\gamma)>0$, such that
 for any $\iota\in(0, \beta)$, $\beta:=\min\{\gamma,1\}$,
\begin{align}
\left |\mathbb{E}\big[\varphi(X(t_m))\big]-\mathbb{E}\big[\varphi(X^m_{\tau,h})\big]\right|
\leq
C(1+t_m^{-\iota}+(t_m-\tau)^{-\frac12})(h^{\min\{2\iota,\rho\}}+\tau^{\min\{\iota,\theta\}}),
 \end{align}
 where
 $C$ is independent of $t_m$.
\end{theorem}
{\it Proof of Theorem \ref{them:weak-vonvergence}.} By recalling the fact $\mu^N(t,x)=\mathbb{E}[\varphi(X^N(t,x))]$, $t\geq 0$, we decompose the weak error into the following three terms
\begin{align}
\begin{split}
\mathbb{E}\big[\varphi(X(t_m))\big]-\mathbb{E}\big[\varphi(X^m_{\tau,h})\big]
=
&
\mathbb{E}\big[\varphi(X(t_m))\big]-\mathbb{E}\big[\varphi(X^N(t_m))\big]
+
\mathbb{E}\big[\varphi(X^N(t_m))\big]-\mathbb{E}\big[\varphi(P_NX^m_{\tau,h})\big]
\\
&
+
\mathbb{E}\big[\varphi(P_NX^m_{\tau,h})\big]-\mathbb{E}\big[\varphi(X^m_{\tau,h})\big].
\end{split}
\end{align}
For the third term, we have, by \eqref{th:spatal-regu-full-solu} and \eqref{eq:projection-property-Pn}, for  $\iota\in(0, \min\{\gamma,1\})$
\begin{align}\label{eq:PnXm-XM}
\mathbb{E}\big[\varphi(P_NX^m_{\tau,h})\big]-\mathbb{E}\big[\varphi(X^m_{\tau,h})\big]
& \leq
\|\varphi\|_{C_b^1(H)}
\|(P_N-I)X^m_{\tau,h}\|_{L^p(\Omega;H)}
\\
\nonumber
& \leq
C\lambda_N^{-\frac\iota2}\|A^{\frac\iota2}X^m_{\tau,h}\|_{L^p(\Omega;H)}
\rightarrow 0, \quad N \to \infty.
\end{align}
For the first term, we use the well-established strong convergence rate within finite time horizon (see \cite[Theorem 4.1]{QQQi2018Optimal} ),
and take the limit $N\rightarrow \infty$ to obtain
\begin{align}\label{eq:xtm-xntm}
|\mathbb{E}\big[\varphi(X(t_m))\big]-\mathbb{E}\big[\varphi(X^N(t_m))\big]|
\leq
\|\varphi\|_{C_b^1(H)}\|X(t_m)-X^N(t_m)\|_{L^2(\Omega;H)}
\leq
 C(t_m,\varphi)\lambda_N^{-\frac{\gamma}2}\rightarrow 0.
\end{align}
For the second term, we note that $\mu^N(t,x)=\mathbb{E}\left[\varphi(X^N(t,x))\right], t\geq 0, x\in H$  and divide it into two parts:
\begin{align}\label{eq:decompose-XN-PNXtau}
\begin{split}
\left|\mathbb{E}\big[\varphi(X^N(t_m))\big]-\mathbb{E}\big[\varphi(P_NX^m_{\tau,h})\big]\right|
=
&
\mathbb{E}\big[\mu^N(t_m,X_0)\big]-\mathbb{E}\big[\mu^N(0,X_{\tau,h}^m)\big]
\\
=
&
\mathbb{E}\big[\mu^N(t_m,X_0)\big]
-
\mathbb{E}\big[\mu^N(t_m,P_hX_0)\big]
\\
&
+
\mathbb{E}\big[\mu^N(t_m,P_hX_0)\big]
-
\mathbb{E}\big[\mu^N(0,X_{\tau,h}^m)\big].
\end{split}
\end{align}
Owing to \eqref{eq:property-Ih}  and Lemma \ref{lem:bound-Komogrolv-equation}, one can see
\begin{align}
\begin{split}
&\left|\mathbb{E}\big[\mu^N(t_m,X_0)\big]
-
\mathbb{E}\big[\mu^N(t_m,P_hX_0)\big]\right|
\\
\leq
&
\int_0^1\left|\mathbb{E}\big[D\mu^N(t_m,P_hX_0+\xi(I-P_h)X_0)\cdot (I-P_h)X_0\big]\right|\,\dd \xi
\\
\leq
&
C\int_0^1(1+\sup_{s\in[0,t_m]}\mathbb{E}
\left[\|X(s,P_hX_0+\xi(I-P_h)X_0)\|_V^{2q}\right])\,\dd \xi(1+t_m^{-\iota})e^{-ct_m}\|(I-P_h)X_0\|_{-2\iota}
\\
\leq
&
C(Q,X_0,q,d)(1+t_m^{-\iota})h^{2\iota}.
\end{split}
\end{align}
To proceed with the second term in \eqref{eq:decompose-XN-PNXtau}, we decompose it further as follows:
\begin{align}
\begin{split}
&\mathbb{E}\big[\mu^N(t_m,P_hX_0)\big]
-
\mathbb{E}\big[\mu^N(0,X_{\tau,h}^m)\big]
\\
=
&
\sum_{n=0}^{m-1}\left(\mathbb{E}\left[\mu^N(m\tau-n\tau,X_{\tau,h}^n)\right]
-
\mathbb{E}\left[\mu^N(m\tau-(n+1)\tau,X_{\tau,h}^{n+1})\right]\right)
\\
=
&
\left(\mathbb{E}\left[\mu^N(m\tau,X_{\tau,h}^0)\right]
-
\mathbb{E}\left[\mu^N(m\tau-\tau,X_{\tau,h}^1)\right]\right)
\\
&
+
\sum_{n=1}^{m-1}\left(\mathbb{E}\left[\mu^N(m\tau-n\tau,X_{\tau,h}^n)\right]
-
\mathbb{E}\left[\mu^N(m\tau-(n+1)\tau,X_{\tau,h}^{n+1})\right]\right).
\end{split}
\end{align}
Applying the It\^o formula to $X_{\tau,h}(t)$ defined by \eqref{eq:continouse-verstion-Xtauh} gives
\begin{align}
\begin{split}
\mu^N(t,X_{\tau,h}(t))
=
&
\mu^N(t_k,X_{\tau,h}(t_k))
+
\int_{t_k}^t\frac{\dd \mu^N}{\dd s}(s,X_{\tau,h}(s))+L_{k,\tau}\mu^N(s,X_{\tau,h}(s))\,\dd s
\\
&
+
\int_{t_k}^t\big<D\mu^N(s,X_{\tau,h}(s)),\,\dd W(s)\big>,
\end{split}
\end{align}
where
\begin{align}
\begin{split}
L_{k,\tau}\mu^N(t,X_{\tau,h}(t))
=
&
\left<-E_{\tau,h}A_hX_{\tau,h}^m+E_{\tau,h}F_{\tau,h}(X_{\tau,h}^m), D\mu^N(t,X_{\tau,h}(t))\right>
\\
&
+
\frac12 \mathrm{Tr}\left\{E_{\tau,h}P_hQ (E_{\tau,h}P_h)^*
D^2\mu^N\big(t,X_{\tau,h}(t)\big)\right\}.
\end{split}
\end{align}
Then utilizing the Kolmogorov equation shows
\begin{align}
\begin{split}
&\mathbb{E}\big[\mu^N(t_m,P_hX_0)\big]
-
\mathbb{E}\big[\mu^N(0,X_{\tau,h}^m)\big]
\\
=&
\left(\mathbb{E}\left[\mu^N(m\tau,X_{\tau,h}^0)\right]
-
\mathbb{E}\left[\mu^N(m\tau,X_{\tau,h}^1)\right]\right)
\\
&
+
\sum_{n=1}^{m-1}\int_{t_n}^{t_{n+1}}\frac12\mathbb{E}\left[\mathrm{Tr}\left\{\left(E_{\tau,h}P_hQ (E_{\tau,h}P_h)^*
-P_NQ^{\frac12}(P_NQ^{\frac12})^*\right)
D^2\mu^N\big(t_m-t,X_{\tau,h}(t)\big)
\right\}\right]
\,\dd t
\\
&
+
\sum_{n=1}^{m-1}\int_{t_n}^{t_{n+1}}\mathbb{E}\left[\left<A_NX_{\tau,h}(t)
-
E_{\tau,h}A_hX_{\tau,h}^n
,
D\mu^N(t_m-t,X_{\tau,h}(t))\right>
\right]\,\dd t
\\
&
+
\int_{t_n}^{t_{n+1}}\mathbb{E}\left[\left<E_{\tau,h}P_hF_{\tau,h}(X_{\tau,h}^n)-P_NF(X_{\tau,h}(t)),D\mu^N(t_m-t, X_{\tau,h}(t))\right>\right]\,\dd t
\\
=
&
\left(\mathbb{E}\left[\mu^N(m\tau,X_{\tau,h}^0)\right]
-
\mathbb{E}\left[\mu^N(m\tau-\tau,X_{\tau,h}^1)\right]\right)
+
\sum_{n=1}^{m-1}
(a_n+b_n+c_n).
\end{split}
\end{align}
From \eqref{th:spatal-regu-full-solu}, the Markov property, a priori estimates of $X_{\tau,h}^m$ and regularity estimates in Lemma \ref{lem:bound-Komogrolv-equation}, it follows that,
\begin{align}
\begin{split}
&\left|\mathbb{E}\left[\mu^N(t_m,X_{\tau,h}^0)\right]
-
\mathbb{E}\left[\mu^N(t_m-\tau,X_{\tau,h}^1)\right]
\right|
\\
=
&
\left|\mathbb{E}\left[\mu^N(t_m-\tau,X(\tau,P_hX_0))\right]
-
\mathbb{E}\left[\mu^N(t_m-\tau,X_{\tau,h}^1)\right]
\right|
\\
=
&
\int_0^1\left|\mathbb{E}\left[D\mu^N(t_m-\tau,\xi X(\tau,P_hX_0)+(1-\xi)X_{\tau,h}^1)\cdot(X(\tau,P_hX_0)-X_{\tau,h}^1)\right]
\right|\, \dd \xi
\\
\leq
&
C(Q,\varphi)\int_0^1\left(1+\sup_{s\in[0,t_m-\tau]}\mathbb{E}\big[\|X^N(s,\xi X(\tau,P_hX_0)+(1-\xi)X_{\tau,h}^1)\|_V^{2q} \big]\,\dd\xi\right)
\\
&\cdot
(1+(t_m-\tau)^{-\frac12})e^{-c(t_m-\tau)}\mathbb{E}\left[\|X(\tau,P_hX_0)-X_{\tau,h}^1\|_{-1}\right]
\\
\leq
&
C(Q,\varphi) \left(1+(t_m-\tau)^{-\frac12}\right)e^{-c(t_m-\tau)}(h^{2\min\{\gamma,1\}}+\tau^{\min\{\gamma,1\}}),
\end{split}
\end{align}
where
\begin{align}
\begin{split}
&\mathbb{E}\left[\|X(\tau,P_hX_0)-X_{\tau,h}^1\|_{-1}\right]
=
\mathbb{E}\left[\left\|\int_0^\tau E(\tau-s)F(X(s))\,\dd s
-
\tau F_{\tau,h}(X_{\tau,h}^0)\right\|_{-1}\right]
\\
&
+
\mathbb{E}\left[\|(E(\tau)-E_{\tau,h}P_h)P_hX_0\|_{-1}\right]
+
\mathbb{E}\left[\left\|\int_0^\tau(E(\tau-s)-E_{\tau,h}P_h)\,\dd W(s)\right\|_{-1}\right]
\\
\leq
&
C(h^{1+\min\{\gamma,1\}}+\tau^{\frac{1+\min\{\gamma,1\}}2})\|X_0\|_{H^{\min\{\gamma,1\}}}
+
C\left(\int_0^\tau\|A^{-\frac12}(E(\tau-s)-E_{\tau,h}P_h)\|_{\mathcal{L}_2^0}^2\,\dd s\right)^{\frac12}
\\
&
C\tau (\sup_{t\in[0,\infty)}\|F(X(t))\|_{L^1(\Omega;H)}+\sup_{m\in \mathbb{N}_0}\|F(X_{\tau,h}^m)\|_{L^1(\Omega;H)})
\\
\leq
&
C(h^{1+\min\{\gamma,1\}}+\tau^{\frac{1+\min\{\gamma,1\}}2}+\tau)
+
C\tau^{\frac12} (h^{\min\{\gamma,1\}}+\tau^{\frac{\min\{\gamma,1\}}2})\|A^{\frac{\min\{\gamma,1\}-1}2}\|_{\mathcal{L}_2^0}
\\
\leq
&
C(h^{2\min\{\gamma,1\}}+\tau^{\min\{\gamma,1\}}).
\end{split}
\end{align}
Here we used \eqref{lem:deterministic-error-in-H-1}, \eqref{th:spatial-regu-exac-solution}, \eqref{th:spatal-regu-full-solu} and the isometry property in the first inequality and \eqref{lem:deterministic-error-in-H-1} in the second inequality.
In what follows, we focus on  estimating  the terms $a_n,b_n,c_n$. First we split the term $a_n$ into six parts:
\begin{align}\label{eq:decompose-an}
\begin{split}
a_n=
&
\frac12\int_{t_n}^{t_{n+1}}\mathbb{E}\left[ \mathrm{Tr}\left((E_{\tau,h}P_h-I)P_hQ (E_{\tau,h}P_h)^*D^2\mu^N(t_m-t,X_{\tau,h}(t))\right)
\right]\,\dd t
\\
&
+
\frac12\int_{t_n}^{t_{n+1}}\mathbb{E}\left[ \mathrm{Tr}\left((P_h-I)Q (E_{\tau,h}P_h)^*D^2\mu^N(t_m-t,X_{\tau,h}(t))\right)
\right]\,\dd t
\\
&
+
\frac12\int_{t_n}^{t_{n+1}}\mathbb{E}\left[ \mathrm{Tr}\left(Q (E_{\tau,h}P_h-P_h)^*D^2\mu^N(t_m-t,X_{\tau,h}(t))\right)
\right]\,\dd t
\\
&
+
\frac12\int_{t_n}^{t_{n+1}}\mathbb{E}\left[ \mathrm{Tr}\left(Q (P_h-I)^*D^2\mu^N(t_m-t,X_{\tau,h}(t))\right)
\right]\,\dd t
\\
&
+
\frac12\int_{t_n}^{t_{n+1}}\mathbb{E}\left[ \mathrm{Tr}\left((I-P_N)Q D^2\mu^N(t_m-t,X_{\tau,h}(t))\right)
\right]\,\dd t
\\
&
+
\frac12\int_{t_n}^{t_{n+1}}\mathbb{E}\left[ \mathrm{Tr}\left(P_NQ (I-P_N)^*D^2\mu^N(t_m-t,X_{\tau,h}(t))\right)
\right]\,\dd t
\\
=:
&
a_n^1
+
a_n^2
+
a_n^3
+
a_n^4
+
a_n^5
+
a_n^6.
\end{split}
\end{align}
Using the property of Hilbert-Schmidt and trace operator, we get, for any $\varepsilon>0$
\begin{align}
\begin{split}
&|a_n^1|
\leq
\frac12\int_{t_n}^{t_{n+1}}\Big|\mathbb{E}\Big[\mathrm{Tr}\Big(A_h^{-\beta+\varepsilon}(E_{\tau,h}-I)
A_h^{\frac{\beta-1}2}P_hA^{\frac{1-\beta}2}A^{\frac{\beta-1}2}QA^{\frac{\beta-1}2}A^{\frac{1-\beta}2}A_h^{\frac{\beta-1}2}P_h
E_{\tau,h}
\\
&
\qquad A_h^{\frac{1-\beta}2}P_hA^{\frac{\beta-1}2}
A^{\frac{1-\beta}2}D^2\mu^N(t_m-t,X_{\tau,h}(t))
A^{\frac{\beta+1}2-\varepsilon}A^{-\frac{\beta+1}2+\varepsilon}A_h^{\frac{\beta+1}2-\varepsilon}P_h\Big)
\Big]\Big|\,\dd t
\\
\leq
&
\frac12\int_{t_n}^{t_{n+1}}\mathbb{E}\Big[\|A_h^{-\beta+\varepsilon}(E_{\tau,h}-I)\|_{\mathcal{L}(H)}
\|A^{\frac{\beta-1}2}\|_{\mathcal{L}_2^0}^2
\|E_{\tau,h}P_h\|_{\mathcal{L}(H)}\|A_h^{\frac{1-\beta}2}P_hA^{\frac{\beta-1}2}\|^2_{\mathcal{L}(H)}
\\
&
\cdot
\|A^{\frac{1-\beta}2}A_h^{\frac{\beta-1}2}P_h\|_{\mathcal{L}(H)}\|A^{\frac{1-\beta}2}D^2\mu^N(t_m-t,X_{\tau,h}(t))
A^{\frac{\beta+1-2\varepsilon}2}\|_{\mathcal{L}(H)}
\|A^{-\frac{\beta+1-2\varepsilon}2}A_h^{\frac{\beta+1-2\varepsilon}2}P_h\|_{\mathcal{L}(H)}
\Big]\,\dd t,
\end{split}
\end{align}
and then obtain by   \eqref{eq:relation-A-Ah}, \eqref{eq:Ah-A-bound}, \eqref{lem:eq-temporal-regu-ENm}, Lemmas \ref{lem:eq-smooth-ENm}, \ref{lem:bound-Komogrolv-equation}
\begin{align}\label{eq:estimate-an1}
\begin{split}
|a_n^1|
\leq
&
C\tau^{\beta-\varepsilon}\int_{t_n}^{t_{n+1}}\left(1+(t_m-t)^{-1+\varepsilon}\right)e^{-c(t_m-t)}
\left(1+\sup_{s\in[0,t_m-t]}\mathbb{E}\big[\|X(s,X_{\tau,h}(t))\|_V^{2q}\big]\right)\dd t
\\
\leq
&
C\tau^{\beta-\varepsilon}\int_{t_n}^{t_{n+1}}\left(1+(t_m-t)^{-1+\varepsilon}\right)e^{-c(t_m-t)}
\,\dd t.
\end{split}
\end{align}
For $a_n^2$, we follow the similar approach used in the proof of \eqref{eq:estimate-an1}, with \eqref{eq:property-Ih} used instead to obtain
\begin{align}\label{eq:estimate-an2}
\begin{split}
|a_n^2|
=
&
\frac12\int_{t_n}^{t_{n+1}}\Big|\mathbb{E}\Big[\mathrm{Tr}\Big(A^{-\frac{1+\beta}2+\varepsilon}
(P_h-I)A^{\frac{1-\beta}2}A^{\frac{\beta-1}2}QA^{\frac{\beta-1}2}
A^{\frac{1-\beta}2}
\\
&
\cdot
A_h^{\frac{\beta-1}2}P_hE_{\tau,h}A_h^{\frac{1-\beta}2}P_hA^{\frac{\beta-1}2}
A^{\frac{1-\beta}2}D^2\mu^N(t_m-t,X_{\tau,h}(t))A^{\frac{1+\beta}2-\varepsilon}\Big)\Big]\Big|\,\dd t
\\
\leq
&
\frac12\int_{t_n}^{t_{n+1}}\mathbb{E}\Big[\|A^{-\frac{1+\beta}2+\varepsilon}(P_h-I)A^{\frac{1-\beta}2}\|_{\mathcal{L}(H)}
\|A^{\frac{\beta-1}2}\|_{\mathcal{L}_2^0}^2
\|A^{\frac{1-\beta}2}A_h^{\frac{\beta-1}2}P_h\|_{\mathcal{L}(H)}
\\
&
\cdot
\|E_{\tau,h}P_h\|_{\mathcal{L}(H)}\|A_h^{\frac{1-\beta}2}P_hA^{\frac{\beta-1}2}\|_{\mathcal{L}(H)}
\|A^{\frac{1-\beta}2}D^2\mu^N(t_m-t,X_{\tau,h}(t))A^{\frac{1+\beta}2-\varepsilon}
\|_{\mathcal{L}(H)}\Big]\,\dd t
\\
\leq
&
Ch^{2-2\varepsilon}\int_{t_n}^{t_{n+1}}\left(1+(t_m-t)^{-1+\varepsilon}\right)e^{-c(t_m-t)}
\left(1+\sup_{s\in[0,t_m-t]}\mathbb{E}\big[\|X(s,X_{\tau,h}(t))\|_V^{2q}\big]\right)\,\dd t
\\
\leq
&
Ch^{2-2\varepsilon}\int_{t_n}^{t_{n+1}}\left(1+(t_m-t)^{-1+\varepsilon}\right)e^{-c(t_m-t)}\,\dd t.
\end{split}
\end{align}
By the same arguments of the proof of \eqref{eq:estimate-an1} and \eqref{eq:estimate-an2}, one can find that
\begin{align}\label{eq:estmate-an3-an4}
\begin{split}
|a_n^3+a_n^4+a_n^5+a_n^6|
\leq
&
C(h^{2\beta-2\varepsilon}+\tau^{\beta-\varepsilon}+\lambda_N^{-\beta+\varepsilon})
\int_{t_n}^{t_{n+1}}\left(1+(t_m-t)^{-1+\varepsilon}\right)e^{-c(t_m-t)}\,\dd t.
\end{split}
\end{align}
Putting the estimates \eqref{eq:estimate-an1}, \eqref{eq:estimate-an2} and \eqref{eq:estmate-an3-an4} back into \eqref{eq:decompose-an}
gives
\begin{align}
\begin{split}
\sum_{n=1}^{m-1} |a_n|
\leq
&
C(h^{2\beta-2\varepsilon}+\tau^{\beta-\varepsilon}+\lambda_N^{-\beta+\varepsilon})
\int_0^{t_m}\left(1+(t_m-t)^{-1+\varepsilon}\right)e^{-c(t_m-t)}\,\dd t
\\
\leq
&
C(Q,X_0,\gamma)(h^{2\beta-2\varepsilon}+\tau^{\beta-\varepsilon}+\lambda_N^{-\beta+\varepsilon}).
\end{split}
\end{align}

Next, we turn our attention to $c_n$. For this purpose, we first have
\begin{align}
\begin{split}
c_n
=
&
\mathbb{E}\int_{t_n}^{t_{n+1}}\left<(E_{\tau,h}P_h-I)F_{\tau,h}(X_{\tau,h}^n),D\mu^N(t_m-t, X_{\tau,h}(t))\right>\,\dd t
\\
&+
\mathbb{E}\int_{t_n}^{t_{n+1}}\left<F_{\tau,h}(X_{\tau,h}^n)-F(X_{\tau,h}^n),D\mu^N(t_m-t, X_{\tau,h}(t))\right>\,\dd t
\\
&
+
\mathbb{E}\int_{t_n}^{t_{n+1}}\left<F(X_{\tau,h}^n)-F(X_{\tau,h}(t)),D\mu^N(t_m-t, X_{\tau,h}(t))\right>\,\dd t
\\
&
+
\mathbb{E}\int_{t_n}^{t_{n+1}}\left<(I-P_N)F(X_{\tau,h}(t)),D\mu^N(t_m-t ,X_{\tau,h}(t))\right>\,\dd t
=:
c_n^1+c_n^2+c_n^3+c_n^4.
\end{split}
\end{align}
Let us begin with estimating $c_n^1$. Using Lemma \ref{lem:bound-Komogrolv-equation} and Theorems \ref{them:property-solution} and \ref{th:uniform-moment-bound}
yields, for any $\varepsilon \in(0,\beta)$
\begin{align}\label{eq:bound-cn1}
\begin{split}
\sum_{n=1}^{m-1} |c_n^1|
=
&
\sum_{n=1}^{m-1}\mathbb{E}\int_{t_n}^{t_{n+1}}\left|\left<(E_{\tau,h}P_h-I)F_{\tau,h}(X_{\tau,h}^n),D\mu^N(t_m-t, X_{\tau,h}(t))\right>\right|\,\dd t
\\
\leq
&
C
\sum_{n=1}^{m-1}\int_{t_n}^{t_{n+1}}\left(1+(t_m-t)^{-\beta+\varepsilon}\right)
e^{-c(t_m-t)}\|A^{-\beta+\varepsilon}(E_{\tau,h}P_h-I)\|_{\mathcal{L}(H)}
\\
&\cdot
\mathbb{E}\Big[\Big(1+\sup_{s\in[0,t_m-t]}\mathbb{E}\big[\|X(s,X_{\tau,h}(s))\|_V^{2q}\big]\Big)
\|F_{\tau,h}(X_{\tau,h}^n)\|\Big]\,\dd t
\\
\leq
&
C(h^{2\beta-2\varepsilon}+\tau^{\beta-\varepsilon})
\sum_{n=1}^{m-1}\int_{t_n}^{t_{n+1}}\left(1+(t_m-t)^{-\beta+\varepsilon}\right)e^{-c(t_m-t)}\,\dd t,
\\
\leq
&
C(h^{2\beta-2\varepsilon}+\tau^{\beta-\varepsilon}),
\end{split}
\end{align}
where in the second inequality we used the fact that
\begin{align}
\begin{split}
\|A^{-\mu}(E_{\tau,h}P_h-I)\|_{\mathcal{L}(H)}
\leq
&
\|A^{-\mu}(P_h-I)\|_{\mathcal{L}(H)}
+
\|A^{-\mu}A_h^{\mu}A_h^{-\mu}(E_{\tau,h}-I)P_h\|_{\mathcal{L}(H)}
\\
\leq
&
C(h^{2\mu}+ \tau^\mu)
\end{split}
\end{align}
for any $\mu\in[0,1]$ due to \eqref{eq:Ah-A-bound}, \eqref{eq:property-Ih} and \eqref{lem:eq-temporal-regu-ENm}.
In the same way, one can acquire for the term $c_n^2$
\begin{align}\label{eq:estimate-sum-cn2}
\begin{split}
&\sum_{n=1}^{m-1}|c_n^2|
\leq
\sum_{n=1}^{m-1}\int_{t_n}^{t_{n+1}}\mathbb{E}|\left<F_{\tau,h}(X_{\tau,h}^n)-F(X_{\tau,h}^n),D\mu^N(t_m-t, X_{\tau,h}(t))\right>|\,\dd t
\\
\leq
&
C\sum_{n=1}^{m-1}\int_{t_n}^{t_{n+1}}\mathbb{E}\Big[\big(1
+
\sup_{s\in[0,t_m-t]}\mathbb{E}\big[\|X(s,X_{\tau,h}(t))\|_V^{2q}\big]\big)
\|F(X_{\tau,h}^n)-F_{\tau,h}(X_{\tau,h}^n)\|\Big]e^{-c(t_m-t)}
\,\dd t
\\
\leq
&
C(\tau^\theta+h^\rho)\sum_{n=1}^{m-1}\int_{t_n}^{t_{n+1}}\sup_{m\in \mathbb{N}_0}
\mathbb{E}\Big[\|(X_{\tau,h}^m)^{\frac{2q-2}\alpha}F(X_{\tau,h}^m)\|\Big]e^{-c(t_m-t)}
\,\dd t
\\
\leq
&
C(\tau^\theta+h^\rho)\sum_{n=1}^{m-1}\int_{t_n}^{t_{n+1}}e^{-c(t_m-t)}\,\dd t
\\
\leq
&
C(\tau^\theta+h^\rho),
\end{split}
\end{align}
where we used  \eqref{lem:f-ftau-I} in the second inequality and Proposition \ref{pro:pro:bound-Xh-f} in the third inequality.
For the term $c_n^3$, we use \eqref{eq:continouse-verstion-Xtauh} and the Taylor expansion of $F$ to make the following decomposition:
\begin{align}
\begin{split}
c_n^3
&
=
\mathbb{E}\int_{t_n}^{t_{n+1}}(t-t_n)
\left<D\mu^N(t_m-t, X_{\tau,h}(t)),DF(X_{\tau,h}^n) A_hE_{\tau,h}X_{\tau,h}^n\right>\,\dd t
\\
&
-
\mathbb{E}\int_{t_n}^{t_{n+1}}(t-t_n)
\left<D\mu^N(t_m-t, X_{\tau,h}(t)),DF(X^n_{\tau,h}) E_{\tau,h}P_hF_{\tau,h}(X_{\tau,h}^n)\right>\,\dd t
\\
&
-
\mathbb{E}\int_{t_n}^{t_{n+1}}
\left<D\mu^N(t_m-t, X_{\tau,h}(t))
,DF(X^n_{\tau,h})\int_{t_n}^t E_{\tau,h}P_h\,\dd W(s)\right>\,\dd t
\\
&
+
\mathbb{E}\int_{t_n}^{t_{n+1}}
\left<D\mu^N(t_m-t,X_{\tau,h}(t))
, DF(X_{\tau,h}^n) R_F(X_{\tau,h}^{n},X_{\tau,h}(t))\right>\,\dd t
\\
=&
c_n^{3,1}
+
c_n^{3,2}
+
c_n^{3,3}
+
c_n^{3,4}
,
\end{split}
\end{align}
where $R_F(X_{\tau,h}^{n}, X_{\tau,h}(t))$ is a remainder term, given by
\[
R_F(X_{\tau,h}^{n},X_{\tau,h}(t))
:
=
-\int_0^1D^2F\big(X_{\tau,h}^{n}+\lambda(X_{\tau,h}(t)-X_{\tau,h}^{n})\big)
\big(X_{\tau,h}(t)-X_{\tau,h}^{n}, X_{\tau,h}(t)-X_{\tau,h}^{n}\big)(1-\lambda)\,\dd \lambda.
\]
Concerning  $c_n^{3,1}$, we first use \cite[(5.12)]{jiang2025uniform} to get, for any $\vartheta\in(0,1)$, some
$\eta\in\left(\max\{\frac d2,1\},2\right)$
\begin{align}
\|A^{-\frac\eta2}DF(u)v\|
\leq
C(1+\max\left\{\|u\|_V,\|u\|_\vartheta\right\}^{2q-2})\|v\|_{-\vartheta},
\end{align}
and then utilizing   \eqref{lem:eq-spatial-regu-ENm}, Theorem \ref{th:uniform-moment-bound} and Lemma \ref{lem:bound-Komogrolv-equation}
gives, for any $\varepsilon\in (0,\beta)$
\begin{align}\label{eq:estimate-cn31}
\begin{split}
|c_n^{3,1}|
\leq
&
C\tau \int_{t_n}^{t_{n+1}}\left(1+(t_m-t)^{-\frac{\eta}2}\right)e^{-c(t_m-t)}
\\
&
\mathbb{E} \Big[\Big(1+\sup_{s\in[0,t_m-t]}
\mathbb{E}[\|X(s, X_{\tau,h}(t))\|_V^{2q}]\Big)
\|A^{-\frac{\eta}2}DF(X_{\tau,h}^n) A_hE_{\tau,h}X_{\tau,h}^n\|\Big]\,\dd t
\\
\leq
&
C\tau \int_{t_n}^{t_{n+1}}\left(1+(t_m-t)^{-\frac{\eta}2}\right)e^{-c(t_m-t)}
\\
&
\qquad\mathbb{E}\left[\Big(1+\max\{\|X_{\tau,h}^n\|_V,\|A_h^{\frac\beta2-\varepsilon}X^n_{\tau,h}\|\}^{2q-2}\Big)
\|A_h^{-\frac{\beta}2+\varepsilon} A_hE_{\tau,h}X_{\tau,h}^n\|\right]\,\dd t
\\
\leq
&
C\tau^{\beta-\varepsilon}  \int_{t_n}^{t_{n+1}}\left(1+(t_m-t)^{-\frac{\eta}2}\right)e^{-c(t_m-t)}  \,\dd t
\sup_{m\in \mathbb{N}}\|A_h^{\frac\beta2}X_{\tau,h}^m\|_{L^2(\Omega;H)}
\\
\leq
&
C\tau^{\beta-\varepsilon} \int_{t_n}^{t_{n+1}}\left(1+(t_m-t)^{-\frac{\eta}2}\right)e^{-c(t_m-t)}  \,\dd t.
\end{split}
\end{align}
Similarly as the above estimate, the term $c_n^{3,2}$ can be estimated as follows
\begin{align}\label{eq:estimate-cn32}
\begin{split}
|c_n^{3,2}|
\leq
&
C\tau \int_{t_n}^{t_{n+1}}e^{-c(t_m-t)}
\mathbb{E} \Big[\big(1+\sup_{s\in[0, t_m-t]}
\mathbb{E}\big[\|X(s, X_{\tau,h}(t))\|_V^{2q}\big]\big)\|f'(X_{\tau,h}^n)\|_V
\\
&
\qquad
\|E_{\tau,h}P_h\|_{\mathcal{L}(H)}\|F_{\tau,h}(X_{\tau,h}^n)\|\Big]\,\dd t
\\
\leq
&
C\tau \int_{t_n}^{t_{n+1}}e^{-c(t_m-t)}\,\dd t.
\end{split}
\end{align}
In light of  \eqref{lem:eq-spatial-regu-ENm}, Lemmas
\ref{lem:eq-mallivin-derivative-numercial-solu}, \ref{lem:bound-Komogrolv-equation},  Theorem \ref{th:uniform-moment-bound}    and the Malliavin integration by parts,
it follows that
\begin{align}\label{eq:estimate-cn33}
\begin{split}
|c_n^{3,3}|
=
&
\left|\mathbb{E}\int_{t_n}^{t_{n+1}}
\int_{t_n}^t
\left<D^2\mu^N(t_m-t ,X_{\tau,h}(t))D_sX_{\tau,h}(t)
, DF(X_{\tau,h}^n) E_{\tau,h}P_h\,\right>_{\mathcal{L}_2^0}\,\dd s \dd t\right|
\\
\leq
&
\int_{t_n}^{t_{n+1}}\int_{t_n}^t\left(1+(t_m-t)^{-\frac{1-\beta}2}\right)
e^{-c(t_m-t)}
\mathbb{E}\Big[\big(1+\sup_{s\in[0,t_m-t]}
\mathbb{E}[\|X(s,X_{\tau,h}(t))\|^{8q-2}_V]\big)
\\
&
\qquad\qquad
\|A_h^{\frac{\beta-1}2}D_sX_{\tau,h}(t)\|_{\mathcal{L}_2^0}
\|DF(X_{\tau,h}^n)\cdot E_{\tau,h}P_h\|_{\mathcal{L}_2^0}\,\dd s
\,\dd t\Big]
\\
\leq
&
C\int_{t_n}^{t_{n+1}}\int_{t_n}^t(1+(t_m-t)^{\frac{\beta-1}2})
e^{-c(t_m-t)}\mathbb{E}[\|A_h^{\frac{\beta-1}2}D_sX_{\tau,h}(t)\|_{\mathcal{L}_2^0}
\\
&
\qquad\qquad
(1+\|X_{\tau,h}^n\|_V^{2q-2})]
\|A_h^{\frac{1-\beta}2}E_{\tau,h}\|_{\mathcal{L}(H)}\|A_h^{\frac{\beta-1}2}P_h\|_{\mathcal{L}_2^0}\,\dd t
\\
\leq
&
C\tau^{\frac{1+\beta}2}\int_{t_n}^{t_{n+1}}\left(1+(t_m-t)^{\frac{\beta-1}2}\right)e^{-c(t_m-t)}\,\dd t.
\end{split}
\end{align}
For the term $c_n^{3,4}$, we  first note  that
 $\|A^{-\frac\eta2}v\|\leq C\|v\|_{L^1(D)}$, for $ v\in L^1(D)$ by the embedding theorem  and some
$\eta\in\left(\max\{\frac d2,1\},2\right)$. Therefore,  by Theorem \ref{th:uniform-moment-bound} and Lemma \ref{lem:bound-Komogrolv-equation},  we arrive at
\begin{align}\label{eq:estimate-cn34}
\begin{split}
&|c_n^{3,4}|
\leq
\int_{t_n}^{t_{n+1}}\left(1+(t_m-t)^{-\frac\eta2}\right)e^{-c(t_m-t)}
\mathbb{E}\Big[\Big(1+\sup_{s\in[0,t_m-t]}\mathbb{E}[\|X(s,X_{\tau,h}(t))\|_V^{2q}]\Big)
\\
&
\qquad \cdot
\|A^{-\frac\eta2}DF(X_{\tau,h}^n) R_F(X_{\tau,h}^{n},X_{\tau,h}(t))\|\Big]\,\dd t
\\
\leq
&
C\int_{t_n}^{t_{n+1}}\left(1+(t_m-t)^{-\frac\eta2}\right)e^{-c(t_m-t)}
\mathbb{E}\Big[
\|DF(X_{\tau,h}^n) R_F(X_{\tau,h}^{n},X_{\tau,h}(t))\|_{L^1(D)}\Big]\,\dd t
\\
\leq
&
C\int_{t_n}^{t_{n+1}}\left(1+(t_m-t)^{-\frac\eta2}\right)e^{-c(t_m-t)}
\mathbb{E}\Big[(1+\|X_{\tau,h}^n\|_V^{2q-2}+\|X_{\tau,h}(t)\|_V^{2q-2})
\|X_{\tau,h}^{n}-X_{\tau,h}(t)\|^2\Big]\,\dd t
\\
\leq
&C\tau^\beta \int_{t_n}^{t_{n+1}}\left(1+(t_m-t)^{-\frac\eta2}\right)e^{-c(t_m-t)}\,\dd t.
\end{split}
\end{align}
Gathering \eqref{eq:estimate-cn31}-\eqref{eq:estimate-cn34}, we conclude that
for any $\varepsilon \in(0,\beta)$ and some $\eta\in \left(\max\{\frac d 2,1\},2\right)$,
\begin{align}\label{eq:bound-sum-cn3}
\begin{split}
\sum_{n=1}^{m-1}| c_n^3|
\leq
C\tau^{\beta-\varepsilon}\sum_{n=1}^{m-1}\int_{t_n}^{t_{n+1}}\left(1+(t_m-t)^{\frac{\beta-1}2}+(t_m-t)^{-\frac{\eta}2}\right)e^{-c(t_m-t)}\,\dd t
\leq
C\tau^{\beta-\varepsilon}.
\end{split}
\end{align}
At the moment we follow similar arguments used in the proof of \eqref{eq:bound-cn1}
to derive, for any $\varepsilon \in(0,\beta)$
\begin{align}\label{eq:bound-sum-cn4}
\begin{split}
\sum_{n=1}^{m-1}|c_n^4|
\leq
C\lambda_N^{-\beta+\varepsilon}
\sum_{n=1}^{m-1}\int_{t_n}^{t_{n+1}}\left(1+(t_m-t)^{-\beta+\varepsilon}\right)e^{-c(t_m-t)}\,\dd t
\leq
C\lambda_N^{-\beta+\varepsilon}.
\end{split}
\end{align}
From the estimates \eqref{eq:bound-cn1}, \eqref{eq:estimate-sum-cn2}, \eqref{eq:bound-sum-cn3} and \eqref{eq:bound-sum-cn4},
it follows that for any $\varepsilon\in (0, \beta)$,
\begin{align}
\sum_{n=1}^{m-1}|c_n|
\leq
C(\tau^{\min\{\theta, \beta-\varepsilon\}}+ h^{\min\{\rho, 2\beta-2\varepsilon\}}+\lambda_N^{-\beta+\varepsilon}).
\end{align}
Now we split $b_n$ as $b_n=b_n^1+b_n^2+b_n^3+b_n^4$, where
\begin{align}
b_n^1
& :=
\mathbb{E}\int_{t_n}^{t_{n+1}}\left<(I-E_{\tau,h})A_hX_{\tau,h}^n
,
D\mu^N(t_m-t,X_{\tau,h}(t)\right>\,\dd t,
\\
b_n^2
& :=\mathbb{E}\int_{t_n}^{t_{n+1}}\left<A_h(X_{\tau,h}(t)-X_{\tau,h}^n)
,
D\mu^N(t_m-t,X_{\tau,h}(t)\right>\,\dd t,
\\
b_n^3
& :=
\mathbb{E}\int_{t_n}^{t_{n+1}}\left<AX_{\tau,h}(t)-A_hX_{\tau,h}(t)
,
D\mu^N(t_m-t, X_{\tau,h}(t)\right>\,\dd t,
\\
b_n^4
& :=
\mathbb{E}\int_{t_n}^{t_{n+1}}\left<(P_N-I)AX_{\tau,h}(t)
,
D\mu^N(t_m-t,X_{\tau,h}(t)\right>\,\dd t.
\end{align}
To bound the term $b_n^1$, we make a further decomposition as follows:
\begin{align}
\begin{split}
b_n^1
=
&
\mathbb{E}\int_{t_n}^{t_{n+1}}\left<(I-E_{\tau,h})A_h E_{\tau,h}^nP_hX_0,D\mu^N(t_m-t, X_{\tau,h}(t))\right>\,\dd t
\\
&
+
\mathbb{E}\int_{t_n}^{t_{n+1}}\tau \sum_{j=0}^{n-1}\left<(I-E_{\tau,h})A_h E_{\tau,h}^{n-j}P_hF_{\tau,h}(X_{\tau,h}^j),D\mu^N(t_m-t, X_{\tau,h}(t))\right>\,\dd t
\\
&
+
\mathbb{E}\int_{t_n}^{t_{n+1}}\left<(I-E_{\tau,h})A_h \sum_{j=1}^{n }\int_{t_
{j-1}}^{t_{j}}E_{\tau,h}^{n-j+1}P_h\,\dd W(s),D\mu^N(t_m-t,X_{\tau,h}(t))\right>\,\dd t
\\
=:
&
b_{n}^{1,1}
+
b_n^{1,2}
+
b_n^{1,3}.
\end{split}
\end{align}
The first term $b_n^{1,1}$ is similar to $c_n^1$ above and is treated in the same way:
\begin{align}
\begin{split}
|b_{n}^{1,1}|
\leq
&
\int_{t_n}^{t_{n+1}}|\mathbb{E}\left<(I-E_{\tau,h})A_h E_{\tau,h}^nP_hX_0,D\mu^N(t_m-t,X_{\tau,h}(t))\right>|\,\dd t
\\
\leq
&
C\int_{t_n}^{t_{n+1}}\left(1+(t_m-t)^{-1+\frac\varepsilon2}\right)e^{-c(t_m-t)}
\mathbb{E}\Big[\big(1+\sup_{s\in[0,t_m-t]}\mathbb{E}[\|X(s,X_{\tau,h}(t))\|_V^{2q}]\big)
\\
&
\quad\|A^{-1+\frac\varepsilon2}A_h^{1-\frac\varepsilon2}P_h\|_{\mathcal{L}(H)}
\|A_h^{\frac\varepsilon2}(I-E_{\tau,h}) E_{\tau,h}^nP_hX_0\|\Big]\,\dd t
\\
\leq
&
C\int_{t_n}^{t_{n+1}}\left(1+(t_m-t)^{-1+\frac\varepsilon2}\right)e^{-c(t_m-t)}
\|A_h^{-\beta+\varepsilon}(I-E_{\tau,h})A_h^{\beta-\frac\varepsilon2} E_{\tau,h}^nP_hX_0\|\,\dd t
\\
\leq
&
 C\tau^{\beta-\varepsilon}\int_{t_n}^{t_{n+1}}\left(1+(t_m-t)^{-1+\frac\varepsilon2}\right)e^{-c(t_m-t)}
\min\left\{t_n^{-\beta+\frac\varepsilon2},t_n^{-2}\right\}\,\dd t
 \\
 \leq
 &
  C\tau^{\beta-\varepsilon}\int_{t_n}^{t_{n+1}}\left(1+(t_m-t)^{-1+\frac\varepsilon2}\right)e^{-c(t_m-t)}
  \min\left\{t^{-\beta+\frac\varepsilon2},t^{-2}\right\}\,\dd t,
\end{split}
\end{align}
where we utilized the fact $\frac 1{t_n}\leq \frac 2 t$, for $t\in(t_n,t_{n+1}), n=1,2,\cdots$ in the last inequality.
Similarly as above,  we have by \eqref{eq:sum}
\begin{align}
\begin{split}
|b_{n}^{1,2}|
\leq
&
\int_{t_n}^{t_{n+1}}\tau \sum_{j=0}^{n-1}\left|\mathbb{E}\left<(I-E_{\tau,h})A_h E_{\tau,h}^{n-j}P_hF_{\tau,h}(X_{\tau,h}^j),D\mu^N(t_m-t, X_{\tau,h}(t))\right>\right|\,\dd t
\\
\leq
&
C\int_{t_n}^{t_{n+1}}\left(1+(t_m-t)^{-1+\frac\varepsilon2}\right)e^{-c(t_m-t)}
\tau \sum_{j=0}^{n-1}\mathbb{E}\big[\|A_h^{-1+\varepsilon}(I-E_{\tau,h})A_h^{1-\frac \varepsilon2} E_{\tau,h}^{n-j}P_hF_{\tau,h}(X_{\tau,h}^j)\|\big]\,\dd t
\\
\leq
&
C\tau^{1-\varepsilon}\int_{t_n}^{t_{n+1}}\left(1+(t_m-t)^{-1+\frac\varepsilon2}\right)e^{-c(t_m-t)}\tau \sum_{j=0}^{n-1}
\min\{t_{n-j}^{-1+\frac\varepsilon2},t_{n-j}^{-2}\}\,\dd t
\\
\leq
&
C\tau^{1-\varepsilon}\int_{t_n}^{t_{n+1}}\left(1+(t_m-t)^{-1+\frac\varepsilon2}\right)e^{-c(t_m-t)}\,\dd t.
\end{split}
\end{align}
For the term $b_n^{1,3}$,  we rely on Malliavin integration by parts, Lemmas \ref{lem:eq-smooth-ENm}, \ref{lem:eq-mallivin-derivative-numercial-solu}, \ref{lem:bound-Komogrolv-equation}
and Theorem \ref{th:uniform-moment-bound} to obtain
\begin{align}\label{eq:bound-b_n13}
\begin{split}
|b_n^{1,3}|
\leq
&
\int_{t_n}^{t_{n+1}}\sum_{j=1}^{n}\int_{t_{j-1}}^{t_{j}}\left|\mathbb{E}\left<(I-E_{\tau,h})A_h E_{\tau,h}^{n-j+1}P_h ,D^2\mu^N(t_m-t,X_{\tau,h}(t))D_sX_{\tau,h}(t)\right>_{\mathcal{L}_2^0}\right|\,\dd s\,\dd t
\\
\leq
&
\int_{t_n}^{t_{n+1}}\sum_{j=1}^{n }\int_{t_{j-1}}^{t_{j}} \left(1+(t_m-t)^{-1+\frac\varepsilon2}\right)e^{-c(t_m-t)}\|A_h^{-\frac{1+\beta-\varepsilon}2}
(I-E_{\tau,h})A_h E_{\tau,h}^{n-j+1}P_h\|_{\mathcal{L}_2^0}
\\
&
\mathbb{E}\Big[(1+\sup_{s\in[0,t_m-t]}\mathbb{E}[\|X(s, X_{\tau,h}(t))\|_V^{8q-2}])
\|A^{\frac{\beta-1}2}D_sX_{\tau,h}(t)\|_{\mathcal{L}_2^0}\Big]\,\dd s\,\dd t
\\
\leq
&
\int_{t_n}^{t_{n+1}}\sum_{j=1}^{n }\int_{t_{j-1}}^{t_{j}} (1+(t_m-t)^{-1+\frac\varepsilon2})e^{-c(t_m-t)}
\\
&
\cdot
\|A_h^{-\beta+\varepsilon}(E_{\tau,h}-I)P_h\|_{\mathcal{L}(H)}\| A_h^{1-\frac\varepsilon2}E_{\tau,h}^{n-j+1}\|_{\mathcal{L}(H)}\|A_h^{\frac{\beta-1}2}P_h\|_{\mathcal{L}_2^0}
\,\dd s\,\dd t
\\
\leq
&
C\tau^{\beta-\varepsilon}\int_{t_n}^{t_{n+1}} \left(1+(t_m-t)^{-1+\frac\varepsilon2}\right)e^{-c(t_m-t)}\dd t
\tau\sum_{j=1}^{n }\min\{t_{n-j+1}^{-1+\frac\varepsilon2},t_{n-j+1}^{-2}\}
\\
\leq
&
C\tau^{\beta-\varepsilon}\int_{t_n}^{t_{n+1}}\left(1+(t_m-t)^{-1+\frac\varepsilon2}\right)e^{-c(t_m-t)}\,\dd t.
\end{split}
\end{align}
We conclude this step by gathering the previous estimates. This enables us to write
\begin{align}\label{eq:bound-bn1}
\begin{split}
\sum_{n=1}^{m-1}|b_n^{1}|
\leq
C\tau^{\beta-\varepsilon}\sum_{n=1}^{m-1}\int_{t_n}^{t_{n+1}}
\left(1+(t_m-t)^{-1+\frac\varepsilon2}\right)(1+\min\{t^{-1+\frac\varepsilon2}, t^{-2}\})e^{-c(t_m-t)}\,\dd t
\leq
C\tau^{\beta-\varepsilon}.
\end{split}
\end{align}

Recalling  the definition $X_{\tau,h}(t)$ in \eqref{eq:continouse-verstion-Xtauh},
the term $b_n^2$ is divided into the following three  terms
\begin{align}
\begin{split}
b_n^2
=&
\mathbb{E}\int_{t_n}^{t_{n+1}}(t-t_n)\left<-A^2_hE_{\tau,h}X_{\tau,h}^n,D\mu^N(t_m-t,X_{\tau,h}(t)\right>\,\dd t
\\
&
+
\mathbb{E}\int_{t_n}^{t_{n+1}}(t-t_n)\left<A_hE_{\tau,h}F_{\tau,h}(X_{\tau,h}^n),D\mu^N(t_m-t,X_{\tau,h}(t)\right>\,\dd t
\\
&
+
\mathbb{E}\int_{t_n}^{t_{n+1}}\left<A_h\int_{t_n}^tE_{\tau,h}P_h\dd W(s),D\mu^N(t_m-t,X_{\tau,h}(t)\right>\,\dd t
\\
=:
&
b_n^{2,1}
+
b_n^{2,2}
+
b_n^{2,3}.
\end{split}
\end{align}
To deal with the term $b_n^{2,1}$, we make a further decomposition:
\begin{align}\label{eq:decomposition-bn21}
\begin{split}
b_n^{2,1}
=
&
\mathbb{E}\int_{t_n}^{t_{n+1}}(t-t_n)\left<-A^2_hE_{\tau,h}E_{\tau,h}^nP_hX_0,
\,D\mu^N(t_m-t, X_{\tau,h}(t)\right>\,\dd t
\\
&
+
\mathbb{E}\int_{t_n}^{t_{n+1}}(t-t_n)\left<-A^2_hE_{\tau,h}
\tau\sum_{j=0}^{n-1}E_{\tau,h}^{n-j}F_{\tau,h}(X_{\tau,h}^j)
\,D\mu^N(t_m-t,X_{\tau,h}(t)\right>\,\dd t
\\
&
+
\mathbb{E}\int_{t_n}^{t_{n+1}}(t-t_n)\left<-A^2_h\sum_{j=1}^{n}E_{\tau,h}^{n-j+1}P_h \Delta W_j,
\,D\mu^N(t_m-t, X_{\tau,h}(t)\right>\,\dd t
\\
=
&
b_n^{2,1,1}+b_n^{2,1,2}+b_n^{2,1,3}.
\end{split}
\end{align}
Owing to  the smoothing effect of $E_{\tau,h}$ in Lemma \ref{lem:eq-smooth-ENm}, the Malliavin regularity in Lemma \ref{lem:bound-Komogrolv-equation}, and the a priori estimate of $X_{\tau,h}(t)$,
we have
\begin{align}\label{eq:bound-bn211}
\begin{split}
|b_n^{2,1,1}|
\leq
&
\int_{t_n}^{t_{n+1}}(t-t_n)
\left|\mathbb{E}\left<A^{-1+\frac\varepsilon2}A_h^2E_{\tau,h}E_{\tau,h}^nP_hX_0,
A^{1-\frac\varepsilon2}D\mu^N(t_m-t, X_{\tau,h}(t))\right>\right|\,\dd t
\\
\leq
&
C\tau\int_{t_n}^{t_{n+1}}
\mathbb{E}\Big[\left(1+(t_m-t)^{-1+\frac\varepsilon2}\right)e^{-c(t_m-t)}
\Big(1+\sup_{s\in[0,t_m-t]}\mathbb{E}\big[\|X(s,X_{\tau,h}(t))\|_V^{2q}\big]\Big)
\\
&
\qquad
\|A^{-1+\frac\varepsilon2}A_h^{1-\frac\varepsilon2}P_h\|_{\mathcal{L}(H)}
\Big\|A_h^{\varepsilon}E_{\tau,h}\|_{\mathcal{L}(H)}\|A_h^{1-\frac\varepsilon2}E_{\tau,h}^nP_hX_0\Big\|\Big]\,\dd t
\\
\leq
&
C\tau^{1-\varepsilon}\int_{t_n}^{t_{n+1}}\left(1+(t_m-t)^{-1+\frac\varepsilon2}\right)e^{-c(t_m-t)}
\min\left\{t_n^{-1+\frac\varepsilon2},t_n^{-2}\right\}
\,\dd t
\\
\leq
&
C\tau^{1-\varepsilon}\int_{t_n}^{t_{n+1}}
\left(1+(t_m-t)^{-1+\frac\varepsilon2}\right)e^{-c(t_m-t)}\min\left\{t^{-1+\frac\varepsilon2},t^{-2}\right\}
\,\dd t
,
\end{split}
\end{align}
for $n=1,2,\cdots, m-1$.
Similarly as above
\begin{align}\label{eq:bound-bn212}
\begin{split}
|b_n^{2,1,2}|
\leq
&
\tau\int_{t_n}^{t_{n+1}}\mathbb{E}\Big[\Big\|A^{-1+\frac\varepsilon2}A_h^2E_{\tau,h}
\tau\sum_{j=0}^{n-1}E_{\tau,h}^{n-j}F_{\tau,h}(X_{\tau,h}^j)\Big\|
\|A^{1-\frac\varepsilon2}D\mu^N(t_m-t, X_{\tau,h}(t))\|\Big]\,\dd t
\\
\leq
&
C\tau\int_{t_n}^{t_{n+1}}\left(1+(t_m-t)^{-1+\frac\varepsilon2}\right)e^{-c(t_m-t)}
\mathbb{E}\Big[\Big(1
+
\sup_{s\in[0,t_m-t]}\mathbb{E}\big[\|X(s,X_{\tau,h}(t))\|_V^{2q}\big]\Big)
\\
&
\|A^{-1+\frac\varepsilon2}A_h^{1-\frac\varepsilon2}P_h\|_{\mathcal{L}(H)}
\|A_h^{\varepsilon}E_{\tau,h}\|_{\mathcal{L}(H)}\tau\sum_{j=0}^{n-1}
\|A_h^{1-\frac\varepsilon2}E_{\tau,h}^{n-j}P_h\|_{\mathcal{L}(H)}\|F_{\tau,h}(X_{\tau,h}^j)\|\Big]\,\dd t
\\
\leq
&
C\tau^{1-\varepsilon}\int_{t_n}^{t_{n+1}}\left(1+(t_m-t)^{-1+\frac\varepsilon2}\right)e^{-c(t_m-t)}
\tau\sum_{j=0}^{n-1}\min\left\{t_{n-j}^{-1+\frac\varepsilon2},t_{n-j}^{-2}\right\}
\,\dd t
\\
\leq
&
C\tau^{1-\varepsilon}\int_{t_n}^{t_{n+1}}
\left(1+(t_m-t)^{-1+\frac\varepsilon2}\right)e^{-c(t_m-t)}
\,\dd t
.
\end{split}
\end{align}
Following similar arguments as used in the proof of \eqref{eq:bound-b_n13}, we obtain
\begin{align}\label{eq:bund-bn213}
\begin{split}
&|b_n^{2,1,3}|
\\
&
\leq
\tau\int_{t_n}^{t_{n+1}} \Big|\mathbb{E}
\Big[\sum_{j=1}^{n}\int_{t_{j-1}}^{t_{j}}\left<A^2_hE_{\tau,h}^{n-j+1}P_h,
D^2\mu^N(t_m-t,X_{\tau,h}(t))D_sX_{\tau,h}(t)\right>_{\mathcal{L}_2^0}\Big]\,\dd s\Big|
\,\dd t
\\
&\leq
C\tau\int_{t_n}^{t_{n+1}} \sum_{j=1}^{n}\int_{t_{j-1}}^{t_{j}}
\left(1+(t_m-t)^{-1+\frac\varepsilon2}\right)e^{-c(t_m-t)}\mathbb{E}\Big[
\Big(1+\sup_{s\in[0,t_m-t]}\mathbb{E}\big[\|X(s,X_{\tau,h}(t))\|_V^{8q-2}\big]\Big)
\\
&\qquad
\|A^{-\frac{1+\beta-\varepsilon}2}A_h^{\frac{1+\beta-\varepsilon}2}P_h\|_{\mathcal{L}(H)}\| A_h^{1-\frac\varepsilon2}E_{\tau,h}^{n-j+1}P_h\|_{\mathcal{L}(H)}\|A_h^{1-\beta+\varepsilon}E_{\tau,h}\|_{\mathcal{L}(H)}
\|A_h^{\frac{\beta-1}2}P_h\|_{\mathcal{L}_2^0}
\|A^{\frac{\beta-1}2}D_sX_{\tau,h}(t)\|_{\mathcal{L}_2^0}
\Big]\,\dd s\,\dd t
\\
&\leq
C\tau \tau^{-1+\beta-\varepsilon}\int_{t_n}^{t_{n+1}}\left(1+(t_m-t)^{-1+\frac\varepsilon2}\right)e^{-c(t_m-t)}
\sum_{j=1}^{n}\tau\min\left\{t_{n-j+1}^{-1+\frac\varepsilon2}, t_{n-j+1}^{-2}\right\}\,\dd t
\\
&\leq
C\tau^{\beta-\varepsilon}\int_{t_n}^{t_{n+1}}\left(1+(t_m-t)^{-1+\frac\varepsilon2}\right)e^{-c(t_m-t)}\,\dd t.
\end{split}
\end{align}
Thus, inserting \eqref{eq:bound-bn211}, \eqref{eq:bound-bn212}  and \eqref{eq:bund-bn213} into \eqref{eq:decomposition-bn21} shows
\begin{align}
|b_n^{2,1}|
\leq
C\tau^{\beta-\varepsilon} \int_{t_n}^{t_{n+1}} \left(1+(t_m-t)^{-1+\frac\varepsilon2}\right)\left(1+\min\left\{t^{-1+\varepsilon}, t^{-2}\right\}\right)e^{-c(t_m-t)}\,\dd t.
\end{align}
Next we treat easy terms $b_n^{2,2}$ and $b_n^{2,3}$. Following similar arguments as used in the proof of \eqref{eq:estimate-cn32} and
thanks to  Lemmas \ref{lem:bound-Komogrolv-equation}, \ref{lem:eq-smooth-ENm} and Theorems \ref{th:uniform-moment-bound}, \ref{them:property-solution}, one can show
\begin{align}
\begin{split}
|b_n^{2,2}|
\leq
&
\int_{t_n}^{t_{n+1}}\left|(t-t_n)\mathbb{E}\left<A_hE_{\tau,h}F_{\tau,h}(X_{\tau,h}^n)
,
D\mu^N(t_m-t,X_{\tau,h}(t))\right>\right|\,\dd t
\\
\leq
&
C\tau\int_{t_n}^{t_{n+1}}(1+(t_m-t)^{-1+\varepsilon})e^{-c(t_m-t)}\|A^{-1+\varepsilon}A_h^{1-\varepsilon}P_h\|_{\mathcal{L}(H)}
\\
&
\qquad
\mathbb{E}\Big[(1+\sup_{s\in[0,t_m-t]}\mathbb{E}\big[\|X(s,X_{\tau,h}(t))\|^{2q}_V])
\|A_h^{\varepsilon}E_{\tau,h}F_{\tau,h}(X_{\tau,h}^n)\|\Big]
\,\dd t
\\
\leq
&
C\tau^{1-\varepsilon}\mathbb{E}\int_{t_n}^{t_{n+1}}(1+(t_m-t)^{-1+\varepsilon})e^{-c(t_m-t)}
\mathbb{E}[\|F_{\tau,h}(X_{\tau,h}^n)\|]\,\dd t\\
\leq
&
C\tau^{1-\varepsilon}\int_{t_n}^{t_{n+1}}(1+(t_m-t)^{-1+\varepsilon})e^{-c(t_m-t)}\,\dd t.
\end{split}
\end{align}
The estimation of $b_n^{2,3}$ is similar to that of $c_n^{2,3}$:
\begin{align}
\begin{split}
|b_n^{2,3}|
=
&
\int_{t_n}^{t_{n+1}}\int_{t_n}^t\left|\mathbb{E}\left<A_hE_{\tau,h}P_h,
D^2\mu^N(t_m-t,X_{\tau,h}(t))D_sX_{\tau,h}(t)\right>_{\mathcal{L}_2^0}\right|\,\dd t
\\
\leq
&
C\int_{t_n}^{t_{n+1}}\int_{t_n}^t(1+(t_m-t)^{-1+\varepsilon})e^{-c(t_m-t)}
\mathbb{E}\Big[(1+\sup_{s\in[0,t_m-t]}\mathbb{E}\big[\|X(s,X_{\tau,h}(t))\|_V^{8q-2}\big]
\\
&
\qquad
\|A^{-\frac{1+\beta}2+\varepsilon}A_h^{\frac{1+\beta}2-\varepsilon}P_h\|_{\mathcal{L}(H)}
\|A_h^{-\frac{1+\beta}2+\varepsilon}A_hE_{\tau,h}\|_{\mathcal{L}_2^0}
\|A^{\frac{\beta-1}2}D_sX_{\tau,h}(t)\|_{\mathcal{L}_2^0}\Big]\,\dd s\,\dd t
\\
\leq
&
C\tau\int_{t_n}^{t_{n+1}}\left(1+(t_m-t)^{-1+\varepsilon}\right)e^{-c(t_m-t)}
\|A^{1-\beta+\varepsilon}_hE_{\tau,h}\|_{\mathcal{L}(H)}\|A_h^{\frac{\beta-1}2}P_h\|_{\mathcal{L}_2^0}
\,\dd t
\\
\leq
&
C\tau^{\beta-\varepsilon}\int_{t_n}^{t_{n+1}}\left(1+(t_m-t)^{-1+\varepsilon}\right)e^{-c(t_m-t)}\,\dd t.
\end{split}
\end{align}
Combining the above estimates, we  conclude
\begin{align}\label{eq:bound-bn2}
\sum_{n=1}^{m-1}|b_n^2|
\leq
C\tau^{\beta-\varepsilon}\sum_{n=1}^{m-1}\int_{t_n}^{t_{n+1}}(1+(t_m-t)^{-1+\varepsilon})
(1+\min\{t^{-1+\varepsilon},t^{-2}\})e^{-c(t_m-t)}\,\dd t
\leq
C\tau^{\beta-\varepsilon}.
\end{align}
Regarding the term $b_n^3$,   we  use the relation $A_hR_h=P_hA$ to write
\begin{align}
\begin{split}
&
\left<AX_{\tau,h}(t)-A_hX_{\tau,h}(t),
D\mu^N(t_m-t,X_{\tau,h}(t))\right>
\\
=
&
\left<X_{\tau,h}(t),(P_hA-A_hP_h)D\mu(t_m-t,X_{\tau,h}(t))\right>
\\
=
&
\left<X_{\tau,h}(t),A_h(R_h-P_h)D\mu^N(t_m-t,X_{\tau,h}(t))\right>,
\end{split}
\end{align}
which together with the definition of $X_{\tau,h}(t)$ in \eqref{eq:continouse-verstion-Xtauh} yields the following decomposition:
\begin{align}\label{eq:decompose-bn3}
\begin{split}
b_n^3
=
&
\mathbb{E}\int_{t_n}^{t_{n+1}}\left<X_{\tau,h}(t),A_h(R_h-P_h)D\mu^N(t_m-t,\widetilde{X}_{\tau,h}(t))\right>\,\dd t
\\
=
&
\mathbb{E}\int_{t_n}^{t_{n+1}}\Big<(t-t_n)E_{\tau,h}P_hF_{\tau,h}(X_{\tau,h}^n)
,A_h(R_h-P_h)D\mu^N(t_m-t,X_{\tau,h}(t))\Big>\,\dd t
\\
&
+
\mathbb{E}\int_{t_n}^{t_{n+1}}\Big<
\int_{t_n}^tE_{\tau,h}P_h\,\dd W(s),A_h(R_h-P_h)D\mu^N(t_m-t,X_{\tau,h}(t))\Big>\,\dd t
\\
&
+
\mathbb{E}\int_{t_n}^{t_{n+1}}\Big<(1-(t-t_n)A_hE_{\tau,h})E_{\tau,h}^nP_hX_0, A_h(R_h-P_h)D\mu^N(t_m-t,
X_{\tau,h}(t))\Big>\,\dd t
\\
&
+
\mathbb{E}\int_{t_n}^{t_{n+1}}\Big<(1-(t-t_n)A_hE_{\tau,h})\sum_{j=0}^{n-1}\tau E_{\tau,h}^{n-j}P_hF_{\tau,h}(X_{\tau,h}^j),
A_h(R_h-P_h)D\mu^N(t_m-t,X_{\tau,h}(t))\Big>\,\dd t
\\
&
+
\mathbb{E}\int_{t_n}^{t_{n+1}}\Big<(1-(t-t_n)A_hE_{\tau,h})\sum_{j=0}^{n-1}\tau \int_{t_j}^{t_{j+1}}E_{\tau,h}^{n-j}P_h\,\dd W(s),
A_h(R_h-P_h)D\mu^N(t_m-t,X_{\tau,h}(t))\Big>\,\dd t
\\
=:
&b_n^{3,1}
+
b_n^{3,2}
+
b_n^{3,3}
+
b_n^{3,4}
+
b_n^{3,5}.
\end{split}
\end{align}
Next, we will bound $b_n^{3,i}, i=1,2,3,4,5$, separately.
  From \eqref{eq:property-Ih}, the smoothing property  of $E_{\tau,h}$ in Lemma \ref{lem:eq-smooth-ENm}, Malliavin regularity in Lemma \ref{lem:bound-Komogrolv-equation} and the a priori estimate of $X_{\tau,h}(t)$, it follows that for any  $\varepsilon\in (0, \beta)$
\begin{align}\label{eq:bound-bn31}
\begin{split}
|b_n^{3,1}|
=
&
\int_{t_n}^{t_{n+1}}(t-t_n)\left|\mathbb{E}\Big<A_hE_{\tau,h}P_hF_{\tau,h}(X_{\tau,h}^n)
,(R_h-P_h)A^{-1+\varepsilon}A^{1-\varepsilon}D\mu^N(t_m-t,X_{\tau,h}(t))\Big>\right|\,\dd t
\\
\leq
&
C\tau\int_{t_n}^{t_{n+1}}(1+(t_m-t)^{-1+\varepsilon})e^{-c(t_m-t)}
\mathbb{E}\Big[(1+\sup_{s\in[0,t_m-t]}\mathbb{E}[\|X(s,X_{\tau,h}(t))\|_V^{2q}])
\\
&
\qquad
\|A_hE_{\tau,h}P_h\|_{\mathcal{L}(H)}
\|F_{\tau,h}(X_{\tau,h}^n)\|
\|(P_h-R_h)A^{-1+\varepsilon}\|_{\mathcal{L}(H)}\Big]\,\dd t
\\
\leq
&
Ch^{2-2\varepsilon}\int_{t_n}^{t_{n+1}}(1+(t_m-t)^{-1+\varepsilon})e^{-c(t_m-t)}\,\dd t.
\end{split}
\end{align}
Utilizing \eqref{eq:property-R_h},  Malliavin calculus integration by parts, the regularity of $\mu^N(\cdot,\cdot)$, Malliavin differentiability of $X_{\tau,h}(t)$ and the a priori estimate of $X_{\tau,h}(t)$, one can observe that for any $\varepsilon\in (0,\beta)$
\begin{align}\label{eq:bound-bn32}
\begin{split}
|b_n^{3,2}|
&\leq
\int_{t_n}^{t_{n+1}}\int_{t_n}^t\Big|\mathbb{E}\Big<A^{\frac{\beta-1}2}A_h^{\frac{1-\beta}2}A_hE_{\tau,h}A_h^{\frac{\beta-1}2}P_h,
\\
&
\qquad
A^{\frac{1-\beta}2}(P_h-R_h)A^{-\frac{1+\beta}2+\varepsilon}A^{\frac{1+\beta}2-\varepsilon}
D^2\mu^N(t_m-t,X_{\tau,h}(t))A^{\frac{1-\beta}2}A^{\frac{\beta-1}2}D_sX_{\tau,h}(t)\Big>_{\mathcal{L}_2^0}\Big|\,\dd s\,\dd t
\\
&\leq
\int_{t_n}^{t_{n+1}}\int_{t_n}^t\left(1+(t_m-t)^{-1+\varepsilon}\right)e^{-c(t_m-t)}
\mathbb{E}\Big[(1+\sup_{s\in[0,t_m-t]}\mathbb{E}[\|X(s,X_{\tau,h}(t))\|_V^{8q-2}])
\|A_hE_{\tau,h}\|_{\mathcal{L}(H)}
\\
&
\qquad
\|A^{\frac{\beta-1}2}A_h^{\frac{\beta-1}2}P_h\|_{\mathcal{L}(H)}\|A_h^{\frac{\beta-1}2}P_h\|_{\mathcal{L}_2^0}
\|A^{\frac{1-\beta}2}(P_h-R_h)A^{-\frac{1+\beta}2+\varepsilon}\|_{\mathcal{L}(H)}
\|A^{\frac{\beta-1}2}D_sX_{\tau,h}(t)\|_{\mathcal{L}_2^0}\Big]\,\dd s\,\dd t
\\
&\leq
Ch^{2\beta-2\varepsilon}\int_{t_n}^{t_{n+1}}\left(1+(t_m-t)^{-1+\varepsilon}\right)e^{-c(t_m-t)}\,\dd t.
\end{split}
\end{align}
Then following similar arguments as used in the estimation of $b_n^{3,1}$, $b_n^{2,1,1}$ and $b_n^{2,1,2}$ yields
\begin{align}\label{eq:bound-bn33-bn34}
|b_n^{3,3}|
+
|b_n^{3,4}|
\leq
Ch^{2\beta-2\varepsilon}
\int_{t_n}^{t_{n+1}} \left(1+(t_m-t)^{-1+\varepsilon}+(t_m-t)^{-1+\frac\varepsilon2}\right)\min\{t^{-1+\frac\varepsilon2},t^{-2}\}e^{-c(t_m-t)}\,\dd t.
\end{align}
Similar arguments as used in estimating $b_n^{3,2}$ implies
\begin{align}\label{eq:bound-bn35}
\begin{split}
b_n^{3,5}
\leq
&
\int_{t_n}^{t_{n+1}}\sum_{j=0}^{n-1} \int_{t_j}^{t_{j+1}}\Big|\mathbb{E}\Big<A_h^{\frac{\beta+1-\varepsilon}2}
(I-(t-t_n)A_hE_{\tau,h})E_{\tau,h}^{n-j}A_h^{\frac{1-\beta}2}
A_h^{\frac{\beta-1}2}P_h,
\\
&\quad
A_h^{\frac{1-\beta+\varepsilon}2}(R_h-P_h)A^{-\frac{1+\beta-\varepsilon}2}
A^{\frac{1+\beta-\varepsilon}2}D^2\mu^N(t_m-t,X_{\tau,h}(t))A^{\frac{1-\beta}2}
A^{\frac{\beta-1}2}D_sX_{\tau,h}(t)\Big>_{\mathcal{L}_2^0}\,\dd s\Big|\,\dd t
\\
\leq
&
Ch^{2\beta-2\varepsilon}\int_{t_n}^{t_{n+1}}\sum_{j=0}^{n-1}\tau \min\{t_{n-j}^{-1+\frac\varepsilon2},t_{n-j}^{-2}\}\left(1+(t_m-t)^{-1+\frac{\varepsilon}2}\right)e^{-c(t_m-t)}\,\dd t
\\
\leq
&Ch^{2\beta-2\varepsilon}\int_{t_n}^{t_{n+1}}\left(1+(t_m-t)^{-1+\frac\varepsilon2}\right)e^{-c(t_m-t)}\,\dd t.
\end{split}
\end{align}
Thanks to \eqref{eq:bound-bn31}-\eqref{eq:bound-bn35} and \eqref{eq:decompose-bn3}, one can deduce
\begin{align}\label{eq:bound-bn3-sum}
\begin{split}
\sum_{n=1}^{m-1}|b_n^3|
\leq
&
Ch^{2\beta-2\varepsilon}
\sum_{n=1}^{m-1}\int_{t_n}^{t_{n+1}} \left(1+(t_m-t)^{-1+\varepsilon}+(t_m-t)^{-1+\frac\varepsilon2}\right)(1+\min\{t^{-1+\frac\varepsilon2},t^{-2}\})e^{-c(t_m-t)}\,\dd t
\\
\leq
&
Ch^{2\beta-2\varepsilon}.
\end{split}
\end{align}
Employing similar techniques used in the proof of \eqref{eq:bound-bn31}, we show
\begin{align}\label{eq:bound-sum-bn4}
\begin{split}
\sum_{n=1}^{m-1}|b_n^4|
=
&
\sum_{n=1}^{m-1}\left|\mathbb{E}\int_{t_n}^{t_{n+1}}\left<A^{-\frac\beta2+\varepsilon}(I-P_N)A^{\frac\beta2}X_{\tau,h}(t)
,
A^{1-\varepsilon}D\mu^N(t_m-t,X_{\tau,h}(t))\right>\,\dd t\right|
\\
\leq
&
C\lambda_N^{-\frac\beta2+\varepsilon}\sum_{n=1}^{m-1}\int_{t_n}^{t_{n+1}}\left(1+(t_m-t)^{-1+\varepsilon}\right)e^{-c(t_m-t)}\,\dd t
\\
\leq
&
C\lambda_N^{-\frac\beta2+\varepsilon}.
\end{split}
\end{align}
By \eqref{eq:bound-bn1}, \eqref{eq:bound-bn2}, \eqref{eq:bound-bn3-sum} and \eqref{eq:bound-sum-bn4}, we conclude
\begin{align}
\sum_{n=1}^{m-1}|b_n|
\leq
C(h^{2\beta-2\varepsilon}+\tau^{\beta-\varepsilon}+C\lambda_N^{-\frac\beta2+\varepsilon}).
\end{align}
Therefore, combining all the estimates of $a_n, b_n, c_n$ together gives
\begin{align}\label{eq:XNTMXO-XNXTAUHX0}
\mathbb{E}\big[\mu^N(t_m,P_hX_0)\big]
-
\mathbb{E}\big[\mu^N(0,X_{\tau,h}^m)\big]
\leq
C(h^{2\beta-2\varepsilon}+\tau^{\beta-\varepsilon}+C\lambda_N^{-\frac\beta2+\varepsilon}).
\end{align}
Finally, by \eqref{eq:PnXm-XM}, \eqref{eq:xtm-xntm}  and \eqref{eq:XNTMXO-XNXTAUHX0} and taking $N\rightarrow \infty$, we obtain
 \begin{align}
\left |\mathbb{E}\big[\varphi(X(t_m))\big]-\mathbb{E}\big[\varphi(X^m_{\tau,h})\big]\right|
\leq
C(h^{\min\{2\beta-2\varepsilon,\rho\}}+\tau^{\min\{\beta-\varepsilon,\theta\}}),
 \end{align}
as required.
 $\square$

Let $\mu$ be the  unique invariant measure of the SPDE \eqref{eq:parabolic-SPDE}. If the fully discrete finite element scheme \eqref{eq:full-discretization} possesses a unique invariant measure $\mu_{\tau,h}$, then we have the following  convergence rate between $\mu$ and $\mu_{\tau,h}$.
\begin{proposition}
Let Assumptions \ref{assum:linear-operator-A}-\ref{assum:intial-value-data} and Assumption \ref{assum:derivate-ftauh-bound} are valid for $\gamma\in(\frac d2,2]$ or $\gamma\in(0,\frac12)$ with $Q=I$
 in dimension one. Then, there exists a constant $c$ such that for any $\varphi\in C_b^2$ and $m$ it holds
\begin{align}
\left|\mathbb{E}\left[\varphi(X_{\tau,h}^m)-\int_H\varphi\,\dd \mu\right]\right|
\leq
C(X_0,Q,D,\varphi)(\tau^\delta+h^\iota+e^{-cm\tau})
\end{align}
with $\delta\in(0,\min\{\theta,\beta\})$, $\iota\in(0,\min\{\rho,2\beta\})$, where $\beta=\min\{\gamma,1\}$.
Furthermore, if $\mu_{\tau,h}$ is an ergodic invariant measure of the numerical solution $\{X_{\tau,h}^m\}_{m\geq 0}$, we have
\begin{align}
\left|\mathbb{E}\left[\int_{P_h(H)} \varphi\,\dd \mu_{\tau,h}-\int_H\varphi\,\dd \mu\right]\right|
\leq
C(X_0,Q,\phi)(\tau^\delta+h^\iota).
\end{align}
\end{proposition}

\section{Numerical experiments}
In this section, we present some numerical experiments to confirm our previous theoretical finding. For simplicity, we consider the stochastic partial differential  equation in one dimension as follows
\begin{align}\label{eq:sswe-one-dimension}
\begin{split}
\left\{\begin{array}{ll}
\dd u(t)=  u_{xx}(t) \dd t
+
(u-u^3) \dd t + \dd W(t),& (t,x)\in [0, T]\times [0,1],
\\
u(0)=0,\;v(0)=0,&x\in [0,1],
\end{array}\right.
\end{split}
\end{align}
where $\{W(t)\}_{t\in[0,1]}$
stands for a standard $Q$-Wiener process with the covariance operator $Q=\Lambda^{-s}$, $s=\{0, 0.5005\}$. One can easily check that Assumption \ref{assum:noise-term} is valid and the condition \eqref{eq:ass-AQ-condition}
is fulfilled with $\gamma=0.5$ for $Q= I$ and $\gamma=1$ for $Q=A^{-0.5005}$.
In the following experiments, we consider the fully discrete scheme \eqref{eq:full-discretization} with $f_{\tau,h}(\cdot)$ given by \eqref{eq:definiton-I-ftauh}, where we take $\beta_1=\beta_2=1$, $\theta=1$, $\rho=2$, $\alpha=0.25$. Throughout the numerical tests, the expectation is approximated by the Monte-Carlo approximation using average over $2000$ samples.


We first test the long-time behaviors of the scheme \eqref{eq:full-discretization}.
To this end, we take  $h=2^{-5}$ and $\tau=2^{-6}$ and the expectation $\mathbb{E}[\sin(\pi/4-\|X_{\tau,h}^m\|^2)]$ is approximated by taking average over $5000$ samples.  Figure \ref{fig:long-time-behav} shows that  the average $\mathbb{E}[\sin(\pi/4-\|X_{\tau,h}^m\|^2)]$  started  from different values with the terminal time $T=8$ and it converges to the same equilibrium in a short time.
\begin{figure}[!ht]\label{fig:long-time-behaveior-results}
\centering
 \scalebox{0.5}[0.5]{\includegraphics[0.3in,
0.1in][5.3in,4.1in]{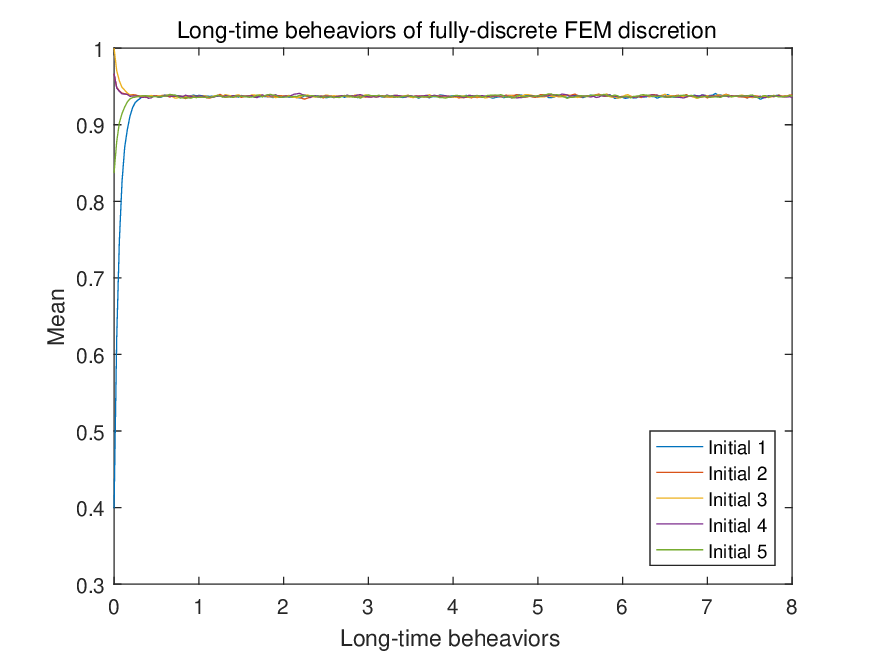}}
 \caption{Long-time behaviors for the scheme in the trace noise case ($\mathrm{Tr}(Q )<\infty$)}
 \label{fig:long-time-behav}
\end{figure}

Next, we test the long-time strong convergence rate. 
We conduct relevant experiments on a large time interval $[0, T]$ with $T=8$.
 To test the convergence rate
 in time, numerical simulations are performed with four different time  step-sizes $\tau=T/2^{m}, m\in\{6,7,8,9\}$ and a fixed space step-size $h=2^{-7}$. The resulting strong errors against mesh sizes in logarithmic scale are depicted in   Figure \ref{fig:strong-converbence} and  one can observe that the strong error converges with order 0.5 in the space-time white noise  case and with order 1 in the trace class noise case, which is  consistent with the previous theoretical findings. The "true" solution is identified with numerical ones using small step-sizes $h_{exact}=2^{-7}$ and $\tau_{exact}=T/2^{12}$, $T=8$.
\begin{figure}[!ht]\label{fig:long-time-strong-temporal-results}
\centering
 \scalebox{0.5}[0.5]{\includegraphics[0.3in,
0.1in][5.3in,4.1in]{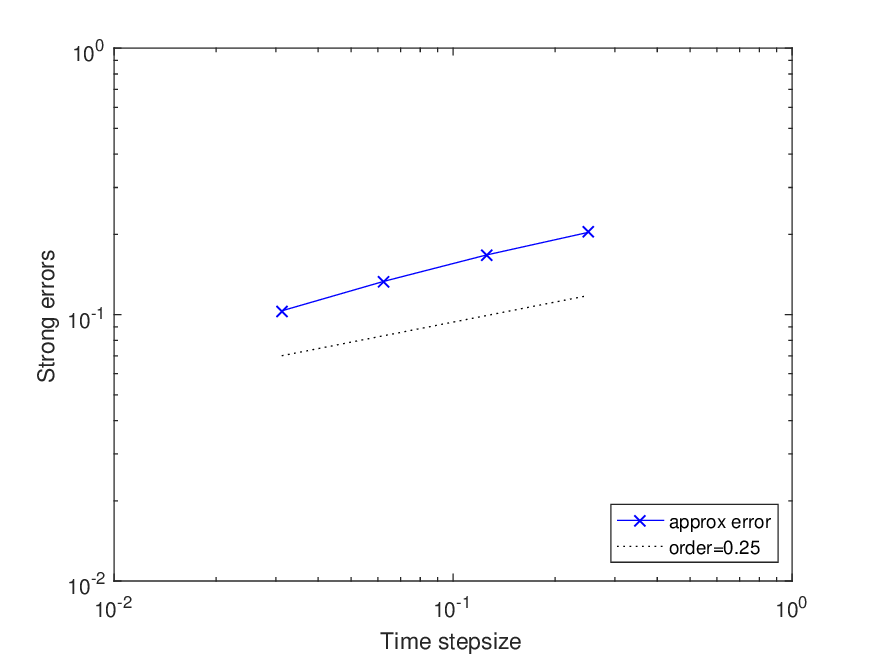}}\qquad\qquad
 \scalebox{0.5}[0.5]{\includegraphics[0.3in,
0.1in][5.4in,4.1in]{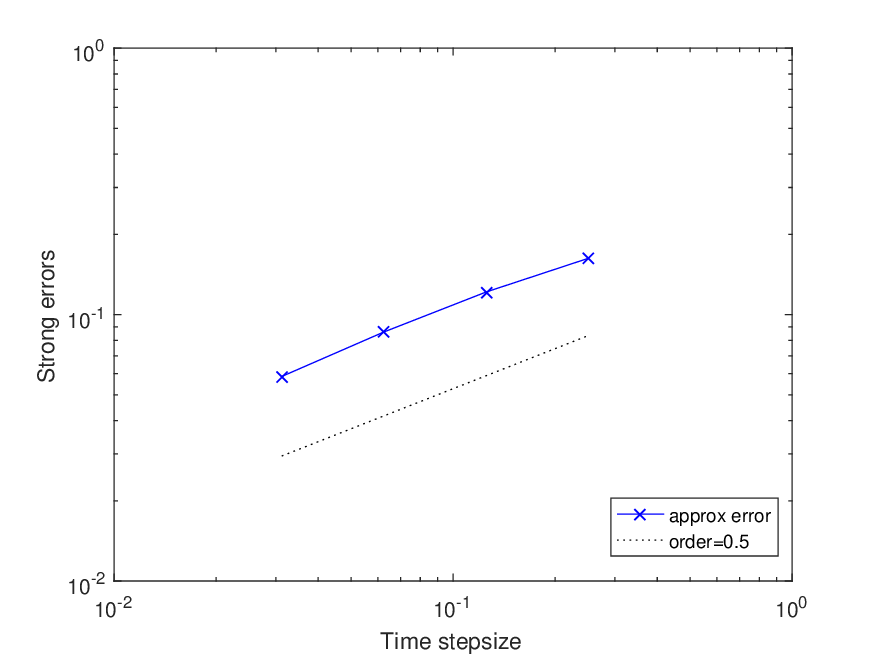}}
 \caption{Long-time strong convergence rates in time (Left: $Q=I$; Right: $Q=A^{-0.5005} $)}
\label{fig:strong-converbence}
\end{figure}

As simulating the strong errors,  we  test the weak convergence rates in time at the endpoint $T=8$. Here,  the "true" solutions are computed by numerical solutions using small step-size $h_{exact} = 2^{-6}$
and $\tau_{exact}=T/2^{12}$. In Figure \ref{fig:weak-converbence},  we present the resulting errors  of the proposed method \eqref{eq:full-discretization} in time direction under four time mesh sizes $\tau=T/2^{m},m\in\{5,6,7,8\}$
 with a small space step-size $h_{exact} = 2^{-6}$.
 As expected, the numerical performance is all consistent with the previous theoretical results.
\begin{figure}[!ht]\label{fig:long-time-weak-temporal-results}
\centering
 \scalebox{0.5}[0.5]{\includegraphics[0.3in,
0.1in][5.3in,4.1in]{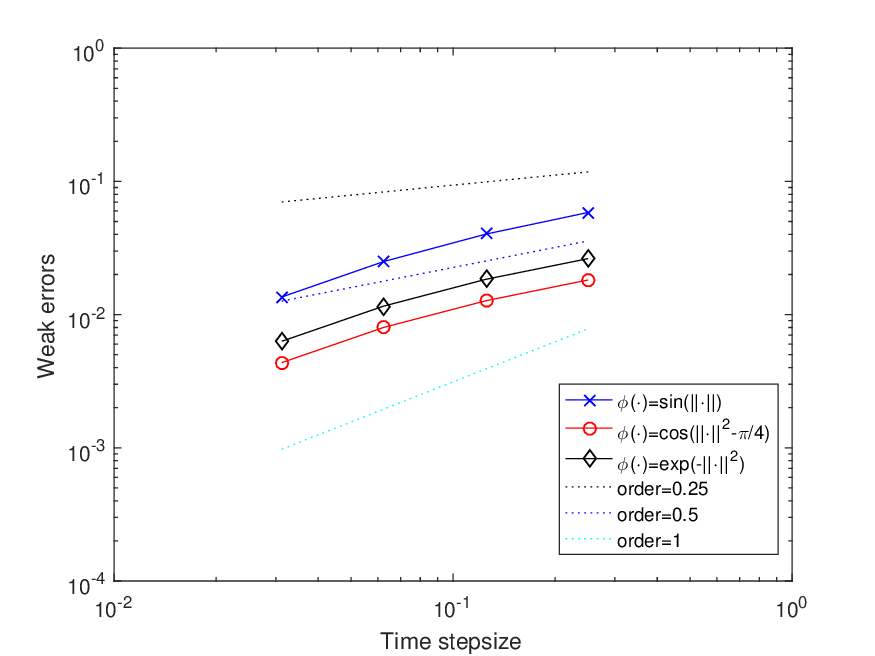}}\qquad\qquad
 \scalebox{0.5}[0.5]{\includegraphics[0.3in,
0.1in][5.4in,4.1in]{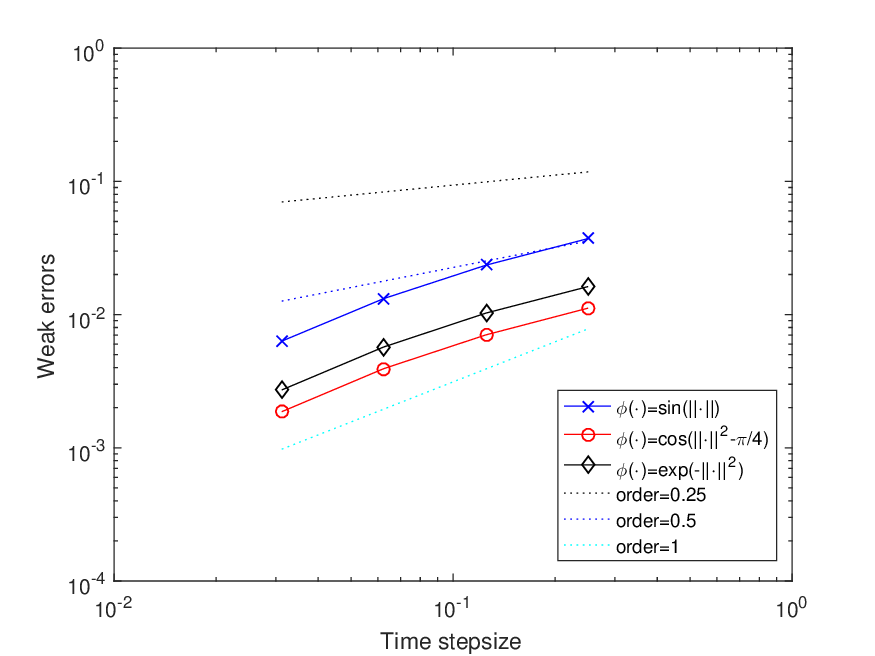}}
 \caption{Long-time weak convergence rates in time  (Left: $Q=I$; Right: $Q=A^{-0.5005} $)}
 \label{fig:weak-converbence}
\end{figure}

\section{Conclusion}
In this work, we propose and analyze new fully discrete schemes for long-time approximations of SPDEs with additive noises and non-globally Lipschitz coefficients in a bounded domain $D \subset \R^d, d =1,2,3 $.
Based on a standard Galerkin finite element spatial semi-discretization, a novel family of linearly implicit time-stepping schemes is introduced, preserving uniform-in-time moment bounds in Banach spaces, without requiring any restriction on the time-space discretization stepsize ratio. Both strong and weak errors are analyzed for the proposed fully discrete schemes, with uniform-in-time strong and weak convergence rates obtained. As an ongoing project, we are focusing on long-time finite element approximations of SPDEs with multiplicative trace-class noises in multiple dimensions.

\section*{Appendix}
{\it Proof of Lemma \ref{lem:eq-smooth-ENm}.}
  For the proof of \eqref{lem:eq-temporal-regu-ENm}, please refer to  \cite[Lemma 5.2]{qi2020error}. The assertion \eqref{lem:eq-spatial-regu-ENm-sum} in the case   $\rho=1$ can be found in  \cite[Lemma 5.2]{qi2020error} and the case $\rho=0$ is obvious by the stability of $E_{\tau,h}^m$. Thus, the
interpolation technique implies the intermediate cases.  For the proof of \eqref{lem:eq-spatial-regu-ENm} in the case $t_m\in(0,1]$, one can consult \cite[Lemma 5.2]{qi2020error}. It remains to show
 \eqref{lem:eq-spatial-regu-ENm}  for $t_m\geq 1$.  By using the expansion of $P_hv$ in terms of $\{e_{j,h}\}_{j=1}^{\mathcal{N}_h}$, one knows
 \begin{align}\label{eq:expansion-full-operator}
 \|A_h^{\frac\mu2}E_{\tau,h}^mP_hv\|^2
 =
 \sum_{j=1}^{\mathcal{N}_h}\lambda_{j,h}^\mu r(\lambda_{j,h}\tau )^{2m}
 \big<v,e_{j,h}\big>^2.
 \end{align}
Thus, it is sufficient  to show that  for any $\mu\in[0,2]$
\begin{align}\label{eq:bound-r}
\lambda_{j,h}^{\frac\mu2} r(\lambda_{j,h}\tau )^m
\leq
C t_m^{-2},
\end{align}
where $C$ is a positive constant, independent of $m$, $h$ and $\tau$.
To prove it, we consider two cases: $\tau \lambda_{j,h}\leq 1$ and $\tau \lambda_{j,h}>1$. For the case $\tau \lambda_{j,h}\leq 1$, we get by \eqref{eq:r(z)-ez}
\begin{align}\label{eq:bound-full-operator-bound-klambda-notbig1}
\lambda_{j,h}^{\frac\mu2} r(\lambda_{j,h}\tau )^m
\leq
C\lambda_{j,h}^{\frac\mu2} e^{-\frac{ct_m\lambda_{1}}2}
\leq
C\tau^{-\frac\mu2} t_m^{-2}
\leq
Ct_m^{-2}.
\end{align}
For the case $\tau \lambda_{j,h}>1$, it is easy to show that, for
$m=1$
\begin{align}\label{eq:r(A)-taulambda-equality=1}
\lambda_{j,h}^{\frac\mu2}r(\lambda_{j,h}\tau)
\leq
C \tau^{-1}\lambda_{j,h}^{-1+\frac\mu2}
\leq
C\tau^{-\frac\mu2}
=
C\tau^{2-\frac\mu2}t_1^{-2},
\end{align}
and
for any $m\geq 2$
\begin{align}\label{eq:r(A)-taulambda-big-1}
\begin{split}
\lambda_{j,h}^{\frac\mu2} r(\lambda_{j,h}\tau )^m
\leq
&
C\tau^{-\frac\mu2} (1+\tau \lambda_{j,h} )^{-m+1}
\leq
 C\tau^{-\frac\mu2}\sup_{\lambda\geq 1}(1+\lambda )^{-m+1}
 \\
\leq
&
 C\tau^{-\frac\mu2} (m-1)^{-2}
\leq
 C\tau^{-\frac\mu2} m^{-2}
\leq
C \tau^{2-\frac\mu2}t_m^{-2}.
\end{split}
\end{align}
Hence the desired assertion \eqref{eq:bound-r} follows and the proof of this lemma is complete. $\square$

{\it Proof of Lemma \ref{lem:bound-discrete-stochastic-convulution}.}
The assertion \eqref{lem:bound-discrete-stochastic-convolution} can be directly proven by using \eqref{eq:relation-Lp-Ah-relatoin}, \eqref{lem:eq-spatial-regu-ENm-sum-II} with $\varrho=\gamma$
and the Burkholder inequality. In view of \eqref{lem:bound-discrete-stochastic-convolution} and \eqref{eq:V-norm-control-by-Ah-norm},  \eqref{lemeq:V-bound-discrete-stoch-conv} in the case $\gamma \in (\frac d2,2]$ follows. To show \eqref{lemeq:V-bound-discrete-stoch-conv} for the case $\gamma\in (0,\frac12)$ with $Q=I$ in dimension one, we assume the domain $D=[0, L]$ and the quasi-uniform triangulation  of $D$ is a uniform mesh. Let $h=\frac{L}{\mathcal{N}_h+1}$. The eigen-system $\{(\lambda_{j,h},e_{j,h})\}_{j=1}^{\mathcal{N}_h}\}$ of $A_h$ in $V_h$ satisfies the following property (see \cite{waslsh2005galerkin}):
\begin{align}\label{eq:bound-discrete-eigen-system-I}
\tfrac{4j^2}{L^2}
\leq
\lambda_{j,h}
\leq
\tfrac{3\pi^2j^2}{L^2}, j=1,2,\cdots, \mathcal{N}_h,
\end{align}
and
\begin{align}
\|e_{j,h}\|_{L^\infty(D)}
\leq
&
\sqrt{\tfrac{6}{L}},
\label{eq:bound-discrete-value-I}
\\
|e_{j,h}(x)-e_{j,h}(y)|
\leq
&
\sqrt{\tfrac{6}{L}}\tfrac{\pi j}{L}|x-y|,
\label{eq:bound-discrete-value-II}
\end{align}
for  $\forall x,y\in [ih,(i+1)h], i=0,1,2,\cdots, \mathcal{N}_h+1$.
Therefore, $W_{\tau,h}^m$ can be written as
\begin{align}
W_{\tau,h}^m
=
\sum_{k\in \mathbb{N}^+}\sum_{j=1}^{m}\sum_{l=1}^{\mathcal{N}_h}
(1+\tau \lambda_{l,h})^{-(m+1-j)}\left<e_k, e_{l,h}\right>e_{l,h}(\beta_k(t_j)-\beta_k(t_{j-1})).
\end{align}
This together with \eqref{eq:bound-discrete-eigen-system-I}-\eqref{eq:bound-discrete-value-II} and using the It\^{o} isometry imply
\begin{align}
\begin{split}
&\mathbb{E}[|W_{\tau,h}^m(x)
-
W_{\tau,h}^m(y)|^2]
\\
\leq
&
\mathbb{E}\left[\sum_{ k\in \mathbb{N}^+}\left(\sum_{j=1}^{m}\sum_{l=1}^{\mathcal{N}_h}
(1+\tau \lambda_{l,h})^{-(m+1-j)}\left<e_k,e_{l,h}\right>\left(e_{l,h}(x)-e_{l,h}(y)\right)(\beta_k(t_j)-\beta_k(t_{j-1}))
\right)^2\right]
\\
\leq
&
\sum_{ k\in \mathbb{N}^+}\sum_{j=1}^{m}
\tau\left[\left(\sum_{l=1}^{\mathcal{N}_h}(1+\tau \lambda_{l,h})^{-(m+1-j)}
\left<e_k,e_{l,h}\right>\left(e_{l,h}(x)-e_{l,h}(y)\right)\right)^2\right]
\\
\leq
&
\sum_{j=1}^{m}
\tau \sum_{l=1}^{\mathcal{N}_h}\sum_{i=1}^{\mathcal{N}_h}(1+\tau \lambda_{l,h})^{-(m+1-j)}(1+\tau \lambda_{i,h})^{-(m+1-j)}
\left<e_{l,h},e_{i,h}\right>\left(e_{l,h}(x)-e_{l,h}(y)\right)\left(e_{i,h}(x)-e_{i,h}(y)\right)
\\
\leq
&
\sum_{l=1}^{\mathcal{N}_h}\sum_{j=1}^{m}\tau(1+\tau \lambda_{l,h})^{-2(m+1-j)}
\left(e_{l,h}(x)-e_{l,h}(y)\right)^2.
\end{split}
\end{align}
To proceed further, one should bound $\sum_{j=1}^{m}\tau(1+\tau \lambda_{l,h})^{-2(m+1-j)}$. To this end, we consider two cases: $ \tau \lambda_{l,h}\leq 1$ and $ \tau \lambda_{l,h}> 1$.
For the case $\tau \lambda_{l,h} \leq 1$, we obtain by \eqref{eq:r(z)-ez} and \eqref{eq:bound-discrete-eigen-system-I}:
\begin{align}
\sum_{j=1}^{m}\tau(1+\tau \lambda_{l,h})^{-2(m+1-j)}
\leq
C\sum_{j=1}^{m}\tau  e^{-2c\lambda_{l,h}t_{m+1-j}}
\leq
C\lambda_{l,h}^{-1}
\leq
Cl^{-2}.
\end{align}
For the case $\tau\lambda_{l,h}>1$, we have
\begin{align}
\sum_{j=1}^{m}\tau(1+\tau \lambda_{l,h})^{-2(m+1-j)}
\leq
\tau(1+\tau \lambda_{l,h})^{-2} \sum_{j=0}^{m-1}(1+\tau \lambda_{l,h})^{-2j}
\leq
C \lambda_{l,h}^{-1} \sum_{j=0}^{m-1}2^{-2j}
\leq
C \lambda_{l,h}^{-1}
\leq Cl^{-2}.
\end{align}
The above two estimates imply
\begin{align}\label{eq:sum-discrete-value}
\sum_{j=1}^{m}\tau(1+\tau \lambda_{l,h})^{-2(m+1-j)}
\leq Cl^{-2}.
\end{align}
Additionally, for $i>j$, $x\in[ih,(i+1)h]$ and $y\in[jh,(j+1)h]$,
it follows from \eqref{eq:bound-discrete-value-I} that
\begin{align}\label{eq:spatial-regularity-discrete-value}
\begin{split}
|e_{l,h}(x)-e_{l,h}(y)|
\leq
&
|e_{l,h}(x)-e_{l,h}(ih)|
+
|e_{l,h}(ih)-e_{l,h}((i-1)h)|
+
\cdots
\\
&
+
|e_{l,h}(£¨j+2)h)-e_{l,h}((j+1)h)|
+
|e_{l,h}(£¨j+1)h)-e_{l,h}(y)|
\\
\leq
&
\sqrt{\frac 6 L}\frac{\pi l} L |x-y|.
\end{split}
\end{align}
As a consequence, using \eqref{eq:bound-discrete-value-I}, \eqref{eq:bound-discrete-value-II},
\eqref{eq:spatial-regularity-discrete-value} and
\eqref{eq:sum-discrete-value}, one can derive
\begin{align}
\begin{split}
\mathbb{E}[|W_{\tau,h}^m(x)
-
W_{\tau,h}^m(y)|^2]
\leq
&
C \sum_{l=1}^{\mathcal{N}_h} l^{-2}
|e_{l,h}(x)-e_{l,h}(y)|^{\frac 45}(|e_{l,h}(x)|+|e_{l,h}(y)|)^{\frac 65}
\\
\leq
&
C \sum_{l=1}^{\mathcal{N}_h}
l^{-2} l^{\frac 45} |x-y|^{\frac 45}
\leq
C |x-y|^{\frac 45}.
\end{split}
\end{align}
In the same way,
\begin{align}
\sup_{m\in  \mathbb{N}^+ }\sup_{x\in D}\mathbb{E}[|W_{\tau,h}^m(x)|^2]
<\infty.
\end{align}
Using the Sobolev embedding inequality $W^{\frac15,p} \subset V, p> 5$,  one can find, for any $m\in  \mathbb{N}^+$
\begin{align}
\begin{split}
\mathbb{E}[\|W_{\tau,h}^m\|_V^p]
\leq
&
C\int_0^L\mathbb{E} [|W_{\tau,h}^m(x)|]^p \,\dd x
+
C \int_0^L\int_0^L \frac{\mathbb{E} [|W_{\tau,h}^m(x)
-W_{\tau,h}^m(y)|]^p}{|x-y|^{\frac p5+1}}\,\dd x \dd y
\\
\leq
&
C\int_0^L(\mathbb{E} [|W_{\tau,h}^m(x)|^2])^{\frac p2} \,\dd x
+
C \int_0^L\int_0^L \frac{(\mathbb{E} [|W_{\tau,h}^m(x)
-W_{\tau,h}^m(y)|]^2)^{\frac p2}}{|x-y|^{\frac p5+1}}\,\dd x \dd y
\\
\leq
&
C\left(1+ \int_0^L\int_0^L |x-y|^{\frac p5-1}\,\dd x \dd y\right)
<\infty,
\end{split}
\end{align}
where in the second inequality we also used the fact
that the discrete stochastic convolution $W_{\tau,h}^m$
 is Gaussian.
The desired assertion in the case $p\in[2,5]$ can be immediately shown by the H\"{o}lder inequality. $\square$

\bibliography{bibfile}

\bibliographystyle{abbrv}
 \end{document}